\documentclass[11pt]{article}
\usepackage{amssymb}
\usepackage{graphicx}
\usepackage{xcolor} 
\usepackage{tensor}
\usepackage{fullpage} 
\usepackage{amsmath}
\usepackage{amsthm}
\usepackage{verbatim}
\usepackage{hyperref}
\usepackage{enumitem}
\usepackage{mathrsfs}
\setlist[enumerate]{leftmargin=1.5em}
\setlist[itemize]{leftmargin=1.5em}

\providecommand{\MR}{\relax\ifhmode\unskip\space\fi MR }

\providecommand{\href}[2]{#2}

\usepackage{seqsplit} 
\definecolor{green}{rgb}{0,0.8,0} 

\newtheorem{theorem}{Theorem}[section]

\newtheorem{lemma}[theorem]{Lemma}
\newtheorem{proposition}[theorem]{Proposition}

\theoremstyle{definition}

\theoremstyle{remark}
\newtheorem{remark}[theorem]{Remark}

\numberwithin{equation}{section}

\newcommand{\nnrm}[1]{{\vert\kern-0.25ex\vert\kern-0.25ex\vert #1 
    \vert\kern-0.25ex\vert\kern-0.25ex\vert}}

\newcommand{\ud}{\mathrm{d}}

\newcommand{\ddt}{\frac{\mathrm{d}}{\mathrm{d}t}}

\newcommand{\tht}{\theta}

\newcommand{\bfb}{{\bf b}}

\newcommand{\bfu}{{\bf u}}

\newcommand{\bbC}{\mathbb C}

\newcommand{\bbN}{\mathbb N}

\newcommand{\bbR}{\mathbb R}

\newcommand{\bbT}{\mathbb T}

\newcommand{\bbX}{\mathbb X}

\newcommand{\bbZ}{\mathbb Z}

\begin{document}

\title{Asymptotic stability and sharp decay rates to the linearly stratified Boussinesq equations in horizontally periodic strip domain}
\author{Juhi Jang\thanks{Department of Mathematics, University of Southern California, Los Angeles, CA 90089, USA and Korea Institute for Advanced Study, Seoul, Korea} \and Junha Kim\thanks{School of Mathematics, Korea Institute for Advanced Study, 85 Hoegi-ro, Dongdaemun-gu, Seoul 02455, Korea.}} 
\date{\today}



\maketitle


\begin{abstract}
We consider an initial boundary value problem of the multi-dimensional Boussinesq equations in the absence of thermal diffusion with velocity damping or velocity diffusion under the stress free boundary condition in horizontally periodic strip domain.   
We prove the global-in-time existence of classical solutions in high order Sobolev spaces 
satisfying high order compatibility conditions around the linearly stratified equilibrium, the convergence of the temperature to the asymptotic profile, and sharp decay rates of the velocity field and temperature fluctuation in all intermediate norms based on spectral analysis combined with energy estimates. To the best of our knowledge, our results provide first sharp decay rates for the temperature fluctuation and the vertical velocity to the linearly stratified Boussinesq equations in all intermediate norms. 

\end{abstract}


\section{Introduction}

We consider the Boussinesq equations for buoyant fluids

\begin{equation}\label{Boussines}
\left\{
\begin{array}{ll}
v_t +\nu (-\Delta)^\alpha v +(v \cdot \nabla)v = -\nabla p +\rho e_d,\\
\rho_t  +\kappa (-\Delta)^\beta \rho
+(v \cdot \nabla)\rho=0,\\
\operatorname{div}\,\, v =0,\\
v(x, 0)=v_0(x),\,\, \rho(x, 0)=\rho_0(x),
\end{array}
\right.
\end{equation}
where $v$, $p$, and $\rho$ denote the fluid velocity field, scalar pressure and density (or  temperature) respectively. The parameter $\alpha\ge 0$ and $\beta\ge 0$  
represent the strength of dissipation, and  thermal diffusion, 
while 
the parameters $\nu\ge0 $ and $\kappa\ge 0$ stand for the nonnegative constant fluid viscosity and thermal diffusivity, respectively. 
 The $d$-dimensional vector $e_d$ stands for $(0, \cdots, 0, 1)^T$. 
 
The Boussinesq equations \eqref{Boussines} arise in geophysical fluid dynamics to model and study atmospheric and oceanographic flows \cite{Majda, Pedlosky} and describe interesting physical phenomena such as Rayleigh-B\'enard convection \cite{ES, Getling} and turbulence \cite{CD96}. From a mathematical point of view, the Boussinesq equations are intimately tied to the Euler and Navier-Stokes equations and they share important features such as the vortex stretching. In fact, the two-dimensional inviscid Boussinesq equations can be viewed as the three-dimensional axisymmetric Euler equations for swirling flows \cite{MB01}. 
Due to its physical and mathematical relevance, there have been a lot of works and progress made on the Boussinesq system in the past decades: for instance, see \cite{AH07, BN, Cao-Wu, Chae, CN97, CH, Danchin, DLS, EW, HL, HKZ, JKL, Ju, KPY, KW, LLT, LPZ, LT17, LWZ1, Tak1, Tak2, Wan, Zhai, Wid} and references therein on the local, global well-posedness and regularity problem. 

On the other hand, it is well-known that the system \eqref{Boussines} has the exact solutions, called hydrostatic equilibrium, with the balance equation 
	\[
	v = 0, \quad \frac{\partial}{\partial x_{d}}p(x_d)= \rho(x_d). 
	\]
In recent years, the 
stability around the linearly stratified state $(v_s, \rho_s, p_s):=(0, \cdots, 0, x_d, x_d^2/2)$ has been a subject of active research in the presence of dissipation where damping is 
understood as a limit of fractional diffusion. For $d=2$, there exist many stability results (see \cite{BN, BCP} and references therein),  
while less works are available for other space dimension. Among others, asymptotic stability with velocity damping was studied in $\mathbb R^3$  \cite{Dong}, and the stability result has been extended to $\mathbb R^d$ with more general initial data in \cite{KL}.




In this paper, we focus on the domain with boundary, in particular $\Omega=\bbT^{d-1} \times [-1,1]$. This type of domain with  $\rho = 1$ and $\rho = -1$ fixed on the bottom boundary and top boundary has been used to demonstrate the Rayleigh-B\'enard convection \cite{ES, Getling}, which leads to the instability of the solution by a continuously heated bottom fluid. On the contrary, the opposite case where  $\rho = -1$ and $\rho = 1$ on the bottom and top boundary respectively stabilizes the system. We will show stabilizing aspects of the latter by analyzing the dynamics near linearly stratified hydrostatic equilibrium $(v_s, \rho_s, p_s)=(0, \cdots, 0, x_d, x_d^2/2)$. 
We consider two cases: $\alpha = 0$ (velocity damping) and $\alpha=1$ (velocity diffusion) without thermal diffusion $(\kappa = 0)$. When $\alpha=0$, we take the no-penetration boundary condition $v \cdot n = 0$ and when $\alpha = 1$, we impose the stress free boundary condition, also known as the Lions boundary condition $v \cdot n = 0$ and $\operatorname{curl}v \times n = 0$, where the temperature is fixed at $\rho_s = -1$ and $\rho_s = 1$ on the each boundary. Here, $n$ denotes the outward unit normal vector to $\partial \Omega$. Let us set
\[
\rho(x,t)= x_d+\theta(x,t),\qquad p(x,t)=x_d^2/2 +P(x,t).
\]
Then, the perturbed system is given by
\begin{equation}\label{EQ}
\left\{\begin{aligned}
&v_t + (-\Delta)^{\alpha}v + (v \cdot \nabla)v = -\nabla P + \theta e_d, \qquad \mathrm{div}\, v = 0, \\
&\theta_t + (v \cdot \nabla)\theta = -v_d, \\
&v(x, 0)=v_0(x),\,\, \theta(x, 0)=\theta_0(x),
\end{aligned}
\right.
\end{equation}
where the boundary conditions of the velocity field are preserved and $\theta$ vanishes on $\partial \Omega$ in each case $\alpha = 0$ and $\alpha = 1$ with $\theta_0|_{\partial\Omega} =0$. 

We now discuss some relevant prior works regarding \eqref{EQ} starting with the case $\alpha = 0$. 
Castro, C\'ordoba, and Lear \cite{CCL} showed 
 the asymptotic stability of \eqref{EQ} for $d=2$. In particular, the authors showed that high order compatibility conditions are satisfied for well-prepared data, and introduced 
proper solution spaces $X^m(\Omega)$, $Y^m(\Omega) \subset H^m(\Omega)$ with orthonormal bases (see Section \ref{sec:2.2} for the definitions). For their main result, 
for $m \in \bbN$ with $m \geq 17$, 
the small data global existence with temporal decay estimate $(1+t)^{\frac {m-7}8} (\| v(t) \|_{H^4} + \| \bar{\tht} (t) \|_{H^4}) \leq C$ was obtained, where $\bar{\tht}(t) := \tht(t) - \int_{\bbT} \tht(t,x) \,\ud x_1$. 
It is worth pointing out that the temporal decay rates of $H^4$-norm increase as $m$ gets larger, namely the solutions are  
more regular. Next we consider the case $\alpha = 1$. In $d=2$, long time behavior was first considered by Doering et al \cite{DWZZ} for $v\in H^2$ and $\theta\in H^1$ and explicit decay rates were given in $\mathbb T^2$ by Tao et al \cite{TWZZ} using the spectral analysis. Recently, Dong and Sun considered the asymptotic stability problem on the infinite flat strip $\bbR^{d-1} \times (0,1)$ for $d = 2$ and $3$ in \cite{DS1} and \cite{DS2} respectively, and Dong \cite{Dong2} obtained the stability result on $\bbT \times (0,1)$.

In the aforementioned works, some explicit decay rates were obtained with high regularity index $m$ or global existence (2D) with more general initial data was obtained without explicit decay rates. However, the convergence of the temperature fluctuation and its optimal equilibration rate has remained elusive. The goal of this paper is to establish the global existence in $H^m$, $m>2+\alpha+ \frac{d}{2}$ satisfying high order compatibility conditions, the convergence of $\theta$ to the asymptotic profile $\sigma$, and sharp decay rates of $(v, \theta-\sigma)$ in $H^s$ norms for all $s\in [0,m]$. We now state the main results: 


\begin{theorem}\label{thm1}
	Let $d \in \mathbb{N}$ with $d \geq 2$ and let $m \in \mathbb{N}$ satisfying $m > 3+\frac d2$. Then there exists a constant $\delta > 0$ such that if initial data $(v_0,\,\theta_0) \in \bbX^m \times X^m(\Omega)$ with  $\operatorname{div} v_0 = 0$, $\int_{\Omega} v_0 \,\ud x = 0$, and $\| (v_0, \theta_0) \|_{H^m}^2 < \delta^2$, then \eqref{EQ} with $\alpha = 1$ possesses a unique global classical solution $(v,\,\theta)$ satisfying 
	\begin{equation*}
		v \in C([0,\infty); \bbX^m(\Omega)) \cap L^2([0,\infty); \bbX^{m+1}(\Omega)), \qquad \tht \in C([0,\infty); X^m(\Omega))
	\end{equation*}
	with
	\begin{equation}\label{sol_bdd_1}
		\sup_{t \in [0,\infty)} \| (v,\tht)(t) \|_{H^m}^2 + \int_0^{\infty} \| \nabla v(t) \|_{H^m}^2 \,\mathrm{d}t+ \int_0^{\infty} \| \nabla_h \theta(t) \|_{H^{m-2}}^2 \,\mathrm{d}t \leq 4 \|  (v_0,\tht_0) \|_{H^m}^2.
	\end{equation}
	Moreover, there exists a function 
	\begin{equation}\label{df_sgm}
	\sigma(x_d) := \int_{\bbT^{d-1}} \tht_0\,\ud x_h - \int_{\bbT^{d-1}} \int_0^\infty \left( (v \cdot \nabla)\tht + v_d \right) \,\ud t \ud x_h
	\end{equation} such that
	\begin{equation}\label{tem_1}
		(1+t)^{\frac {m-s}4} \| \tht (t) - \sigma(x_d) \|_{H^s} + (1+t)^{\frac 12 + \frac {m-s}4} || v(t) \|_{H^s} + (1+t)^{1+\frac {m-s}4} \| v_d (t) \|_{H^s} \leq C
	\end{equation}
	for any $s \in [0,m]$.
\end{theorem}
\begin{remark}
	The assumption $\int_{\Omega} v_{0}\,\ud x = 0$ is essential for the velocity field $v$ decaying in $t$ (see Lemma~\ref{lem_obs}).
\end{remark}
\begin{remark}
	Indeed for any $\epsilon >0$, there exists a constant $C>0$ such that
	\begin{equation*}
		t^{\frac 34 + \frac {m-s}4} \| v (t) \|_{H^{s-\epsilon}} \leq C
	\end{equation*}
	for any $s \in [0,m+1]$. See Proposition~\ref{prop_rmk}.
\end{remark}

\begin{theorem}\label{thm2}
	Let $d \in \mathbb{N}$ with $d \geq 2$ and let $m \in \mathbb{N}$ satisfying $m > 2+\frac d2$. Then there exists a constant $\delta > 0$ such that if initial data $(v_0,\,\theta_0) \in \bbX^m \times X^m(\Omega)$ with  $\operatorname{div} v_0 = 0$ and $\| (v_0, \theta_0) \|_{H^m}^2 < \delta^2$, then \eqref{EQ} with $\alpha = 0$ possesses a unique global classical solution $(v,\,\theta)$ satisfying 
	\begin{equation*}
		v \in C([0,\infty); \bbX^m(\Omega)) \cap L^2([0,\infty); \bbX^{m}(\Omega)), \qquad \tht \in C([0,\infty); X^m(\Omega))
	\end{equation*}
	with
	\begin{equation}\label{sol_bdd_0}
		\sup_{t \in [0,\infty)} \| (v,\tht)(t) \|_{H^m}^2 + \int_0^{\infty} \| v(t) \|_{H^m}^2 \,\mathrm{d}t+ \int_0^{\infty} \| \nabla_h \theta(t) \|_{H^{m-1}}^2 \,\mathrm{d}t \leq 4 \|  (v_0,\tht_0) \|_{H^m}^2.
	\end{equation}
	Moreover, there exists a function $\sigma(x_d)$ defined by \eqref{df_sgm} such that
	\begin{equation}\label{tem_0}
		(1+t)^{\frac {m-s}2} \| \tht (t) - \sigma(x_d) \|_{H^s} + (1+t)^{\frac 12 + \frac {m-s}2} \| v (t) \|_{H^s}  + (1+t)^{1+\frac {m-s}2} \| v_d (t) \|_{H^s} \leq C
	\end{equation}
	for any $s \in [0,m]$.
\end{theorem}

\begin{remark}
The decay rates for $\tht$ and $v_d$ in Theorem~\ref{thm1} and \ref{thm2} are sharp (see section 7).
\end{remark}


To the best of our knowledge, our results provide the first sharp decay rates for the temperature fluctuation and the vertical velocity in all intermediate norms. In particular, they show  
the enhanced $L^2$ decay rate for higher order initial data, while $H^m$ decay rate doesn't change for both velocity damping and velocity diffusion. This is in contrast to parabolic equations for which higher norms enjoy faster decay rates. The regularity index $m$ required in our analysis is higher than the one required for the local existence, but it is still significantly smaller than the ones required in the previous results.  Also our results demonstrate that the velocity damping leads to faster decay than the velocity diffusion in the presence of the slip boundary, despite having the Poincar\'{e} inequality for the velocity field in hand. This is because of coupling structure between the velocity field and the temperature fluctuation of Boussinesq equations: it causes the temperature to decay much slower than the velocity field and the velocity diffusion weakens the temperature damping in high frequency. 
Moreover, the method developed in this paper is robust and applicable to the periodic box $\mathbb T^d$, and to various partially dissipative PDEs including  
non-resistive MHD and IPM (cf. \cite{JK}). 

The main difficulty comes from the non-decaying $\theta$ and weak damping in $\nabla_h\theta$, which makes the standard energy estimates alone hard to bootstrap the local theory to global theory and to capture precise decay rates. To establish the results, 
we employ the spectral analysis using the orthonormal basis associated to our domain with the slip boundary together with energy estimates, first to obtain the global existence and then to prove the decay rates by relying on the already established uniform bounds of the solutions. The relaxed condition for $m$ comes from estimating the key quantities $\int \| \nabla v(t) \|_{L^{\infty}} \,\ud t $ and $ \int \| \partial_d v_d(t) \|_{L^{\infty}} \,\ud t$ 
which appear in the energy estimates. The previous works on the stability problem of \eqref{EQ} ($d=2$) were devoted to obtaining the temporal decay estimate for $\| u(t) \|_{H^4}$ or $\| \partial_1 \operatorname{curl} v(t) \|_{H^2}$, which obviously require stronger condition for $m$ (see \cite{CCL} and \cite{Dong2}). Getting decay rates in bounded domains turns out to be more subtle than in the whole space, since $\tht$ does not decay, while it decays in the whole space. 
To prove the sharp decay rates of $(v, \theta-\sigma)$ in $H^s$ norms for all $s\in [0,m]$ in our domain, we adapt Elgindi's the splitting scheme of the density 
first used for the linearly stratified IPM equation in $\bbT^2$ \cite{Elgindi}. 
In particular, splitting the density into decaying part and non-decay part and using the boundedness of high norms obtained from the global existence part, the decay of low norms can be obtained through optimizing splitting scale of frequency in spirit of \cite{Elgindi}. 
We refer to Lemma~\ref{lem_lin} for a clear view of the sharp decay estimates for the linearized system of \eqref{EQ}, and Section~\ref{sec6} for controlling the nonlinear terms in \eqref{EQ} with the splitting scheme. 



The rest of this paper proceeds as follows. In Section \ref{sec2}, we give some preliminary results used for the paper and introduce key function spaces $X^m(\Omega), Y^{m}(\Omega), \bbX^m(\Omega)$ and their orthonormal bases. Section \ref{sec_spec} is devoted to spectral analysis of \eqref{EQ} in frequency variables and the proof of linear decay estimates.  
In Section \ref{sec:energy}, we present the energy-dissipation inequalities for \eqref{EQ}. In Section \ref{sec:global}, we extend the local existence to global-in-time result by combining the energy estimates with spectral analysis to estimate key quantities \eqref{key_quan} appearing in the energy estimates. Section \ref{sec6} is devoted to the proof of temporal decay estimates based on the spectral analysis and the splitting scheme. In Section \ref{sec:sharp}, we argue  that the decay rates are sharp by showing that the linear decay rates can't be algebraically improved. 

\section{Preliminaries}\label{sec2}
We first introduce some notations that will be used throughout this paper. Let $\langle \cdot,\cdot \rangle$ be the standard 
inner product on $\bbC^d$ for any $d \geq 2$. We use $\gamma$ as a multi-index, and let $v_h := (v_1, \cdots, v_{d-1})^T$, $x_h := (x_1, \cdots, x_{d-1})^T$, and $\nabla_h := (\partial_1, \cdots, \partial_{d-1})^T$. For any smooth function $f:\Omega 
\to \bbR$, we use the notation $$\bar{f} := f - \int_{\bbT^{d-1}} f(x) \,\ud x_h.$$

Next we investigate the average of the solution $(v,\tht)$ over time.
\begin{lemma}\label{lem_obs}
	Let $(v,\tht)$ be a smooth solution to \eqref{EQ} with $\alpha \in \{0, 1\}$. Then, there hold
	\begin{equation}\label{avg_vd}
		\int_{\bbT^{d-1}} v_d(t,x) \,\ud x_h = 0, \qquad x_d \in [-1,1]
	\end{equation}
	and
	\begin{equation}\label{avg_tht}
		\int_{\Omega} \tht(t,x) \,\ud x = \int_{\Omega} \tht_0(x) \,\ud x
	\end{equation}
	for all $t \geq 0$. Moreover, if $\alpha = 1$, then
	\begin{equation}\label{avg_vh}
		\int_{\Omega} v_h(t,x) \,\ud x = \int_{\Omega} v_h(0,x) \,\ud x,
	\end{equation}
	and if $\alpha = 0$,
	\begin{equation}\label{avg_vh2}
		\int_{\Omega} v_h(t,x) \,\ud x = e^{-t}\int_{\Omega} v_h(0,x) \,\ud x.
	\end{equation}
\end{lemma}
\begin{proof}
	By the divergence-free condition and the boundary condition $v_d(x_h,-1) = 0$, we have $$0 = -\int_{\bbT^{d-1} \times [-1,x_d]} \nabla_h \cdot v_h \,\ud x = \int_{\bbT^{d-1} \times [-1,x_d]} \partial_d v_d \,\ud x = \int_{\bbT^{d-1}} v_d(x_h,x_d) \,\ud x_h$$ for all $x_d \in [-1,1]$. From the $v_h$ equation in \eqref{EQ}, we have $$\frac {\ud}{\ud t} \int_{\Omega} v_h \,\ud x_h + \int_{\Omega} (-\Delta)^{\alpha} v_h \,\ud x + \int_{\Omega} (v \cdot \nabla) v_h \,\ud x = -\int_{\Omega} \nabla_h P \,\ud x. $$ Integration by parts and the boundary condition for $v_d$ yield $$\int_{\Omega} (v \cdot \nabla) v_h \,\ud x = \int_{\Omega} \nabla_h P \,\ud x = 0,$$ thus, $$\frac {\ud}{\ud t} \int_{\Omega} v_h \,\ud x_h + \int_{\Omega} (-\Delta)^{\alpha} v_h \,\ud x = 0.$$ This gives \eqref{avg_vh2} when $\alpha = 0$. In the case of $\alpha = 1$, $\partial_d v_h = 0$ on $\partial \Omega$ implies \eqref{avg_vh}. Similarly, we can obtain \eqref{avg_tht} by the use of \eqref{avg_vd}. This completes the proof.
\end{proof}

\subsection{Boundary conditions}\label{sec:2.1}
In the section, we briefly show in both cases $\alpha = 0$ and $\alpha =1$ the high compatibility conditions, whose statement is as follows: 
Let $(v,\tht)$ be a global-in-time smooth solution to \eqref{EQ} and suppose that there exists $n \in \bbN$ such that $\partial_d^{2k} \tht_0 = 0$ holds on the boundary for all $0 \leq k \leq n$. Then, we have \begin{equation}\label{bd_cond2}
\partial_d^{2k} v_d = \partial_d^{2k-1+2\alpha} v_h = \partial_d^{2k} \tht = \partial_d^{2k-1} P = 0
\end{equation}
for any $1 \leq k \leq n$. 

When $d=2$, Castro, C\'ordoba, and Lear \cite{CCL} and Dong \cite{Dong2} showed \eqref{bd_cond2} for $\alpha=0$ and $\alpha = 1$ respectively. It is not hard to extend it to the $d \geq 3$ case. Here, we only give details for the case $\alpha = 1$.

From our boundary conditions, we see that \begin{equation}\label{bd_cond}
	v_d(x) = \tht(x) = 0 \qquad \mbox{and} \qquad \partial_d v_h(x) = 0, \qquad x \in \partial\Omega. 
\end{equation}
By \eqref{bd_cond} and the incompressibility, it holds $$\partial_d^2 v_d(x) = -\nabla_h \cdot \partial_dv_h(x) = 0, \qquad x \in \partial \Omega.$$ Then from the $v_d$ equation in \eqref{EQ}, we can see 
\begin{equation*}
	-\partial_d P = \partial_t v_d -\Delta v_d + (v \cdot \nabla)v_d - \tht = 0
\end{equation*}
on the boundary. Next, we apply $\partial_d$ to the $v_h$ equation in \eqref{EQ} and have
\begin{equation*}
	\partial_t \partial_d v_h - \Delta \partial_d v_h + \partial_d (v \cdot \nabla) v_h = -\nabla_h \partial_d P.
\end{equation*}
The previous results implies that $\partial_d^3 v_h = 0$ on the boundary. From the $\tht$ equation in \eqref{EQ}, we can see $$\partial_t \partial_d^2 \tht + \partial_d^2 (v \cdot \nabla) \tht = -\partial_d^2 v_d,$$ hence, $$\partial_t \partial_d^2 \tht + \partial_d v_d \partial_d^2 \tht + (v_h \cdot \nabla_h) \partial_d^2 \tht = 0, \qquad x \in \partial \Omega.$$ Consider the flow map $\Phi(t,x)$ with $\partial_t \Phi(t,x) = (v_h(t,\Phi(t,x)),0)$. Then, it holds $$\frac{\ud}{\ud t} \partial_d^2 \tht(t,\Phi(t,x)) + \partial_dv_d(t,\Phi(t,x)) \partial_d^2 \tht(t,\Phi(t,x))= 0.$$ By the use of Gr\"{o}nwall's inequality, we have $$\partial_d^2 \tht(t,\Phi(t,x)) = \partial_d^2 \tht_0(x)\exp\left(\int_0^t \partial_dv_d(\tau,\Phi(\tau,x)) \,\ud \tau\right).$$ Thus, $\partial_d^2 \tht_0 = 0$ is conserved over time on the boundary, whenever $\partial_d v_d \in L^1_t$. 
Thus, \eqref{bd_cond2} with $k=1$ is obtained. It is clear that $$\partial_d^4 v_d(x) = -\nabla_h \cdot \partial_d^3v_h(x) = 0, \qquad x \in \partial \Omega.$$ Repeating the above processes, we can deduce \eqref{bd_cond2} for all $1 \leq k \leq n$.

\subsection{Functional spaces and orthonormal bases}\label{sec:2.2}
To introduce our solution spaces, we define orthonormal sets $\{ b_q \}_{q \in \bbN}$ and $\{ c_q \}_{q \in \bbN \cup \{0\}}$ by
\begin{align*}
	b_{q} (x_d) &= 
	\begin{cases}
		\displaystyle \sin(\frac {\pi}{2} q x_d)  &\quad  q : \mbox{ even }  \vspace{2mm}\\
		\displaystyle \cos(\frac {\pi}{2} q x_d)  &\quad  q : \mbox{ odd } 
	\end{cases} \qquad \mbox{with }x_d \in [-1,1],
\end{align*}
\begin{align*}
	c_{q} (x_d) &= 
	\begin{cases}
		\displaystyle -\sin(\frac {\pi}{2} q x_d)  &\quad  q : \mbox{ odd } \vspace{2mm}\\
		\displaystyle \cos(\frac {\pi}{2} q x_d)  &\quad  q : \mbox{ even }
	\end{cases} \qquad \mbox{with }x_d \in [-1,1].
\end{align*}
Note that each set is orthonormal basis for $L^2([-1,1])$. Let $$\mathscr{B}_{n,q}(x) := e^{2\pi i n \cdot x_h} b_{q}(x_d), \qquad (n,q) \in \bbZ^{d-1} \times \bbN,$$ $$\mathscr{C}_{n,q}(x) := e^{2 \pi n \cdot x_h} c_{q}(x_d), \qquad (n,q) \in \bbZ^{d-1} \times \bbN \cup \{ 0 \}.$$ Then we have the following relations 
$$\nabla_h \mathscr{B}_{n,q} = 2\pi i n \mathscr{B}_{n,q}, \qquad \nabla_h \mathscr{C}_{n,q} = 2\pi i n \mathscr{C}_{n,q}, \qquad \partial_d \mathscr{B}_{n,q} = \frac \pi2q \mathscr{C}_{n,q}, \qquad \partial_d \mathscr{C}_{n,q} = -\frac \pi2 q \mathscr{B}_{n,q}.$$ Now, we consider the function spaces
\begin{align*}
	X^m(\Omega) &:= \{ f \in H^m(\Omega) ; \partial_d^k f |_{\partial\Omega} = 0, \quad k = 0,2,4, \cdots, m^*\}, \\
	Y^m(\Omega) &:= \{ f \in H^m(\Omega) ; \partial_d^k f |_{\partial\Omega} = 0, \quad k = 1,3,5, \cdots, m_*\},
\end{align*}
where
\begin{align*}
	m^* := 
	\begin{cases}
		\displaystyle m-2, &\quad m : \mbox{ even } \vspace{2mm}\\
		\displaystyle m-1, &\quad m : \mbox{ odd } 
	\end{cases}  \qquad and \qquad m_* :=
	\begin{cases}
		\displaystyle m-1, &\quad m : \mbox{ even } \vspace{2mm}\\
		\displaystyle m-2, &\quad m : \mbox{ odd } .
	\end{cases}	
\end{align*}
Then, $\{ \mathscr{B}_{n,q} \}_{(n,q) \in \bbZ^{d-1} \times \bbN}$ and $\{ \mathscr{C}_{n,q} \}_{(n,q) \in \bbZ^{d-1} \times \bbN \cup \{0\}}$ become orthonormal bases of $X^m(\Omega)$ and $Y^m(\Omega)$ respectively. For the velocity field, we define a $d$-dimensional vector space $\bbX^m(\Omega)$ by $$\bbX^m(\Omega) := \{ v \in H^m(\Omega) ; v = (v_h,v_d) \in Y^m(\Omega) \times X^m(\Omega)\}.$$ We introduce series expansions of the elements in $X^m(\Omega)$ and $Y^m(\Omega)$. Let $$\mathscr{F}_b f(n,q) := \int_{\Omega} f(x) \overline{\mathscr{B}_{n,q}(x)} \,\ud x, \qquad \mathscr{F}_c f(n,q) := \int_{\Omega} f(x) \overline{\mathscr{C}_{n,q}(x)} \,\ud x$$ for each $(n,q) \in \bbZ^{d-1} \times \bbN$ and $(n,q) \in \bbZ^{d-1} \times \bbN \cup \{ 0 \}$ respectively. Then for any $f \in X^m(\Omega)$ and $g \in Y^m(\Omega)$, we can write $$f(x) = \sum_{(n,q) \in \bbZ^{d-1} \times \bbN} \mathscr{F}_bf(n,q) \mathscr{B}_{n,q}(x), \qquad g(x) = \sum_{(n,q) \in \bbZ^{d-1} \times \bbN \cup \{ 0 \}} \mathscr{F}_cg(n,q) \mathscr{C}_{n,q}(x).$$ We refer to \cite[Lemma 3.1]{CCL} for details.

We give two simple lemmas. The first one implies $fg \in X^m$ when $f \in X^m$ and $g \in Y^m$, and the second one implies $fg \in Y^m$ when $f,g \in X^m$ or $f,g \in Y^m$ for any given $m \in \bbN$ with $m > d/2$. Since the proofs are elementary, we omit them.
\begin{lemma}\label{lem_conv}
	Let $q_e$ and $q_o$ be even and odd number respectively. Then, there hold
	\begin{equation*}
		\begin{aligned}
			-\sin(\frac {\pi}{2}q_e x_d) \sin(\frac {\pi}{2} q_o x_d) &= \frac 12 \left( \cos(\frac {\pi}{2}(q_e + q_o)x_d) - \cos(\frac {\pi}{2}(q_e - q_o)x_d) \right), \\
			\sin(\frac {\pi}{2}q_e x_d) \cos(\frac {\pi}{2}q'_e x_d) &= \frac 12 \left( \sin(\frac {\pi}{2}(q_e+q_e')x_d) + \sin(\frac {\pi}{2}(q_e-q_e')x_d) \right), \\
			-\cos(\frac {\pi}{2} q_o x_d) \sin(\frac {\pi}{2} q_o' x_d) &= -\frac 12 \left( \sin(\frac {\pi}{2}(q'_o+q_o)x_d) + \sin(\frac {\pi}{2}(q_o'-q_o)x_d) \right), \\
			\cos(\frac {\pi}{2}q_o x_d) \cos(\frac {\pi}{2} q_e x_d) &= \frac 12 \left( \cos(\frac {\pi}{2}(q_o+q_e)x_d) + \cos(\frac {\pi}{2}(q_o - q_e)x_d) \right).
		\end{aligned}
	\end{equation*}
\end{lemma}
\begin{lemma}
	Let $q-q'$ and $q-q''$ be odd and even number respectively. Then, there hold
	\begin{equation*}
		\begin{aligned}
			-\sin(\frac {\pi}{2}q x_d) \sin(\frac {\pi}{2} q'' x_d) &= \frac 12 \left( \cos(\frac {\pi}{2}(q + q'')x_d) - \cos(\frac {\pi}{2}(q - q'')x_d) \right), \\
			\sin(\frac {\pi}{2}q x_d) \cos(\frac {\pi}{2}q' x_d) &= \frac 12 \left( \sin(\frac {\pi}{2}(q+q')x_d) + \sin(\frac {\pi}{2}(q-q')x_d) \right), \\
			-\cos(\frac {\pi}{2} q x_d) \sin(\frac {\pi}{2} q' x_d) &=- \frac 12 \left( \sin(\frac {\pi}{2}(q'+q)x_d) + \sin(\frac {\pi}{2}(q'-q)x_d) \right), \\
			\cos(\frac {\pi}{2}q x_d) \cos(\frac {\pi}{2} q'' x_d) &= \frac 12 \left( \cos(\frac {\pi}{2}(q+q'')x_d) + \cos(\frac {\pi}{2}(q - q'')x_d) \right).
		\end{aligned}
	\end{equation*}
\end{lemma}
The next proposition provides convolution estimates similar to the Fourier expansion.
\begin{proposition}\label{cor_conv}
	Let $f,f' \in X^m$ and $g,g' \in Y^m$ for some $m \in \bbN$ with $m > \frac d2$. Then, there hold
	\begin{equation*}
	\begin{aligned}
		\sum_{(n,q) \in \bbZ^{d-1} \times \bbN} |\mathscr{F}_b[fg](n,q)| &\leq \left( \sum_{(n,q) \in \bbZ^{d-1} \times \bbN} |\mathscr{F}_b f(n,q)| \right) \left(\sum_{(n,q) \in \bbZ^{d-1} \times \bbN \cup \{ 0 \}} |\mathscr{F}_c g(n,q)| \right), \\
		\sum_{(n,q) \in \bbZ^{d-1} \times \bbN \cup \{ 0 \}} |\mathscr{F}_c[ff'](n,q)| &\leq \left( \sum_{(n,q) \in \bbZ^{d-1} \times \bbN} |\mathscr{F}_b f(n,q)| \right) \left( \sum_{(n,q) \in \bbZ^{d-1} \times \bbN} |\mathscr{F}_b f'(n,q)| \right), \\
		\sum_{(n,q) \in \bbZ^{d-1} \times \bbN \cup \{0\}} |\mathscr{F}_c[gg'](n,q)| &\leq \left( \sum_{(n,q) \in \bbZ^{d-1} \times \bbN \cup \{ 0 \}} |\mathscr{F}_c g(n,q)| \right) \left( \sum_{(n,q) \in \bbZ^{d-1} \times \bbN \cup \{ 0 \}} |\mathscr{F}_c g'(n,q)| \right).
	\end{aligned}
	\end{equation*}
\end{proposition}
\begin{proof}
	We only show the first inequality because the others can be proved similarly. By the series expansions of $f \in X^m(\Omega)$ and $g \in Y^m(\Omega)$, it holds
	\begin{align*}
		fg(x) &= \left( \sum_{(n,q) \in \bbZ^{d-1} \times \bbN} \mathscr{F}_bf(n,q) \mathscr{B}_{n,q}(x) \right) \left( \sum_{(n,q) \in \bbZ^{d-1} \times \bbN \cup \{ 0 \}} \mathscr{F}_cg(n,q) \mathscr{C}_{n,q}(x) \right).
	\end{align*}
	By the use of Lemma~\ref{lem_conv}, 
	we can see for each $(n,q) \in \bbZ^{d-1} \times \bbN$ that
	\begin{gather*}
		|\mathscr{F}_b[fg](n,q)| \\
		\leq \sum_{n'+n''=n} \left( \frac 12 \sum_{q'+q''=q} |\mathscr{F}_bf(n',q')| |\mathscr{F}_cg(n'',q'')| + \frac 12 \sum_{|q'-q''|=q} |\mathscr{F}_bf(n',q')| |\mathscr{F}_cg(n'',q'')| \right).
	\end{gather*}
	This estimate infers
	\begin{equation*}
	\begin{aligned}
		\sum_{(n,q) \in \bbZ^{d-1} \times \bbN} |\mathscr{F}_b[fg]| &\leq \left( \sum_{(n,q) \in \bbZ^{d-1} \times \bbN} |\mathscr{F}_b f| \right) \left( \sum_{(n,q) \in \bbZ^{d-1} \times \bbN \cup \{ 0 \}} |\mathscr{F}_c g| \right).
	\end{aligned}
	\end{equation*}
	This finishes the proof.
\end{proof}

\section{Spectral analysis}\label{sec_spec}
In this section, we give a different form of \eqref{EQ} via spectral analysis. Then, we provide temporal decay estimates for the linear operator of \eqref{EQ}. From now on, we use the notations $\tilde{n} := 2\pi n$ and $\tilde{q} := \frac \pi2 q$ for each $n \in \bbZ^{d-1}$ and $q \in \bbN \cup \{0\}$. We define two sets $$I:= \{\eta = (\tilde{n},\tilde{q}) ; (n,q) \in \bbZ^{d-1} \times \bbN \cup \{0\} \}, \qquad J:= \{\eta = (\tilde{n},\tilde{q}) ; (n,q) \in \bbZ^{d-1} \times \bbN \}.$$ We estimate the pressure term frist. From the $v$ equation in \eqref{EQ}, we can see $$\mathrm{div}\, (v \cdot \nabla) v = -\Delta P + \partial_d \tht.$$ Using the basis $\mathscr{C}_{n,q}(x) = \mathscr{C}_{\eta}(x)$, we have for each $\eta \in I \setminus \{0\}$ that
\begin{equation*}
	\begin{aligned}
		\mathscr{F}_c P(\eta) &= \frac {1}{|\eta|^2} \mathscr{F}_c [\mathrm{div}\, (v \cdot \nabla) v] (\eta) - \frac {1}{|\eta|^2} \mathscr{F}_c \partial_d \tht (\eta).
	\end{aligned}
\end{equation*}
Since $\nabla_h \mathscr{C}_{\eta} = i\tilde{n} \mathscr{C}_{\eta}$ and $\partial_d \mathscr{C}_{\eta} = -\tilde{q} \mathscr{B}_{\eta}$, we can see 
\begin{align*}
	\mathscr{F}_c [\mathrm{div}\, (v \cdot \nabla) v] (\eta) &= \int_{\Omega} \nabla_h \cdot (v \cdot \nabla)v_h(x) \overline{\mathscr{C}_{\eta}(x)}\,\ud x + \int_{\Omega} \partial_d (v \cdot \nabla)v_d (x)\overline{\mathscr{C}_{\eta}(x)} \,\ud x \\
	&= i\tilde{n} \cdot \int_{\Omega} (v \cdot \nabla)v_h(x) \overline{\mathscr{C}_{\eta}(x)}\,\ud x + \tilde{q} \int_{\Omega} (v \cdot \nabla)v_d(x) \overline{\mathscr{B}_{\eta}(x)}\,\ud x \\
	&= i\tilde{n} \cdot \mathscr{F}_c [(v \cdot \nabla)v_h] (\eta) + \tilde{q} \mathscr{F}_b [(v \cdot \nabla)v_d] (\eta)
\end{align*}
and $$ \mathscr{F}_c \partial_d \tht (\eta) = \tilde{q} \mathscr{F}_b \tht (\eta).$$ Thus, we obtain $$\nabla P(x) = \left( \sum_{\eta \in I} \mathscr{F}_c \nabla_h P(\eta) \mathscr{C}_{\eta}(x), \sum_{\eta \in J} \mathscr{F}_b \partial_d P(\eta) \mathscr{B}_{\eta}(x) \right)^{T},$$ where
\begin{equation*}
	\begin{aligned}
		\mathscr{F}_c \nabla_h P(\eta) &= -\frac {\tilde{n} \otimes \tilde{n}}{|\eta|^2} \mathscr{F}_c [(v \cdot \nabla)v_h] (\eta) + i\frac {\tilde{q}\tilde{n}}{|\eta|^2} \mathscr{F}_b [(v \cdot \nabla)v_d] (\eta) - i \frac {\tilde{q}\tilde{n}}{|\eta|^2}  \mathscr{F}_b \tht (\eta)
	\end{aligned}
\end{equation*}
	and
\begin{equation*}
	\begin{aligned}
		\mathscr{F}_b \partial_d P(\eta) &= -i\frac {\tilde{q}\tilde{n}}{|\eta|^2} \cdot \mathscr{F}_c [(v \cdot \nabla)v_h] (\eta) - \frac {\tilde{q}^2}{|\eta|^2} \mathscr{F}_b [(v \cdot \nabla)v_d] (\eta) + \frac {\tilde{q}^2}{|\eta|^2}  \mathscr{F}_b \tht (\eta).
	\end{aligned}
\end{equation*}
From these formulas, we have
\begin{equation}\label{df_vh}
	\begin{gathered}
		\partial_t \mathscr{F}_c v_h + |\eta|^{2\alpha} \mathscr{F}_c v_h + \left(I - \frac {\tilde{n} \otimes \tilde{n}}{|\eta|^2} \right) \mathscr{F}_c [(v \cdot \nabla)v_h] + i\frac {\tilde{q}\tilde{n}}{|\eta|^2} \mathscr{F}_b [(v \cdot \nabla)v_d]  - i\frac {\tilde{q} \tilde{n}}{|\eta|^2}  \mathscr{F}_b \tht = 0
	\end{gathered}
\end{equation}
for $\eta \in I \setminus \{0\}$, and
\begin{equation}\label{df_vd}
	\begin{gathered}
		\partial_t \mathscr{F}_b v_d + |\eta|^{2\alpha} \mathscr{F}_b v_d + \left( 1- \frac {\tilde{q}^2}{|\eta|^2} \right) \mathscr{F}_b [(v \cdot \nabla)v_d] - i\frac {\tilde{q} \tilde{n}}{|\eta|^2} \cdot \mathscr{F}_c [(v \cdot \nabla)v_h] - \frac {|\tilde{n}|^2}{|\eta|^2}  \mathscr{F}_b \tht = 0,
	\end{gathered}
\end{equation}
\begin{equation}\label{df_tht}
	\begin{gathered}
		\partial_t \mathscr{F}_b \tht + \mathscr{F}_b [(v \cdot \nabla)\tht]  + \mathscr{F}_b v_d = 0
	\end{gathered}
\end{equation}
for $\eta \in J$. Due to the linear structure of \eqref{df_vd} and \eqref{df_tht}, we can observe a partially dissipative nature by writing the two equaitons at once with $\bold{u} := (v_d,\tht)^{T}$. 

Let us define an operator $\mathscr{F} : (L^2)^{d-1} \times L^2  \to \bbC^{d-1} \times \bbC$ by $\mathscr{F} := (\mathscr{F}_c,\mathscr{F}_b)$. Then, it follows
\begin{equation}\label{UEQ}
	\partial_t \mathscr{F}_b\bold{u} + M \mathscr{F}_b \bold{u} + \langle \mathbb{P}\, \mathscr{F}(v \cdot \nabla)v, e_{d} \rangle e_1 + \mathscr{F}_b [(v \cdot \nabla)\tht] e_2  = 0,
\end{equation}
where
\begin{equation*}
	\mathbb{P} := I - \frac {1}{|\eta|^2}
	\left(
		\begin{array}{c|c}
			\tilde{n} \otimes \tilde{n} & -i \tilde{q}\tilde{n} \\
			\hline
			i\tilde{q} \tilde{n} & \tilde{q}^2
		\end{array}
	\right), \qquad M := \begin{pmatrix}
		|\eta|^{2\alpha} & -\frac {|\tilde{n}|^2}{|\eta|^2} \\
		1 & 0
	\end{pmatrix}.
\end{equation*}
For simplicity, we use the notation $$N(v,\tht) := \langle \mathbb{P}\, \mathscr{F}(v \cdot \nabla)v, e_{d} \rangle e_1 + \mathscr{F}_b [(v \cdot \nabla)\tht] e_2.$$ Since the characteristic equation of $M^{T}$ is given by
\begin{equation*}
	\operatorname{det} \big( M^{T} - \lambda I \big) = \lambda^2 - |\eta|^{2\alpha} \lambda + \frac {|\tilde{n}|^2}{|\eta|^2},
\end{equation*}
the two pair of eigenvalue and eigenvector $(\lambda_\pm(\eta), \overline{\bold{a}_\pm(\eta)})$ satisfy
\begin{equation*}
	\lambda_{\pm}(\eta) = \frac {|\eta|^{2\alpha} \pm \sqrt{|\eta|^{4\alpha} - {4 |\tilde{n}|^2}/{|\eta|^2}}}{2}, \qquad \overline{\bold{a}_{\pm} (\eta)} = 
	\begin{pmatrix}
		\lambda_{\pm} \\
		-\frac {|\tilde{n}|^2}{|\eta|^2}
	\end{pmatrix} ,
\end{equation*}
where $M^{T} \overline{\bold{a}_\pm (\eta)} = \lambda(\eta)_\pm \overline{\bold{a}_\pm (\eta)}$ holds. We note that there is no pair $\eta \in J$ satisfying $|\eta|^{4\alpha} - {4 |\tilde{n}|^2}/{|\eta|^2} = 0$. Since 
\begin{equation*}
	A:= (\overline{\bold{a}_+}\,\, \overline{\bold{a}_-}) \qquad \mbox{and} \qquad  B := \frac {1}{\lambda_{+}-\lambda_{-}} 
	\begin{pmatrix}
		1 & \frac {| \eta|^2}{|\tilde{n}|^2} \lambda_{-} \\
		-1 & -\frac {| \eta|^2}{|\tilde{n}|^2} \lambda_{+}
		\end{pmatrix} = 
	\begin{pmatrix}
		\bold{b}_+ \\
		\bold{b}_-
	\end{pmatrix}
\end{equation*}
satisfy $BA = I$, it follows by Duhamel's principle
\begin{equation}\label{df_u+}
	\langle \mathscr{F}_b\bold{u}(t),\bold{a}_\pm \rangle \bold{b}_\pm = e^{-\lambda_\pm t} \langle \mathscr{F}_b\bold{u}_0,\bold{a}_\pm \rangle \bold{b}_\pm - \int_0^t e^{-\lambda_\pm (t - \tau)} \langle N(v,\tht)(\tau),\bold{a}_\pm \rangle \bold{b}_\pm \,\mathrm{d}\tau .
\end{equation}
However, using this formula directly can be problematic because of the unboundedness of $|\bold{b}_{\pm}|$ around the set $\{ |\eta|^{4\alpha} = 4|\tilde{n}|^2/|\eta|^2 \}$. For this reason, we employ
\begin{equation}\label{df_tht_1}
	\begin{aligned}
	\mathscr{F}_b \tht(t) &= \sum_{j \in \pm} e^{-\lambda_j t} \langle \mathscr{F}_b\bold{u}_0,\bold{a}_j \rangle \langle \bold{b}_j,e_{2} \rangle - \sum_{j \in \pm} \int_0^t e^{-\lambda_j (t - \tau)} \langle N(v,\tht)(\tau),\bold{a}_j \rangle \langle \bold{b}_j,e_{2} \rangle \,\mathrm{d}\tau \\
	&= (e^{-\lambda_- t} - e^{-\lambda_+ t}) \langle \mathscr{F}_b\bold{u}_0,\bold{a}_- \rangle \langle \bold{b}_-,e_{2} \rangle + e^{-\lambda_+ t} \mathscr{F}_b \tht_0 \\
	&\hphantom{\qquad\qquad}- \int_0^t (e^{-\lambda_- (t - \tau)} - e^{-\lambda_+ (t - \tau)}) \langle N(v,\tht)(\tau),\bold{a}_- \rangle \langle \bold{b}_-,e_{2} \rangle \,\mathrm{d}\tau \\
	&\hphantom{\qquad\qquad}- \int_0^t e^{-\lambda_+ (t - \tau)} \mathscr{F}_b [(v \cdot \nabla)\tht] \,\mathrm{d}\tau,
	\end{aligned}
\end{equation}
which allows us to get rid of the singularity of $\mathbf{b}_{\pm}$. Here, we note some useful calculations when using \eqref{df_tht_1}. From the definition of $\lambda_{\pm}$, $\mathbf{a}_{\pm}$, and $\mathbf{b}_{\pm}$, we have
\begin{equation}\label{ker_est}
	|e^{-\lambda_{+}(\eta)t}| \leq e^{-|\eta|^{2\alpha} \frac t2}, \qquad |e^{\lambda_{-}(\eta)t}| \leq \begin{cases} e^{-|\eta|^{2\alpha} \frac t2}, &\qquad |\eta|^{2+4\alpha} - 4|\tilde{n}|^2 \leq 0, \\
	e^{-\frac {|\tilde{n}|^2}{|\eta|^{2+2\alpha}}t}, &\qquad |\eta|^{2+4\alpha} - 4|\tilde{n}|^2 \geq 0, \end{cases}
\end{equation}
\begin{equation*}
	|\bold{a}_{-}|^2 = |\lambda_-|^2 + \frac {|\tilde{n}|^4}{|\eta|^4}, \qquad |\langle \bold{b}_-,e_2 \rangle|^2 = \frac {|\eta|^4|\lambda_{+}|^2}{|\tilde{n}|^4|\lambda_+- \lambda_-|^2}.
\end{equation*}
Thus, it follows
\begin{equation*}
	|\bold{a}_-||\langle \bold{b}_-,e_2 \rangle| \leq \begin{cases} \displaystyle\frac {C}{|\lambda_+ - \lambda_-|}, &\qquad |\eta|^{2+4\alpha} - 4|\tilde{n}|^2 \leq 0, \vspace{2mm}\\
	\displaystyle\frac {C|\eta|^{2\alpha}}{|\lambda_+ - \lambda_-|}, &\qquad |\eta|^{2+4\alpha} - 4|\tilde{n}|^2 \geq 0. \end{cases}
\end{equation*}
Let us consider the three sets
\begin{equation*}
	\begin{aligned}
		D_1 &:= \{ \eta \in J ; |\eta|^{4\alpha} - \frac {4 |\tilde{n}|^2}{|\eta|^2} \leq 0\}, \\
		D_2 &:= \{ \eta \in J ; 0 \leq |\eta|^{4\alpha} - \frac {4 |\tilde{n}|^2}{|\eta|^2} \leq \frac 14 |\eta|^{4\alpha}\}, \\
		D_3 &= \{ \eta \in J ; |\eta|^{4\alpha} - \frac {4 |\tilde{n}|^2}{|\eta|^2} \geq \frac 14 |\eta|^{4\alpha}\},
	\end{aligned}	
\end{equation*}
with $J = D_1 \cup D_2 \cup D_3$. Then, for any $\bold{f} \in \bbC^2$, there exists a constant $C>0$ such that
\begin{equation}\label{sing_est}
\begin{aligned}
	|(e^{-\lambda_- t} - e^{-\lambda_+ t}) \langle \bold{f},\bold{a}_- \rangle \langle \bold{b}_-,e_{2} \rangle| &\leq C e^{-|\eta|^{2\alpha} \frac t4} |\bold{f}|, &\quad\eta \in D_1, \\
	|(e^{-\lambda_- t} - e^{-\lambda_+ t}) \langle \bold{f},\bold{a}_- \rangle \langle \bold{b}_-,e_{2} \rangle| &\leq Ce^{-|\eta|^{2\alpha} \frac t4} |\bold{f}|, &\quad\eta \in D_2, \\
	|(e^{-\lambda_- t} - e^{-\lambda_+ t}) \langle \bold{f},\bold{a}_- \rangle \langle \bold{b}_-,e_{2} \rangle| &\leq Ce^{-\frac {|\tilde{n}|^2}{|\eta|^{2+2\alpha}} t} |\bold{f}|, &\quad\eta \in D_3.
\end{aligned}
\end{equation}
For the first and second inequalities, we apply the mean value theorem so that $$|(e^{-\lambda_- t} - e^{-\lambda_+ t}) \langle \bold{f},\bold{a}_- \rangle \langle \bold{b}_-,e_{2} \rangle| \leq C \frac {|e^{-\lambda_- t} - e^{-\lambda_+ t}|}{|\lambda_+ - \lambda_-|} |\bold{f}| \leq Cte^{-|\eta|^{2\alpha} \frac t2} |\bold{f}|$$ for any $\eta \in D_1$ and $$|(e^{-\lambda_- t} - e^{-\lambda_+ t}) \langle \bold{f},\bold{a}_- \rangle \langle \bold{b}_-,e_{2} \rangle| \leq C|\eta|^{2\alpha} \frac {|e^{-\lambda_- t} - e^{-\lambda_+ t}|}{|\lambda_+ - \lambda_-|} |\bold{f}| \leq C|\eta|^{2\alpha}te^{-|\eta|^{2\alpha} \tau} |\bold{f}|, \qquad \tau \in (\frac t4 , \frac t2)$$ for any $\eta \in D_2$. One can easily obtain the last inequality in \eqref{sing_est} by the use of $|\lambda_+ - \lambda_-| \geq \frac 12|\eta|^{2\alpha}$.

Now, we are ready to show temporal decay estimates of solutions to the linearized system of \eqref{UEQ}:
\begin{equation}\label{lin_UEQ}
	\partial_t \mathscr{F}_b\bold{u} + M \mathscr{F}_b \bold{u} = 0.
\end{equation}
\begin{lemma}\label{lem_lin}
	Let $d \in \mathbb{N}$ with $d \geq 2$ and $m \in \mathbb{N}$. Let $\bfu_0 \in X^m(\Omega)$. Then, there exists a unique smooth global smooth global solution $\bfu = (v_d,\tht)$ to \eqref{lin_UEQ} such that
	\begin{equation}\label{lin_dec_vd}
		\| v_d(t) \|_{\dot{H}^s}^2 \leq Ce^{-\frac t4} \| \bfu_0 \|_{\dot{H}^s} + C(1+t)^{-(1+\frac {m-s}{2(1+\alpha)})} \| \bfu_0 \|_{\dot{H}^m}
	\end{equation}
	and
	\begin{equation}\label{lin_dec_tht}
		\| \bar{\tht}(t) \|_{\dot{H}^s} \leq Ce^{-\frac t4} \| \bfu_0 \|_{\dot{H}^s} + C(1+t)^{-\frac {m-s}{2(1+\alpha)}} \| \bfu_0 \|_{\dot{H}^m}
	\end{equation}	
	for all $s \in [0,m]$.
\end{lemma}
\begin{proof}
	We recall
	\begin{equation*}
	\mathscr{F}_b\bfu  = (e^{-\lambda_- t} - e^{-\lambda_+ t}) \langle \mathscr{F}_b\bold{u}_0,\bold{a}_- \rangle \bold{b}_- + e^{-\lambda_+ t} \mathscr{F}_b \bfu_0
\end{equation*}
	and prove \eqref{lin_dec_tht} first. We can see $$\| \bar{\tht} \|_{\dot{H}^s} \leq \left( \sum_{\tilde{n} \neq 0} |\eta|^{2s} |(e^{-\lambda_- t} - e^{-\lambda_+ t}) \langle \mathscr{F}_b\bold{u}_0,\bold{a}_- \rangle \langle \bold{b}_-,e_2 \rangle|^2 \right)^{\frac 12} + \left( \sum_{\tilde{n} \neq 0} |\eta|^{2s} |e^{-\lambda_+ t} \langle \mathscr{F}_b \bfu_0,e_2 \rangle|^2 \right)^{\frac 12}.$$ From \eqref{ker_est} it is clear that $$\left( \sum_{\tilde{n} \neq 0} |\eta|^{2s} |e^{-\lambda_+ t} \langle \mathscr{F}_b \bfu_0,e_2 \rangle|^2 \right)^{\frac 12} \leq e^{-\frac t2} \left( \sum_{\tilde{n} \neq 0} |\eta|^{2s} |\mathscr{F}_b \bfu_0|^2 \right)^{\frac 12}.$$ On the other hand, \eqref{sing_est} gives 
	\begin{gather*}
		\left( \sum_{\tilde{n} \neq 0} |\eta|^{2s} |(e^{-\lambda_- t} - e^{-\lambda_+ t}) \langle \mathscr{F}_b\bold{u}_0,\bold{a}_- \rangle \langle \bold{b}_-,e_2 \rangle|^2 \right)^{\frac 12} \\
		\leq Ce^{-\frac t4} \left( \sum_{\tilde{n} \neq 0} |\eta|^{2s} |\mathscr{F}_b \bfu_0|^2 \right)^{\frac 12} + C\left( \sum_{\{ \tilde{n} \neq 0 \} \cap D_3} e^{-\frac {2|\tilde{n}|^2}{|\eta|^{2+2\alpha}} t} |\eta|^{2s} |\mathscr{F}_b \bfu_0|^2 \right)^{\frac 12}.
	\end{gather*}
	Since
	\begin{align*}
		e^{-\frac {2|\tilde{n}|^2}{|\eta|^{2+2\alpha}} t} |\eta|^{2s} |\mathscr{F}_b \bfu_0|^2 &\leq \left( \frac {|\tilde{n}|^2}{|\eta|^{2+2\alpha}} t \right)^{-\frac {m-s}{1+\alpha}} \left( \frac {|\tilde{n}|^2}{|\eta|^{2+2\alpha}} t \right)^{\frac {m-s}{1+\alpha}} e^{-\frac {2|\tilde{n}|^2}{|\eta|^{2+2\alpha}} t} |\eta|^{2s} |\mathscr{F}_b \bfu_0|^2 \\
		&\leq Ct^{-\frac {m-s}{1+\alpha}} |\eta|^{2m} |\mathscr{F}_b \bfu_0|^2, 
	\end{align*}
	and $$e^{-\frac {2|\tilde{n}|^2}{|\eta|^{2+2\alpha}} t} |\eta|^{2s} |\mathscr{F}_b \bfu_0|^2 \leq |\eta|^{2s} |\mathscr{F}_b \bfu_0|^2,$$
	it follows 
	\begin{equation}\label{lin_est}
	\left( \sum_{\{ \tilde{n} \neq 0 \} \cap D_3} e^{-\frac {|\tilde{n}|^2}{|\eta|^{2+2\alpha}} t} |\eta|^{2s} |\mathscr{F}_b \bfu_0|^2 \right)^{\frac 12} \leq C(1+t)^{-\frac {m-s}{2(1+\alpha)}} \left( \sum_{\tilde{n} \neq 0} |\eta|^{2m} |\mathscr{F}_b \bfu_0|^2 \right)^{\frac 12}.
	\end{equation} Collecting the above estimates, we deduce \eqref{lin_dec_tht}.
	
	It remains to show \eqref{lin_dec_vd}. We can see $$\| v_d \|_{\dot{H}^s} \leq \left( \sum_{\tilde{n} \neq 0} |\eta|^{2s} |(e^{-\lambda_- t} - e^{-\lambda_+ t}) \langle \mathscr{F}_b\bold{u}_0,\bold{a}_- \rangle \langle \bold{b}_-,e_1 \rangle|^2 \right)^{\frac 12} + e^{-\frac t2} \left( \sum_{\tilde{n} \neq 0} |\eta|^{2s} |\mathscr{F}_b \bfu_0|^2 \right)^{\frac 12}.$$ Since $\langle \bfb_{-},e_1 \rangle = \frac {|\tilde{n}|^2}{|\lambda_+||\eta|^2}\langle \bfb_{-},e_2 \rangle$, using $\frac {|\tilde{n}|^2}{|\lambda_+||\eta|^2} \leq C$ for $\eta \in D_1 \cup D_2$ and $\frac {|\tilde{n}|^2}{|\lambda_+||\eta|^2} \leq 2 \frac {|\tilde{n}|^2}{|\eta|^{2+2\alpha}}$ for $\eta \in D_3$, we have 	\begin{gather*}
		\left( \sum_{\tilde{n} \neq 0} |\eta|^{2s} |(e^{-\lambda_- t} - e^{-\lambda_+ t}) \langle \mathscr{F}_b\bold{u}_0,\bold{a}_- \rangle \langle \bold{b}_-,e_1 \rangle|^2 \right)^{\frac 12} \\
		\leq Ce^{-\frac t4} \left( \sum_{\tilde{n} \neq 0} |\eta|^{2s} |\mathscr{F}_b \bfu_0|^2 \right)^{\frac 12} + C\left( \sum_{\{ \tilde{n} \neq 0 \} \cap D_3} \frac {|\tilde{n}|^4}{|\eta|^{4+4\alpha}} e^{-\frac {2|\tilde{n}|^2}{|\eta|^{2+2\alpha}} t} |\eta|^{2s} |\mathscr{F}_b \bfu_0|^2 \right)^{\frac 12}.
	\end{gather*}
  	Estimating as in \eqref{lin_est}, we deduce $$\left( \sum_{\{ \tilde{n} \neq 0 \} \cap D_3} \frac {|\tilde{n}|^4}{|\eta|^{4+4\alpha}} e^{-\frac {2|\tilde{n}|^2}{|\eta|^{2+2\alpha}} t} |\eta|^{2s} |\mathscr{F}_b \bfu_0|^2 \right)^{\frac 12} \leq C(1+t)^{-(1+\frac {m-s}{2(1+\alpha)})} \left( \sum_{\tilde{n} \neq 0} |\eta|^{2m} |\mathscr{F}_b \bfu_0|^2 \right)^{\frac 12}$$ from which we obtain \eqref{lin_dec_vd}. 
	This completes the proof.
\end{proof}

\section{Energy estimates}\label{sec:energy}

In this section, we provide 
the energy estimates which specify the quantities that should be computed via the spectral analysis. We start with the following standard local existence result. For the proof, we refer to \cite{CCL,MB01}. 

\begin{proposition}[Local well-posedness]\label{prop_loc}
	Let $d \in \mathbb{N}$ with $d \geq 2$ and $\alpha \in \{0,1\}$. Let $m \in \bbN$ with $m > 1+\frac d2-\alpha$ and an initial data $\tht_0 \in X^m$ and $v_0 \in \bbX^m$. Then there exists a $T>0$ such that there exists a unique classical solution $(v, \theta)$ to the stratified Boussinesq equations \eqref{EQ} satisfying
	\begin{equation*}
		v \in L^\infty(0,T ; \bbX^m(\Omega)), \qquad \tht \in L^\infty(0,T ; X^m(\Omega)).
	\end{equation*}
	Let $T^* \in (0,\infty]$ be the maximal time of existence. Moreover, if $T^* < \infty$, then it holds
	\begin{equation*}
		\lim_{t \nearrow T^*} \big( \| v( t) \|_{H^m}^2 + \| \theta( t) \|_{H^m}^2 \big) = \infty.
	\end{equation*}
\end{proposition}

We will frequently used the following result on the product estimates 
 (see \cite{Ferrari} for the proof).

\begin{lemma}\label{cal_lem}
	Let $m \in \bbN$. Then for any subset $D \subset \{ \gamma ; |\gamma| = m \}$, there exists a constant $C=C(m)>0$ such that $$\left\| \sum_{\gamma \in D} \partial^{\gamma} (fg) \right\|_{L^2} \leq C \left( \| f \|_{H^m} \| g \|_{L^{\infty}} + \| f \|_{L^{\infty}} \| g \|_{H^m} \right)$$ for all $f, g \in H^m(\Omega) \cap L^{\infty}(\Omega)$. Moreover, if $m > 1$, then it holds $$\left\| \sum_{\gamma \in D} \partial^{\gamma} (fg) \right\|_{L^2} \leq C \| f \|_{H^m} \| g \|_{H^m}.$$ 
\end{lemma}

Let us use the notations for $k \in \bbN$
\begin{equation*}
	E_k(t) := \left( \| v(t) \|_{H^k}^2 + \| \theta(t) \|_{H^k}^2 \right)^{\frac 12},
\end{equation*}
\begin{equation*}
	A_k(t) := \sum_{|\gamma|=1}^{k} \int_{\Omega} \partial^\gamma v_d(t) \partial^\gamma \theta(t) \,\mathrm{d}x.
\end{equation*}
Note that Young's inequality implies
\begin{equation}\label{A}
	|A_{k}(t)| \leq \frac 12 E_{k}(t)^2 .
\end{equation}
\begin{proposition}\label{prop_eng}
	Let $d \in \mathbb{N}$ with $d \geq 2$ and $m \in \mathbb{N}$ with $m \geq 2+\frac d2$. Assume that $(v, \theta)$ is a smooth global solution to \eqref{EQ} with $\alpha = 1$. Then there exists a constant $C > 0$ such that
	\begin{equation}\label{eng_est}
		\begin{gathered}
			\ddt ( E_m(t)^2 - A_{m-1}(t) ) + \frac 12 \| \nabla v(t) \|_{H^m}^2 + \frac 12 \| \nabla_h \theta(t) \|_{H^{m-2}}^2 \\
			\leq C\| \nabla v(t) \|_{H^m} \| \nabla_h \theta(t) \|_{H^{m-2}} \| \tht(t) \|_{H^m} + C\big(\| \tht(t) \|_{H^m}^2 + \| v(t) \|_{H^m}^2 \big) \| \nabla v(t) \|_{L^{\infty}} 
		\end{gathered}
	\end{equation}
	for all $t>0$.
\end{proposition}
\begin{proof}
	From the system \eqref{EQ}, we have
	\begin{equation*}
		\begin{gathered}
			\frac 12 \ddt \big( \| v(t) \|_{H^m}^2 + \| \theta(t) \|_{H^m}^2 \big) + \| \nabla v \|_{H^m}^2 \\
			\leq - \sum_{1 \leq |\gamma| \leq m} \int_{\Omega} \partial^\gamma (v \cdot \nabla) v  \partial^\gamma v \, \mathrm{d}x - \sum_{1 \leq |\gamma| \leq m} \int_{\Omega} \partial^\gamma (v \cdot \nabla) \theta \partial^\gamma \theta \, \mathrm{d}x.
		\end{gathered}
	\end{equation*}
	We only consider the $|\gamma| = m$ case because the others can be treated similarly. It is clear by the divergence-free condition and the boundary condition that $$\int_{\Omega} (v \cdot \nabla \partial^\gamma v) \partial^\gamma v \,\mathrm{d}x = \int_{\Omega} (v \cdot \nabla \partial^\gamma \theta) \partial^\gamma \theta \,\mathrm{d}x = 0, \qquad |\gamma| = m.$$ Thus, Lemma~\ref{cal_lem} implies
	\begin{equation*}
		\left| - \sum_{|\gamma| = m} \int_{\Omega} \partial^\gamma (v \cdot \nabla) v  \partial^\gamma v \, \mathrm{d}x -\sum_{|\gamma|=m-1} \int_{\Omega} (\nabla v \cdot \partial^\gamma\nabla \theta) \cdot \partial^\gamma \nabla \theta \,\mathrm{d}x \right| \leq C\| \nabla v \|_{L^{\infty}} (\| v \|_{H^m}^2+\| \tht \|_{H^m}^2).
	\end{equation*}
	For estimating the remainder term
	\begin{equation*}
		\left| \sum_{|\gamma| = m-2} \int_{\Omega} \partial^\gamma (\Delta v \cdot \nabla \theta) \partial^\gamma \Delta \theta \,\mathrm{d}x \right|,
	\end{equation*}
	we use a simple formula $\Delta v \cdot \nabla \tht = \Delta v_h \cdot \nabla_h \tht + \Delta v_d \partial_d \tht$. We can infer from H\"{o}lder's inequality with Sobolev embeddings that 
	\begin{align*}
		\left| \sum_{|\gamma| = m-2} \int_{\Omega} \partial^\gamma (\Delta v_h \cdot \nabla_h \theta) \partial^\gamma \Delta \theta \,\mathrm{d}x \right| &\leq \| \Delta v_h \cdot \nabla_h \theta \|_{H^{m-2}} \| \tht \|_{H^{m}} \\
		&\leq C\| \nabla v \|_{H^m} \| \nabla_h \tht \|_{H^{m-2}} \| \tht \|_{H^m} + \| \Delta v \|_{L^{\infty}} \| \nabla_h \tht \|_{H^{m-2}} \| \tht \|_{H^m}.
	\end{align*}
	Since $m \geq 2 + d/2$ implies $$\| \Delta v \|_{L^{\infty}} \leq C\| \nabla v \|_{L^{\infty}} + C\| \nabla v \|_{H^m},$$ it follows
	\begin{align*}
		\left| \sum_{|\gamma| = m-2} \int_{\Omega} \partial^\gamma (\Delta v_h \cdot \nabla_h \theta) \partial^\gamma \Delta \theta \,\mathrm{d}x \right| &\leq C\| \nabla v \|_{H^m} \| \nabla_h \tht \|_{H^{m-2}} \| \tht \|_{H^m} + C\| \nabla v \|_{L^{\infty}} \| \tht \|_{H^m}^2.
	\end{align*}	
	Otherwise, we have 
	\begin{equation}\label{eng_1}
	\begin{gathered}
		\left| \sum_{|\gamma| = m-2} \int_{\Omega} \partial^\gamma (\Delta v_d \partial_d \theta) \partial^\gamma  \Delta \theta \,\mathrm{d}x \right| \\
		\leq \left| \sum_{|\gamma| = m-3} \int_{\Omega} \partial^\gamma \nabla_h(\Delta v_d \partial_d \theta) \cdot \partial^\gamma \Delta \nabla_h \theta \,\mathrm{d}x \right| + \left| \int_{\Omega} \partial_d^{m-2}(\partial_d^2 v_d \partial_d \theta) \partial_d^m \theta \,\mathrm{d}x \right|.
	\end{gathered}
	\end{equation}
	Here, we need to estimate carefully with the boundary conditions. The first integral on the right-hand side is bounded by $$\left| \sum_{|\gamma| = m-3} \int_{\partial\Omega} \partial^\gamma \nabla_h(\Delta v_d \partial_d \theta) \cdot \partial_d \partial^\gamma \nabla_h \theta \,\mathrm{d}x_h \right| + \left| \sum_{|\gamma| = m-2} \int_{\Omega} \partial^\gamma \nabla_h(\Delta v_d \partial_d \theta) \cdot \partial^\gamma \nabla_h \theta \,\mathrm{d}x \right|.$$ Since $v_d \in X^{m+1}(\Omega)$ and $\theta \in X^m(\Omega)$ for a.e. $t>0$, the boundary term vanishes. Lemma~\ref{cal_lem} implies $$\left| \sum_{|\gamma| = m-2} \int_{\Omega} \partial^\gamma \nabla_h(\Delta v_d \partial_d \theta) \cdot \partial^\gamma \nabla_h \theta \,\mathrm{d}x \right| \leq C\| \nabla v \|_{H^m} \| \nabla_h \tht \|_{H^{m-2}} \| \tht \|_{H^m}.$$ On the other hand, we write the second integral in \eqref{eng_1} as
	\begin{align*}
		\left| \int_{\Omega} \partial_d^{m-2}(\partial_d^2 v_d \partial_d \theta) \partial_d^m \theta \,\mathrm{d}x \right| \leq \left| \int_{\Omega} \partial_d^{m-2}\nabla_h \cdot (\partial_d v_h \partial_d \theta) \partial_d^m \theta \,\mathrm{d}x \right| + \left| \int_{\Omega} \partial_d^{m-2}(\partial_d v_h \cdot \nabla_h \partial_d \theta) \partial_d^m \theta \,\mathrm{d}x \right|.
	\end{align*}
	It can be shown by H\"older's inequalities and Sobolev embeddings that 
	\begin{align*}
		\left| \int_{\Omega} \partial_d^{m-2}(\partial_d v_h \cdot \nabla_h \partial_d \theta) \partial_d^m \theta \,\mathrm{d}x \right| &\leq (\| \partial_d^{m-3}(\partial_d^2 v_h \cdot \nabla_h \partial_d \theta) \|_{L^2} + \| \partial_d v_h \cdot \nabla_h \partial_d^{m-1} \theta \|_{L^2}) \| \partial_d^m \tht \|_{L^2} \\
		&\leq C \| \nabla v \|_{H^m} \| \nabla_h \tht \|_{H^{m-2}} \| \tht \|_{H^{m}} + C \| \nabla v \|_{L^{\infty}} \| \tht \|_{H^{m}}^2.
	\end{align*}
	Since $\partial_d v_h \partial_d \tht \in X^{m-1}$ and $\tht \in X^m$, it holds
	\begin{equation}\label{sum_est}
	\begin{aligned}
		\left| \int_{\Omega} \partial_d^{m-2} \nabla_h \cdot (\partial_d v_h \partial_d \tht) \partial_d^m \tht \,\ud x \right| &= \left| \sum_{\eta \in J} \tilde{q}^{m-2} i\tilde{n} \cdot \mathscr{F}_b (\partial_d v_h \partial_d \tht)(\eta) \overline{\tilde{q}^{m} \mathscr{F}_b \tht(\eta)} \right| \\
		&= \left| \sum_{\eta \in J} \tilde{q}^{m-1} \mathscr{F}_b (\partial_d v_h \partial_d \tht)(\eta) \cdot \overline{i\tilde{n} \tilde{q}^{m-1} \mathscr{F}_b \tht(\eta)} \right| \\
		&= \left| \int_{\Omega} \partial_d^{m-1} (\partial_d v_h \partial_d \tht) \cdot \nabla_h \partial_d^{m-1} \tht \,\ud x \right|.
	\end{aligned}
	\end{equation}
	Then, we have
	\begin{align*}
		&\left| \int_{\Omega} \partial_d^{m-1}(\partial_d v_h \partial_d \theta) \cdot \nabla_h \partial_d^{m-1} \theta \,\mathrm{d}x \right| \\
		&\hphantom{\qquad\qquad} \leq \left| \int_{\Omega} (\partial_d v_h \partial_d^m \theta) \cdot \nabla_h \partial_d^{m-1} \theta \,\mathrm{d}x \right| + \left| \int_{\Omega} \partial_d^{m-2}(\partial_d^2 v_h \partial_d \theta) \cdot \nabla_h \partial_d^{m-1} \theta \,\mathrm{d}x \right| \\
		&\hphantom{\qquad\qquad}= \left| \int_{\Omega} (\partial_d v_h \partial_d^m \theta) \cdot \nabla_h \partial_d^{m-1} \theta \,\mathrm{d}x \right| + \left| \int_{\Omega} \partial_d^{m-1}(\partial_d^2 v_h \partial_d \theta) \cdot \nabla_h \partial_d^{m-2} \theta \,\mathrm{d}x \right| \\
		&\hphantom{\qquad\qquad}\leq C \| \nabla v \|_{L^{\infty}} \| \tht \|_{H^{m}}^2 + C \| \nabla v \|_{H^m} \| \nabla_h \tht \|_{H^{m-2}} \| \tht \|_{H^{m}}.
	\end{align*}
Combining the above estimates, we have
	\begin{equation}\label{est_A_1}
		\begin{gathered}
			\frac 12 \ddt \big( \| v(t) \|_{H^m}^2 + \| \theta(t) \|_{H^m}^2 \big) + \| \nabla v \|_{H^m}^2 \\
			 \leq C \| \nabla v \|_{L^{\infty}} (\| v \|_{H^m}^2 + \| \tht \|_{H^{m}}^2) + C \| \nabla v \|_{H^m} \| \nabla_h \tht \|_{H^{m-2}} \| \tht \|_{H^{m}}.
		\end{gathered}
	\end{equation}

	Now we claim that
	\begin{equation}\label{est_A}
		\frac 32 \| \nabla v \|_{H^m}^2 \geq \frac 12 \| \nabla_h \theta \|_{H^{m-2}}^2 - \ddt A_{m-1}(t) - C E_m(t) \| \nabla v \|_{H^{m-2}}^2.
	\end{equation}
	We recall the $v_d$ equation in \eqref{EQ}
	\begin{equation}\label{Pvd}
		\partial_t v_d -\Delta v_d + \langle \mathbb{P}\,(v \cdot \nabla)v , e_d \rangle = \langle \mathbb{P} \,\theta e_d,e_d \rangle.
	\end{equation}
	We first take $-\Delta$ on the both sides of \eqref{Pvd}. Since the definition of $\mathbb{P}$ implies
	\begin{equation*}
		-\Delta \langle \mathbb{P}\,(v \cdot \nabla)v , e_d \rangle = -\Delta (v\cdot \nabla) v_d + \partial_d \nabla \cdot (v\cdot \nabla) v 
	\end{equation*}
	and
	\begin{equation*}
		-\Delta \langle \mathbb{P} \,\theta e_d,e_d \rangle =  - \Delta_h \theta,
	\end{equation*}
	it follows $$\partial_t (-\Delta)v_d + (-\Delta)^2 v_d  -\Delta (v\cdot \nabla) v_d + \partial_d \nabla \cdot (v\cdot \nabla) v  = -\Delta_h \theta.$$ Then, we have for $|\gamma| = m-2$ that
	\begin{gather*}
		\int_{\Omega} \partial_t \nabla \partial^{\gamma} v_d \cdot \nabla \partial^{\gamma} \tht \,\ud x + \int_{\Omega} \nabla (-\Delta) \partial^{\gamma} v_d \cdot \nabla  \partial^{\gamma}  \tht \,\ud x + \int_{\Omega} \nabla \partial^{\gamma} (v \cdot \nabla) v_d \cdot \nabla \partial^{\gamma} \tht \,\ud x \\
	- \int_{\Omega} \nabla \cdot \partial^{\gamma} (v \cdot \nabla)v \partial_d \partial^{\gamma} \tht \,\ud x = \int_{\Omega} | \nabla_h \partial^{\gamma} \tht|^2 \,\ud x. 
	\end{gather*}
	On the other hand, we have from the $\tht$ equation in \eqref{EQ}
	\begin{equation*}
		\int_{\Omega} \partial_t \nabla \partial^{\gamma} \theta \cdot \nabla \partial^{\gamma} v_d \,\mathrm{d}x + \int_{\Omega} \nabla \partial^{\gamma} (v\cdot \nabla) \theta \cdot \nabla \partial^{\gamma} v_d \,\mathrm{d}x + \int_{\Omega} |\nabla \partial^{\gamma} v_d|^2 \,\ud x = 0.
	\end{equation*}
	Adding these two equalities gives
	\begin{gather*}
		\frac {\mathrm{d}}{\mathrm{d}t} \int_{\Omega} \nabla \partial^{\gamma} v_d \cdot \nabla \partial^{\gamma} \theta \, \mathrm{d}x + \int_{\Omega} \nabla (-\Delta) \partial^{\gamma} v_d \cdot \nabla \partial^{\gamma} \tht \,\ud x - \int_{\Omega} \nabla \cdot \partial^{\gamma} (v \cdot \nabla)v \partial_d \partial^{\gamma} \tht \,\ud x \\
		+ \int_{\Omega} \big( \nabla\partial^{\gamma} (v \cdot \nabla) v_d \cdot \nabla \partial^{\gamma} \theta + \nabla \partial^{\gamma} (v\cdot \nabla) \theta \cdot \nabla \partial^{\gamma} v_d\big) \, \mathrm{d}x + \int_{\Omega} |\nabla \partial^{\gamma} v_d|^2 \,\ud x = \int_{\Omega} | \nabla_h \partial^{\gamma} \tht|^2 \,\ud x.
	\end{gather*}
	We have by Lemma~\ref{cal_lem} and the divergence-free condition that
	\begin{equation*}
		\bigg| -\int_{\Omega} \nabla \cdot \partial^{\gamma} (v \cdot \nabla)v \partial_d \partial^{\gamma} \tht \,\ud x \bigg| \leq C \| v \|_{H^{m}}^2 \| \theta \|_{H^{m}}.
	\end{equation*}
	From $\partial_dv_d = -\nabla_h \cdot v_h$ and the cancellation property, we deduce $$\left| \int_{\Omega} \nabla (-\Delta) \partial^{\gamma} v_d \cdot \nabla \partial^{\gamma} \tht \,\ud x \right| \leq \| \Delta \partial^{\gamma} \nabla v \|_{L^2} \| \partial^{\gamma} \nabla_h \tht \|_{L^2} \leq \frac 12 \| \Delta \partial^{\gamma} \nabla v \|_{L^2}^2 + \frac 12 \| \partial^{\gamma} \nabla_h \tht \|_{L^2}^2$$
	and
	\begin{equation*}
		\bigg| \int_{\Omega} \big( \nabla\partial^{\gamma} (v \cdot \nabla) v_d \cdot \nabla \partial^{\gamma} \theta + \nabla \partial^{\gamma} (v\cdot \nabla) \theta \cdot \nabla \partial^{\gamma} v_d\big) \, \mathrm{d}x \bigg| \leq C \| v \|_{H^{m}}^2 \| \theta \|_{H^{m}}
	\end{equation*}
	respectively. The above estimates yields
	\begin{gather*}
		\sum_{|\gamma|=m-2} \frac {\mathrm{d}}{\mathrm{d}t} \int_{\Omega} \nabla \partial^{\gamma} v_d \cdot \nabla \partial^{\gamma} \theta \, \mathrm{d}x + \frac 12 \| \nabla v \|_{\dot{H}^m}^2 + \| \nabla v_d \|_{\dot{H}^{m-2}}^2 + C \| v \|_{H^{m}}^2 \| \theta \|_{H^{m}} \geq \frac 12 \| \nabla_h \tht \|_{\dot{H}^{m-2}}^2.
	\end{gather*}
	Similarly, we can repeat the above procedure for the lower order derivatives. Then, \eqref{est_A} is obtained. 
	
	We multiply \eqref{est_A_1} by $2$
	\begin{equation*}
		\begin{gathered}
			\ddt \big( \| v(t) \|_{H^m}^2 + \| \theta(t) \|_{H^m}^2 \big) + 2\| \nabla v \|_{H^m}^2 \\
			 \leq C \| \nabla v \|_{L^{\infty}} (\| v \|_{H^m}^2 + \| \tht \|_{H^{m}}^2) + C \| \nabla v \|_{H^m} \| \nabla_h \tht \|_{H^{m-2}} \| \tht \|_{H^{m}},
		\end{gathered}
	\end{equation*}
	and recall \eqref{est_A}
	\begin{equation*}
		-\frac 32 \| \nabla v \|_{H^m}^2 + \frac 12 \| \nabla_h \theta \|_{H^{m-2}}^2 - \ddt A_{m-1}(t) \leq C E_m(t) \| \nabla v \|_{H^{m-2}}^2.
	\end{equation*}
	Adding these two inequality, we obtain \eqref{eng_est}. This completes the proof.
\end{proof}

\begin{proposition}\label{prop_eng2}
	Let $d \in \mathbb{N}$ with $d \geq 2$ and $m \in \mathbb{N}$ with $m > 1+\frac d2$. Assume that $(v, \theta)$ is a smooth global solution to \eqref{EQ} with $\alpha = 0$. Then there exists a constant $C > 0$ such that
	\begin{equation}\label{eng_est2}
		\begin{gathered}
			\ddt ( E_m(t)^2 - A_m(t) ) + \frac 12 \| v(t) \|_{H^m}^2 + \frac 12 \| \nabla_h \theta(t) \|_{H^{m-1}}^2 \\
			\leq C\| v(t) \|_{H^m} \| \nabla_h \theta(t) \|_{H^{m-1}} \| \tht(t) \|_{H^m} + C\big(\| \tht(t) \|_{H^m}^2 + \| v(t) \|_{H^m}^2 \big) \| \partial_d v_d(t) \|_{L^{\infty}} 
		\end{gathered}
	\end{equation}
	for all $t>0$.
\end{proposition}
\begin{proof}
	From \eqref{EQ}, we can have
	\begin{equation*}
		\begin{gathered}
			\frac 12 \ddt \big( \| v(t) \|_{H^m}^2 + \| \theta(t) \|_{H^m}^2 \big) + \| v \|_{H^m}^2 \\
			\leq - \sum_{1 \leq |\gamma| \leq m} \int_{\Omega} \partial^{\gamma} (v \cdot \nabla) v \partial^{\gamma} v \,\ud x - \sum_{1 \leq |\gamma| \leq m} \int_{\Omega} \partial^\gamma (v \cdot \nabla) \theta \partial^\gamma \theta \, \mathrm{d}x.
		\end{gathered}
	\end{equation*}
	We only estimate $|\gamma| = m$ case because the others can be treated similarly. The first integral on the right-hand side can be estimated by lemma~\ref{cal_lem} and the divergence-free condition that $$\left| \int_{\Omega} \partial^{\gamma} (v \cdot \nabla) v \partial^{\gamma} v \,\ud x \right| \leq C \| v \|_{H^m}^3.$$ To estimate the remainder term, we consider the case $\partial^{\gamma} \neq \partial_d^m$ first. As estimating the first one, we can see $$\left| \int_{\Omega} \partial^\gamma(v \cdot \nabla) \theta \partial^\gamma \theta \,\mathrm{d}x \right| \leq C \| v \|_{H^m} \| R_h \tht \|_{H^m} \| \tht \|_{H^m}.$$
	In the case of $\partial^{\gamma} = \partial_d^m$, we have
	\begin{align*}
		\int_{\Omega} \partial_d^m (v \cdot \nabla) \tht \partial_d^m \tht \,\ud x &= \int_{\Omega} \partial_d^{m-1} (\partial_d v_h \cdot \nabla_h \tht) \partial_d^m \tht \,\ud x + \int_{\Omega} \partial_d^{m-1} (\partial_d v_d \partial_d \tht) \partial_d^m \tht \,\ud x \\
		&\leq C \| v \|_{H^m} \| R_h \tht \|_{H^m} \| \tht \|_{H^m} + C\| \partial_d v_d \|_{L^{\infty}} \| \tht \|_{H^m}^2 + \int_{\Omega} \partial_d^{m-2} (\partial_d^2 v_d \partial_d \tht) \partial_d^m \tht \,\ud x.
	\end{align*}
	We note that 
	\begin{align*}
		\int_{\Omega} \partial_d^{m-2} (\partial_d^2 v_d \partial_d \tht) \partial_d^m \tht \,\ud x &= -\int_{\Omega} \partial_d^{m-2} \nabla_h \cdot (\partial_d v_h \partial_d \tht) \partial_d^m \tht \,\ud x + \int_{\Omega} \partial_d^{m-2} (\partial_d v_h \cdot \nabla_h \partial_d \tht) \partial_d^m \tht \,\ud x\\
		&\leq - \int_{\Omega} \partial_d^{m-2} \nabla_h \cdot (\partial_d v_h \partial_d \tht) \partial_d^m \tht \,\ud x + C \| v \|_{H^m} \| R_h \tht \|_{H^m} \| \tht \|_{H^m}.
	\end{align*}
	By \eqref{sum_est}, we can deduce
	\begin{equation*}
	\begin{aligned}
		\left| \int_{\Omega} \partial_d^{m-2} \nabla_h \cdot (\partial_d v_h \partial_d \tht) \partial_d^m \tht \,\ud x \right| \leq  C \| v \|_{H^m} \| R_h \tht \|_{H^m} \| \tht \|_{H^m},
	\end{aligned}
	\end{equation*}
	thus,
	\begin{align*}
		\left| \int_{\Omega} \partial_d^{m-2} (\partial_d^2 v_d \partial_d \tht) \partial_d^m \tht \,\ud x \right| \leq C \| v \|_{H^m} \| R_h \tht \|_{H^m} \| \tht \|_{H^m}.
	\end{align*}
	Collecting the above estimates, we have
	\begin{equation}\label{est_A3}
		\begin{gathered}
			\frac 12 \ddt \big( \| v(t) \|_{H^m}^2 + \| \theta(t) \|_{H^m}^2 \big) + \| v \|_{H^m}^2 \leq C \| v \|_{H^m}^3 + C \| v \|_{H^m} \| R_h \tht \|_{H^m} \| \tht \|_{H^m} + C\| \partial_d v_d \|_{L^{\infty}} \| \tht \|_{H^m}^2.
		\end{gathered}
	\end{equation}

	As estimating \eqref{est_A}, we can show 
	\begin{equation}\label{est_A2}
		\frac 32 \| v \|_{H^m}^2 \geq \frac 12 \| \nabla_h \theta \|_{H^{m-1}}^2 - \ddt A_{m}(t) - C E_m(t) \| v \|_{H^{m}}^2.
	\end{equation}
	We only consider the highest derivative case. We can see from the $v_d$ equation in \eqref{EQ} $$\partial_t (-\Delta)v_d + (-\Delta) v_d  -\Delta (v\cdot \nabla v_d) + \partial_d \nabla \cdot (v\cdot \nabla) v  = -\Delta_h \theta.$$ Thus, we have for $|\gamma| = m-1$ that
	\begin{gather*}
		\int_{\Omega} \partial_t \nabla \partial^{\gamma} v_d \cdot \nabla \partial^{\gamma} \tht \,\ud x + \int_{\Omega} \nabla \partial^{\gamma} v_d \cdot \nabla \partial^{\gamma}  \tht \,\ud x + \int_{\Omega} \nabla \partial^{\gamma} (v \cdot \nabla v_d) \cdot \nabla \partial^{\gamma} \tht \,\ud x \\
	- \int_{\Omega} \nabla \cdot \partial^{\gamma} (v \cdot \nabla)v \partial_d \partial^{\gamma} \tht \,\ud x = \int_{\Omega} | \nabla_h \partial^{\gamma} \tht|^2 \,\ud x. 
	\end{gather*}
	From the $\tht$ equation in \eqref{EQ}, it follows
	\begin{equation*}
		\int_{\Omega} \partial_t \nabla \partial^{\gamma} \theta \cdot \nabla \partial^{\gamma} v_d \,\mathrm{d}x + \int_{\Omega} \nabla \partial^{\gamma} (v\cdot \nabla \theta) \cdot \nabla \partial^{\gamma} v_d \,\mathrm{d}x + \int_{\Omega} |\nabla \partial^{\gamma} v_d|^2 \,\ud x = 0.
	\end{equation*}
	Combining the two above gives
	\begin{gather*}
		\frac {\mathrm{d}}{\mathrm{d}t} \int_{\Omega} \nabla \partial^{\gamma} v_d \cdot \nabla \partial^{\gamma} \theta \, \mathrm{d}x + \int_{\Omega} \nabla \partial^{\gamma} v_d \cdot \nabla \partial^{\gamma} \tht \,\ud x - \int_{\Omega} \nabla \cdot \partial^{\gamma} (v \cdot \nabla)v \partial_d \partial^{\gamma} \tht \,\ud x \\
		+ \int_{\Omega} \big( \nabla\partial^{\gamma} (v \cdot \nabla v_d) \cdot \nabla \partial^{\gamma} \theta + \nabla \partial^{\gamma} (v\cdot \nabla \theta) \cdot \nabla \partial^{\gamma} v_d\big) \, \mathrm{d}x + \int_{\Omega} |\nabla \partial^{\gamma} v_d|^2 \,\ud x = \int_{\Omega} | \nabla_h \partial^{\gamma} \tht|^2 \,\ud x.
	\end{gather*}
	The divergence-free condition and Lemma~\ref{cal_lem} imply
	\begin{equation*}
		\bigg| \int_{\Omega} \nabla \cdot \partial^{\gamma} (v \cdot \nabla)v \partial_d \partial^{\gamma} \tht \,\ud x \bigg| \leq C \| v \|_{H^{m}}^2 \| \theta \|_{H^{m}}.
	\end{equation*}
	We also have with the divergence-free condition that $$\left| \int_{\Omega} \nabla \partial^{\gamma} v_d \cdot \nabla \partial^{\gamma} \tht \,\ud x \right| \leq \| \partial^{\gamma} \nabla v \|_{L^2} \| \partial^{\gamma} \nabla_h \tht \|_{L^2} \leq \frac 12 \| \partial^{\gamma} \nabla v \|_{L^2}^2 + \frac 12 \| \partial^{\gamma} \nabla_h \tht \|_{L^2}^2.$$
	The cancellation property yields
	\begin{equation*}
		\bigg| \int_{\Omega} \big( \nabla\partial^{\gamma} (v \cdot \nabla v_d) \cdot \nabla \partial^{\gamma} \theta + \nabla \partial^{\gamma} (v\cdot \nabla \theta) \cdot \nabla \partial^{\gamma} v_d\big) \, \mathrm{d}x \bigg| \leq C \| v \|_{H^{m}}^2 \| \theta \|_{H^{m}}.
	\end{equation*}
	Therefore, we deduce that
	\begin{gather*}
		\sum_{|\gamma|=m-1} \frac {\mathrm{d}}{\mathrm{d}t} \int_{\Omega} \nabla \partial^{\gamma} v_d \cdot \nabla \partial^{\gamma} \theta \, \mathrm{d}x + \frac 32 \| v \|_{\dot{H}^m}^2 + C \| v \|_{H^{m}}^2 \| \theta \|_{H^{m}} \geq \frac 12 \| \nabla_h \tht \|_{\dot{H}^{m-1}}^2,
	\end{gather*}
	which implies \eqref{est_A2}.
	
	Multiplying \eqref{est_A3} by 2 and adding \eqref{est_A2} gives \eqref{eng_est2}. This completes the proof.
\end{proof}

\section{Global-in-time existence}\label{sec:global}
In this section, we prove the global existence part of Theorem \ref{thm1} and \ref{thm2}. It remains to estimate the key quantities in Proposition~\ref{prop_eng} and Proposition~\ref{prop_eng2}, namely, 
\begin{equation}\label{key_quan}
\int \| \nabla v(t) \|_{L^{\infty}} \,\ud t \qquad \mbox{and} \qquad \int \| \partial_d v_d(t) \|_{L^{\infty}} \,\ud t
\end{equation} respectively. For this purpose, we recall the notations introduced 
in Section~\ref{sec_spec}. From now on, we use the notations 
\begin{equation*}
	R_h f = \nabla_h \Lambda^{-1} f = -\sum_{|\tilde{n}| \neq 0} \frac {i\tilde{n}}{|\eta|} \mathscr{F}_bf(\eta) \mathscr{B}_\eta(x), \qquad R_h g = \nabla_h \Lambda^{-1} g = -\sum_{|\tilde{n}| \neq 0} \frac {i\tilde{n}}{|\eta|} \mathscr{F}_cg(\eta) \mathscr{C}_{\eta}(x),
\end{equation*}
for $f \in X^m(\Omega)$ and $g \in Y^m(\Omega)$.

\subsection{Proof of Theorem~\ref{thm1}:Global-in-time existence part}
Here, we fix $\alpha = 1$. We show the two propositions that provide proper upper-bound of the key quantity. Then combining with Proposition~\ref{prop_eng}, we finish the proof.

\begin{proposition}\label{prop_v}
	Let $d \in \mathbb{N}$ with $d \geq 2$ and $m \in \mathbb{N}$ satisfying $m > 1+\frac d2$. Assume that $(v, \theta)$ be a smooth global solution to \eqref{EQ} with $\alpha = 1$. Then there exists a constant $C > 0$ such that
	\begin{equation}\label{v_L1_est}
		\begin{gathered}
			\sum_{\eta \in I} \int_0^T |\eta| |\mathscr{F}_cv_h(t) | \,\mathrm{d}t + \sum_{\eta \in J} \int_0^T |\eta|^2 |\mathscr{F}_bv_d(t) | \,\mathrm{d}t \leq C \| v_0 \|_{H^m} \\
			+  C \sup_{t \in [0,T]} \| v(t) \|_{H^m} \left( \sum_{\eta \in I} \int_0^T |\eta| |\mathscr{F}_cv_h(t) | \,\mathrm{d}t + \sum_{\eta \in J} \int_0^T |\eta|^2 |\mathscr{F}_bv_d(t) | \,\mathrm{d}t \right) + \sum_{\eta \in J} \int_0^T \frac {|\tilde{n}|^2}{|\eta|^2} |\mathscr{F}_b \theta(t)| \,\mathrm{d}t
		\end{gathered}
	\end{equation}
	for all $T > 0$.
\end{proposition}
\begin{proof}
	We first note that the divergence-free condition implies
	\begin{equation}\label{vd_v}
		|\eta||\mathscr{F}_b v_d(\eta)| \leq C |\tilde{n}| |\mathscr{F}_cv_h(\eta)| + C |\tilde{n}| |\mathscr{F}_bv_d(\eta)|.
	\end{equation}
	Thus, it suffices to show
	\begin{gather*}
		\sum_{\eta \in J} \int_0^T |\tilde{n}||\eta| |\mathscr{F}_cv_h(t) | \,\ud t + \sum_{\eta \in J} \int_0^T |\tilde{n}||\eta| |\mathscr{F}_bv_d(\eta)| \,\mathrm{d}t \leq C \| v_0 \|_{H^m} \\
		+  C \sup_{t \in [0,T]} \| v(t) \|_{H^m} \left( \sum_{\eta \in I} \int_0^T |\eta| |\mathscr{F}_cv_h(t) | \,\mathrm{d}t + \sum_{\eta \in J} \int_0^T |\eta|^2 |\mathscr{F}_bv_d(t) | \,\mathrm{d}t \right)  + \sum_{\eta \in J} \int_0^T \frac {|\tilde{n}|^2}{|\eta|^2} |\mathscr{F}_b \theta(t)| \,\mathrm{d}t.
	\end{gather*}
	We only estimate the first term on the left-hand side because the other can be treated similarly. Applying Duhamel's principle to \eqref{df_vh}, we can have
	\begin{equation*}
		\sum_{\eta \in J} \int_0^T |\tilde{n}||\eta| |\mathscr{F}_cv_h(t) | \,\mathrm{d}t \leq I_1 + I_2 + I_3 + I_4,
	\end{equation*}
	where
	\begin{align*}
		I_1 &:= \sum_{\eta \in J}  \int_0^T |\tilde{n}||\eta| e^{-|\eta|^2t} |\mathscr{F} v_0| \,\mathrm{d}t, \\
		I_2 &:= \sum_{\eta \in J} \int_0^T \int_0^t |\tilde{n}||\eta|e^{-|\eta|^2(t-\tau)} |\mathscr{F}_c[(v \cdot \nabla)v_h](\tau)| \,\ud \tau \mathrm{d}t, \\
		I_3 &:= \sum_{\eta \in J} \int_0^T \int_0^t |\tilde{n}||\eta|e^{-|\eta|^2(t-\tau)} |\mathscr{F}_b[(v \cdot \nabla)v_d](\tau)| \,\ud \tau \mathrm{d}t, \\
		I_4 &:= \sum_{\eta \in J} \int_0^T \int_0^t |\tilde{n}||\eta| e^{-|\eta|^2(t - \tau)} \frac {|\tilde{n}|}{|\eta|} |\mathscr{F}_b \theta(\tau)| \,\mathrm{d}\tau \mathrm{d}t.
	\end{align*}
	We have used that $|\mathscr{F} \mathbb{P} f| \leq |\mathscr{F} f|$ for $I_2$ and $I_3$. It is clear that
	\begin{equation*}
		I_1 \leq \sum_{\eta \in J} \int_0^T |\eta|^2 e^{-|\eta|^2t} |\mathscr{F}v_0| \,\mathrm{d}t \leq \sum_{\eta \in I} |\mathscr{F}v_0| \leq C\| v_0 \|_{H^m}.
	\end{equation*}
	We can see by Fubini's theorem that
	\begin{equation*}
		I_4 \leq \sum_{\eta \in J} \int_0^T \frac {|\tilde{n}|^2}{|\eta|^2} |\mathscr{F}_b \theta(t)| \, \mathrm{d}t.
	\end{equation*}
	Similarly,
	\begin{gather*}
		I_2 \leq C \sum_{\eta \in I} \int_0^T |\mathscr{F}_c[(v \cdot \nabla)v_h](t)| \,\ud t.
	\end{gather*}
	Using $(v \cdot \nabla)v_h = (v_h \cdot \nabla_h)v_h + v_d \partial_d v_h$ and Proposition~\ref{cor_conv}, we have
	\begin{gather*}
		I_2 \leq \int_0^T \left( \sum_{\eta \in I} |\mathscr{F}_c v_h(t)| \right) \left( \sum_{\eta \in I} | \mathscr{F}_c \nabla_h v_h(t)| \right) \,\ud t + \int_0^T \left( \sum_{\eta \in J} |\mathscr{F}_b v_d(t)| \right) \left( \sum_{\eta \in J} | \mathscr{F}_b \partial_d v_h(t)| \right) \,\ud t \\
		\leq C \sup_{t \in [0,T]} \| v(t) \|_{H^m} \left( \sum_{\eta \in I} \int_0^T |\eta| |\mathscr{F}_cv_h(t) | \,\mathrm{d}t + \sum_{\eta \in J} \int_0^T |\eta|^2 |\mathscr{F}_bv_d(t) | \,\mathrm{d}t \right).
	\end{gather*}
	In a similar way with the above, we can have the same upper-bound for $I_3$. Collecting the estimates for $I_1$, $I_2$, $I_3$, and $I_4$, we obtain the claim. 
	This completes the proof.
\end{proof}

\begin{proposition}\label{prop_tht}
	Let $d \in \mathbb{N}$ with $d \geq 2$ and $m \in \mathbb{N}$ satisfying $m > 3+\frac d2$. Assume that $(v, \theta)$ be a smooth global solution to \eqref{EQ} with $\alpha = 1$, and $\int_{\Omega} v_0 \,\ud x$ be satisfied. Then there exists a constant $C > 0$ such that
	\begin{equation}\label{tht_L1_est}
		\begin{gathered}
			\sum_{\eta \in J} \int_0^T \frac {|\tilde{n}|^2}{|\eta|^2} |\mathscr{F}_b \theta(t)| \,\mathrm{d}t \leq C \| (v_0,\tht_0) \|_{H^{m}} + C\int_0^T \| \nabla v(t) \|_{H^{m}}^2 \,\ud t + C\int_0^T \| \nabla_h \tht(t) \|_{H^{m-2}}^2 \,\ud t \\
			+ C \sup_{t \in [0,T]} (\| v(t) \|_{H^m} + \| \tht (t) \|_{H^m}) \left( \int_0^T \sum_{\eta \in I} |\eta| |\mathscr{F}_cv_h(t) | \,\mathrm{d}t + \int_0^T \sum_{\eta \in J} |\eta|^2 |\mathscr{F}_bv_d(t) | \,\mathrm{d}t \right).
		\end{gathered}
	\end{equation}
	for all $T>0$.
\end{proposition}
\begin{proof}
	We recall \eqref{df_tht_1} and have
	\begin{equation*}
		\sum_{\eta \in J}\int_0^T \frac {|\tilde{n}|^2}{|\eta|^2} |\mathscr{F}_b \theta(t)| \,\mathrm{d}t \leq I_5 + I_6 + I_7 + I_8,
	\end{equation*}
	where
	\begin{align*}
		I_5 &:= \sum_{\eta \in J} \int_0^T \frac {|\tilde{n}|^2}{|\eta|^2} (e^{-\lambda_- t} - e^{-\lambda_+ t}) |\langle \mathscr{F}_b\bold{u}_0,\bold{a}_- \rangle| |\langle \bold{b}_-,e_{2} \rangle| \,\mathrm{d}t, \\
		I_6 &:= \sum_{\eta \in J} \int_0^T e^{-\lambda_+ t} |\mathscr{F}_b \tht_0|  \, \mathrm{d}t, \\
		I_7 &:= \sum_{\eta \in J} \int_0^T \int_0^t \frac {|\tilde{n}|^2}{|\eta|^2} (e^{-\lambda_- (t - \tau)} - e^{-\lambda_+ (t - \tau)}) |\langle N(v,\tht)(\tau),\bold{a}_- \rangle| |\langle \bold{b}_-,e_{2} \rangle| \,\mathrm{d}\tau\mathrm{d}t,\\
		I_8 &:= \sum_{\eta \in J} \int_0^T \int_0^t e^{-\lambda_+ (t - \tau)} |\mathscr{F}_b [(v \cdot \nabla)\tht](\tau)| \,\mathrm{d}\tau\mathrm{d}t.
	\end{align*}
	We estimate $I_6$ and $I_8$ first. By \eqref{ker_est} we have
	\begin{equation*}
		I_6 \leq \sum_{\eta \in J} \int_0^T |\eta|^2 e^{-|\eta|^2 \frac t2} |\mathscr{F}_b \tht_0|  \, \mathrm{d}t \leq C\| \tht_0 \|_{H^m}.
	\end{equation*}
	With Fubini's theorem, we also have
	\begin{align*}
		I_8 &\leq C \sum_{\eta \in J} \int_0^T |\mathscr{F}_b [(v \cdot \nabla)\tht](t)| \,\mathrm{d}t.
	\end{align*}
	Due to \eqref{avg_vd} and \eqref{avg_vh}, Poincar\'{e} inequality implies
	\begin{align*}
		I_8 &\leq C\int_0^T \left( \sum_{\eta \in I} |\mathscr{F} (v-\int_{\Omega} v \,\ud x )(t)| \right) \left( \sum_{\eta \in J} |\eta| |\mathscr{F}_b \tht(t)| \right) \,\ud t \\
		&\leq C \sup_{t \in [0,T]} \| \theta(t) \|_{H^m} \left( \int_0^T \sum_{\eta \in I} |\eta| |\mathscr{F}_cv_h(t) | \,\mathrm{d}t + \int_0^T \sum_{\eta \in J} |\eta|^2 |\mathscr{F}_bv_d(t) | \,\mathrm{d}t \right).
	\end{align*}
	Now, we estimate $I_5$ and $I_7$. To apply \eqref{sing_est}, we consider $\eta \in D_2$ and $\eta \in D_3$ separately. Note that $D_1 = \emptyset$ when $\alpha = 1$. In the former case, as the previous estimates, we have
	\begin{align*}
		\sum_{\eta \in D_2} \int_0^T \frac {|\tilde{n}|^2}{|\eta|^2} (e^{-\lambda_- t} - e^{-\lambda_+ t}) |\langle \mathscr{F}_b\bold{u}_0,\bold{a}_- \rangle| |\langle \bold{b}_-,e_{2} \rangle| \,\mathrm{d}t &\leq C\sum_{\eta \in J} \int_0^T e^{-|\eta|^2 \frac t4} |\mathscr{F}_b \mathbf{u}_0|  \, \mathrm{d}t \leq C\| \mathbf{u}_0 \|_{H^m}
	\end{align*}
	and
	\begin{align*}
		&\sum_{\eta \in D_2}\int_0^T \int_0^t \frac {|\tilde{n}|^2}{|\eta|^2} (e^{-\lambda_- (t - \tau)} - e^{-\lambda_+ (t - \tau)}) |\langle N(v,\tht)(\tau),\bold{a}_- \rangle| |\langle \bold{b}_-,e_{2} \rangle| \,\mathrm{d}\tau\mathrm{d}t \\
		&\hphantom{\qquad\qquad}\leq C\sum_{\eta \in J} \int_0^T |N(v,\tht)(t)|  \, \mathrm{d}t \\
		&\hphantom{\qquad\qquad}\leq C\sum_{\eta \in J} \int_0^T (|\mathscr{F}[(v \cdot \nabla)v](t)| + |\mathscr{F}_b[((v-\int_{\Omega} v \,\ud x) \cdot \nabla)\tht](t)|)  \, \mathrm{d}t \\
		&\hphantom{\qquad\qquad}\leq C\int_0^T \left\{ \left(\sum_{\eta \in I} |\mathscr{F}v(t)| \right) \left( \sum_{\eta \in I} |\eta| |\mathscr{F}v(t)| \right) + \left( \sum_{\eta \in I} |\eta| |\mathscr{F}v(t)| \right) \left( \sum_{\eta \in J} |\eta| |\mathscr{F}_b \tht(t)|\right) \right\}  \, \mathrm{d}t \\
		&\hphantom{\qquad\qquad}\leq C \sup_{t \in [0,T]} (\| v(t) \|_{H^m} + \| \tht(t) \|_{H^m}) \left( \int_0^T \sum_{\eta \in I} |\eta| |\mathscr{F}_cv_h(t) | \,\mathrm{d}t + \int_0^T \sum_{\eta \in J} |\eta|^2 |\mathscr{F}_bv_d(t) | \,\mathrm{d}t \right).
	\end{align*}
	In the latter case,
	\begin{align*}
		\sum_{\eta \in D_3} \int_0^T \frac {|\tilde{n}|^2}{|\eta|^2} (e^{-\lambda_- t} - e^{-\lambda_+ t}) |\langle \mathscr{F}_b\bold{u}_0,\bold{a}_- \rangle| |\langle \bold{b}_-,e_{2} \rangle| \,\mathrm{d}t &\leq C\sum_{\eta \in J} \int_0^T \frac {|\tilde{n}|^2}{|\eta|^2} e^{-\frac {|\tilde{n}|^2}{|\eta|^4} t} |\mathscr{F}_b \bold{u}_0|  \, \mathrm{d}t \\
		&\leq C\sum_{\eta \in J} |\eta|^2 |\mathscr{F}_b \bold{u}_0| \\
		&\leq C \| \bfu_0 \|_{H^m}.
	\end{align*}
	On the other hand, we similarly have
	\begin{align*}
		&\sum_{\eta \in D_3}\int_0^T \int_0^t \frac {|\tilde{n}|^2}{|\eta|^2} (e^{-\lambda_- (t - \tau)} - e^{-\lambda_+ (t - \tau)}) |\langle N(v,\tht)(\tau),\bold{a}_- \rangle| |\langle \bold{b}_-,e_{2} \rangle| \,\mathrm{d}\tau\mathrm{d}t \\
		&\hphantom{\qquad\qquad}\leq \sum_{\eta \in J} \int_0^T |\eta|^2 |N(v,\tht)(t)|  \, \mathrm{d}t \\
		&\hphantom{\qquad\qquad}\leq \sum_{\eta \in J} \int_0^T |\eta|^2 |\mathscr{F}(v \cdot \nabla)v(t)| \, \mathrm{d}t + \sum_{\eta \in J} \int_0^T |\eta|^2 |\mathscr{F}_b(v \cdot \nabla)\tht(t)|  \, \mathrm{d}t.
	\end{align*}
	Let $\tilde{n}' + \tilde{n}'' = \tilde{n}$ and $\tilde{q}' + \tilde{q}'' = \tilde{q}$. Then, it holds $$|\eta|^2 = |\tilde{n}|^2 + |\tilde{q}|^2 \leq 2|\tilde{n}'|^2 + 2|\tilde{q}'|^2 + 2|\tilde{n}''|^2 + 2|\tilde{q}''|^2 = 2|\eta'|^2 + 2|\eta''|^2.$$ Similarly, for $\tilde{n}' + \tilde{n}'' = \tilde{n}$ and $|\tilde{q}' - \tilde{q}''| = \tilde{q}$, we can see $$|\eta|^2 \leq 2|\eta'|^2 + 2|\eta''|^2.$$ They imply
	\begin{gather*}
		\sum_{\eta \in J} \int_0^T |\eta|^2 (|\mathscr{F}(v \cdot \nabla)v(t)| \, \mathrm{d}t \\
		\leq C\int_0^T \left\{ \left(\sum_{\eta \in I} |\eta|^2 |\mathscr{F}v(t)| \right) \left( \sum_{\eta \in I} |\eta| |\mathscr{F}v(t)| \right) + \left( \sum_{\eta \in I} |\mathscr{F}v(t)| \right) \left( \sum_{\eta \in I} |\eta|^3 |\mathscr{F}v(t)| \right) \right\}  \, \mathrm{d}t \\
		\leq C \int_0^T \| \nabla v \|_{H^m}^2 \,\ud t.
	\end{gather*}	
	On the other hand, with $\mathscr{F}_b(v \cdot \nabla)\tht = \mathscr{F}_b(v_h \cdot \nabla_h)\tht + \mathscr{F}_b[v_d \partial_d \tht]$ it follows for $m > 3+d/2$
	\begin{gather*}
		\sum_{\eta \in J} \int_0^T |\eta|^2 |\mathscr{F}_b(v \cdot \nabla)\tht(t)|  \, \mathrm{d}t \\
		\leq \sum_{\eta \in J} \int_0^T |\eta|^2 |\mathscr{F}_b(v_h \cdot \nabla_h)\tht(t)|  \, \mathrm{d}t + \sum_{\eta \in J} \int_0^T |\eta|^2 |\mathscr{F}_b[v_d \partial_d \tht](t)|  \, \mathrm{d}t \\
		\leq C\int_0^T \left\{ \left( \sum_{\eta \in I} |\eta|^2 |\mathscr{F}_cv_h(t)| \right) \left( \sum_{\eta \in J} |\tilde{n}| |\mathscr{F}_b\tht(t)| \right) + \left( \sum_{\eta \in I} |\mathscr{F}_cv_h(t)| \right) \left( \sum_{\eta \in J} |\eta|^2 |\tilde{n}| |\mathscr{F}_b \tht(t)|\right) \right\}  \, \mathrm{d}t \\
		+ C\int_0^T \left\{ \left(\sum_{\eta \in J} |\eta|^2 |\mathscr{F}_bv_d(t)| \right) \left( \sum_{\eta \in J} |\eta| |\mathscr{F}_b\tht (t)| \right) + \left( \sum_{\eta \in J} |\mathscr{F}_bv_d(t)| \right) \left( \sum _{\eta \in J} |\eta|^3 |\mathscr{F}_b\tht(t)| \right) \right\}  \, \mathrm{d}t \\
		\leq C\int_0^T \left( \sum_{\eta \in I} |\eta|^2 |\mathscr{F}_cv_h(t)| \right) \left( \sum_{\eta \in J} |\tilde{n}| |\mathscr{F}_b\tht(t)| \right) \, \mathrm{d}t \\
		+C \sup_{t \in [0,T]} \| \tht(t) \|_{H^m} \left( \int_0^T \sum_{\eta \in I} |\eta| |\mathscr{F}_cv_h(t) | \,\mathrm{d}t + \int_0^T \sum_{\eta \in J} |\eta|^2 |\mathscr{F}_bv_d(t) | \,\mathrm{d}t \right).
	\end{gather*}
	Since $$\sum_{\eta \in J} |\tilde{n}| |\mathscr{F}_b\tht(t)| = \sum_{\eta \in J} |\mathscr{F}_b \nabla_h \tht(t)| \leq C \| \nabla_h \tht(t) \|_{H^{m-2}}$$ and $$\sum_{\eta \in I} |\eta|^2 |\mathscr{F}_cv_h(t)| \leq C \| \nabla v (t) \|_{H^{m}},
	$$ we have
	\begin{gather*}
		C\int_0^T \left( \sum_{\eta \in I} |\eta|^2 |\mathscr{F}_cv_h(t)| \right) \left( \sum_{\eta \in J} |\tilde{n}| |\mathscr{F}_b\tht(t)| \right) \, \mathrm{d}t \leq C\int_0^T \| \nabla v(t) \|_{H^{m}}^2 \,\ud t + C\int_0^T \| \nabla_h \tht(t) \|_{H^{m-2}}^2 \,\ud t.
	\end{gather*}
	Collecting the estimates for $I_5$, $I_6$, $I_7$, and $I_8$, we complete the proof.
\end{proof}

Now, we are ready to prove the global existence part of Theorem~\ref{thm1}. Let $T^*>0$ and $(v,\tht)$ be the maximal time of existence and the local solution given in Proposition~\ref{prop_loc} respectively. We define
\begin{equation}\label{Def_B}
	B_m(T) := \left( \sup_{t \in [0,T]} E_m(t)^2 + \int_0^T \| \Lambda^{\alpha} v(t) \|_{H^m}^2 \,\mathrm{d}t+ \int_0^T \| \nabla_h \theta(t) \|_{H^{m-1-\alpha}}^2 \,\mathrm{d}t \right)^{\frac 12}.
\end{equation}
Then, from \eqref{v_L1_est} and \eqref{tht_L1_est}, we have 
\begin{gather*}
	\sum_{\eta \in I} \int_0^T |\eta| |\mathscr{F}_cv_h(t) | \,\mathrm{d}t + \sum_{\eta \in J} \int_0^T |\eta|^2 |\mathscr{F}_bv_d(t) | \,\mathrm{d}t \leq C \| (v_0,\tht_0) \|_{H^m} \\
	+ C_1B_m(T) \left( \sum_{\eta \in I} \int_0^T |\eta| |\mathscr{F}_cv_h(t) | \,\mathrm{d}t + \sum_{\eta \in J} \int_0^T |\eta|^2 |\mathscr{F}_bv_d(t) | \,\mathrm{d}t \right) + C_1B_m(T)^2
\end{gather*}
for some $C_1>0$. For a while, we assume that $C_1B_m(T) \leq \frac 12$ for all $T \in (0,T^*)$. Then, we have
\begin{gather*}
	\sum_{\eta \in I} \int_0^T |\eta| |\mathscr{F}_cv_h(t) | \,\mathrm{d}t + \sum_{\eta \in J} \int_0^T |\eta|^2 |\mathscr{F}_bv_d(t) | \,\mathrm{d}t \leq C \| (v_0,\tht_0) \|_{H^m} + B_m(T).
\end{gather*}
On the other hand, we recall \eqref{eng_est} and integrate it on the interval $[0,T]$. Then by \eqref{A}, we have
\begin{equation*}
	\frac 12 B_m(T)^2 \leq \frac 32 \| (v_0,\tht_0) \|_{H^m}^2 + CB_m(T)^3 + CB_m(T)^2 \int_0^T \| \nabla v(t) \|_{L^{\infty}} \,\ud t.
\end{equation*}
Since $$\| \nabla v \|_{L^{\infty}} \leq \sum_{\eta \in I} |\eta||\mathscr{F}_c v_h| + \sum_{\eta \in J} |\eta|^2|\mathscr{F}_bv_d|,$$ it holds
\begin{align*}
	\frac 12 B_m(T)^2 &\leq \frac 32\| (v_0,\tht_0) \|_{H^m}^2 + CB_m(T)^3 + CB_m(T)^2 \left( \| (v_0,\tht_0) \|_{H^m} + B_m(T) \right) \\
	&\leq \frac 32 \| (v_0,\tht_0) \|_{H^m}^2 + C_2 \| (v_0,\tht_0) \|_{H^m} B_m(T)^2 + C_2B_m(T)^3
\end{align*}
for some $C_2>0$. If we assume $C_2 \|  (v_0,\tht_0) \|_{H^m} \leq \frac 1{16}$ and $C_2 B_m(T) \leq \frac 1{16}$, then 
\begin{equation}\label{Bm_bdd}
	B_m(T)^2 \leq 4 \|  (v_0,\tht_0) \|_{H^m}^2 \leq 4\delta^2.
\end{equation}
Here, we take $\delta>0$ such that $C_1 (2\delta) < \frac 12$ and $C_2 (2\delta) < \frac 1{16}$. By the above estimates, we can deduce that \eqref{Bm_bdd} holds for all $T \in (0,T^*)$, hence, $T^* = \infty$. Thus, \eqref{sol_bdd_1} is obtained. This completes the proof.

\subsection{Proof of Theorem~\ref{thm2}:Global-in-time existence part}
In this subsection, we fix $\alpha = 0$. We only provide two propositions counterparts of Proposition~\ref{prop_v} and \ref{prop_tht}, because the rest of the proof is similar with that of theorem~\ref{thm1}.

\begin{proposition}
	Let $d \in \mathbb{N}$ with $d \geq 2$ and $m \in \mathbb{N}$ satisfying $m > 2+\frac d2$. Assume that $(v, \theta)$ be a smooth global solution to \eqref{EQ} with $\alpha = 0$. Then there exists a constant $C > 0$ such that
	\begin{equation}\label{v_L1_est2}
		\begin{gathered}
			\sum_{\eta \in J} \int_0^T |\eta| |\mathscr{F}_b v_d(t) | \,\mathrm{d}t \leq C \| v_0 \|_{H^m} +  C \int_0^T \| v(t) \|_{H^m}^2 \,\ud t + \sum_{\eta \in J} \int_0^T \frac {|\tilde{n}|^2}{|\eta|} |\mathscr{F}_b \theta(t)| \,\mathrm{d}t
		\end{gathered}
	\end{equation}
	for all $T > 0$.
\end{proposition}
\begin{proof}
	From \eqref{df_vd} and Duhamel's principle, we can have
	\begin{equation*}
		\sum_{\eta \in J} \int_0^T |\eta| |\mathscr{F}_bv_d(t) | \,\mathrm{d}t \leq J_1 + J_2 + J_3 + J_4,
	\end{equation*}
	where
	\begin{align*}
		J_1 &:= \sum_{\eta \in J}  \int_0^T |\eta| e^{-t} |\mathscr{F}_bv_0| \,\mathrm{d}t, \\
		J_2 &:= \sum_{\eta \in J} \int_0^T \int_0^t |\eta|e^{-(t-\tau)} |\mathscr{F}_b[(v \cdot \nabla)v_d](\tau)| \,\ud \tau \mathrm{d}t, \\
		J_3 &:= \sum_{\eta \in J} \int_0^T \int_0^t |\eta|e^{-(t-\tau)} |\mathscr{F}_c[(v \cdot \nabla)v_h](\tau)| \,\mathrm{d}\tau \mathrm{d}t, \\
		J_4 &:= \sum_{\eta \in J} \int_0^T \int_0^t e^{-(t - \tau)} \frac {|\tilde{n}|^2}{|\eta|} |\mathscr{F}_b \theta(\tau)| \,\mathrm{d}\tau \mathrm{d}t.
	\end{align*}
	We can easily show $$J_1 \leq \sum_{\eta \in J} |\eta| |\mathscr{F}_bv_0| \leq C \| v_0 \|_{H^m}$$ and $$J_4 \leq \sum_{\eta \in J} \int_0^T \frac {|\tilde{n}|^2}{|\eta|} |\mathscr{F}_b \theta(t)| \, \mathrm{d}t.$$ Fubini's theorem and Propostion~\ref{cor_conv} gives $$J_2 \leq \sum_{\eta \in J} \int_0^T |\eta| |\mathscr{F}_b[(v \cdot \nabla)v_d](t)| \,\mathrm{d}t \leq C\int_0^T \| v(t) \|_{H^m}^2 \,\ud t.$$ Similarly, we can estimate $J_3$ and have $$J_3 \leq C\int_0^T \| v(t) \|_{H^m}^2 \,\ud t.$$ From the estimates for $J_1$, $J_2$, $J_3$, and $J_4$, we deduce \eqref{v_L1_est2}. This completes the proof.
\end{proof}

\begin{proposition}
	Let $d \in \mathbb{N}$ with $d \geq 2$ and $m \in \mathbb{N}$ satisfying $m > 2+\frac d2$. Assume that $(v, \theta)$ be a smooth global solution to \eqref{EQ} with $\alpha = 1$. Then there exists a constant $C > 0$ such that
	\begin{equation}\label{tht_L1_est2}
		\begin{gathered}
			\sum_{\eta \in J} \int_0^T \frac {|\tilde{n}|^2}{|\eta|} |\mathscr{F}_b \theta(t)| \,\mathrm{d}t \leq C \| \bold{u}_0 \|_{H^{m}} +  C \int_0^T \| v(t) \|_{H^m}^2 \,\ud t + C \int_0^T \| \nabla_h \theta(t) \|_{H^{m-1}}^2 \,\mathrm{d}t \\
			+ C \sup_{t \in [0,T]} \| \tht(t) \|_{H^m} \sum_{\eta \in J} \int_0^T |\eta| |\mathscr{F}_bv_d(t) | \,\mathrm{d}t .
		\end{gathered}
	\end{equation}
	for all $T>0$.
\end{proposition}
\begin{proof}
	We recall \eqref{df_tht_1} and have
	\begin{equation*}
		\sum_{\eta \in J}\int_0^T \frac {|\tilde{n}|^2}{|\eta|} |\mathscr{F}_b \theta(t)| \,\mathrm{d}t \leq J_5 + J_6 + J_7 + J_8,
	\end{equation*}
	where
	\begin{align*}
		J_5 &:= \sum_{\eta \in J} \int_0^T \frac {|\tilde{n}|^2}{|\eta|} (e^{-\lambda_- t} - e^{-\lambda_+ t}) |\langle \mathscr{F}_b\bold{u}_0,\bold{a}_- \rangle| |\langle \bold{b}_-,e_{2} \rangle| \,\mathrm{d}t, \\
		J_6 &:= \sum_{\eta \in J} \int_0^T |\tilde{n}| e^{-\lambda_+ t} |\mathscr{F}_b \tht_0|  \, \mathrm{d}t, \\
		J_7 &:= \sum_{\eta \in J} \int_0^T \int_0^t \frac {|\tilde{n}|^2}{|\eta|} (e^{-\lambda_- (t - \tau)} - e^{-\lambda_+ (t - \tau)}) |\langle N(v,\tht)(\tau),\bold{a}_- \rangle| |\langle \bold{b}_-,e_{2} \rangle| \,\mathrm{d}\tau\mathrm{d}t,\\
		J_8 &:= \sum_{\eta \in J} \int_0^T \int_0^t |\tilde{n}| e^{-\lambda_+ (t - \tau)} |\mathscr{F}_b [(v \cdot \nabla)\tht](\tau)| \,\mathrm{d}\tau\mathrm{d}t.
	\end{align*}
	It is clear by \eqref{ker_est}
	\begin{equation*}
		J_6 \leq \sum_{\eta \in J} \int_0^T e^{-\frac t2} |\mathscr{F}_b \nabla_h \tht_0|  \, \mathrm{d}t \leq C\| \tht_0 \|_{H^m}.
	\end{equation*}
	Similarly, we have with Proposition~\ref{cor_conv} that
	\begin{align*}
		J_8 &\leq C \sum_{\eta \in J} \int_0^T |\mathscr{F}_b [\nabla_h (v \cdot \nabla)\tht]| \,\mathrm{d}t \\
		&\leq C \int_0^T \left\{ \left( \sum_{\eta \in I} |\eta| |\mathscr{F}_c v_h| \right) \left( \sum_{\eta \in J} |\eta| |\mathscr{F}_b \nabla_h \tht| \right) + \left( \sum_{\eta \in J} |\eta| |\mathscr{F}_b v_d| \right) \left( \sum_{\eta \in J} |\eta| |\mathscr{F}_b \partial_d \tht| \right) \right\} \,\ud t \\
		&\leq C \int_0^T \| v \|_{H^m}^2 \,\ud t + C \int_0^T \| \nabla_h \tht \|_{H^{m-1}}^2 \,\ud t + C \sup_{t \in [0,T]} \| \tht(t) \|_{H^m} \sum_{\eta \in J} \int_0^T |\eta| |\mathscr{F}_bv_d | \,\mathrm{d}t .
	\end{align*}
	To estimate $J_5$ and $J_7$ with \eqref{sing_est}, we consider $\eta \in D_1 \cup D_2$ and $\eta \in D_3$ separately. We note that
	\begin{align*}
		\sum_{\eta \in D_1 \cup D_2} \int_0^T \frac {|\tilde{n}|^2}{|\eta|} (e^{-\lambda_- t} - e^{-\lambda_+ t}) |\langle \mathscr{F}_b\bold{u}_0,\bold{a}_- \rangle| |\langle \bold{b}_-,e_{2} \rangle| \,\mathrm{d}t &\leq C\sum_{\eta \in J} \int_0^T |\tilde{n}| e^{- \frac t4} |\mathscr{F}_b \mathbf{u}_0|  \, \mathrm{d}t \leq C\| \mathbf{u}_0 \|_{H^m}
	\end{align*}
	and
	\begin{align*}
		&\sum_{\eta \in D_1 \cup D_2}\int_0^T \int_0^t \frac {|\tilde{n}|^2}{|\eta|} (e^{-\lambda_- (t - \tau)} - e^{-\lambda_+ (t - \tau)}) |\langle N(v,\tht)(\tau),\bold{a}_- \rangle| |\langle \bold{b}_-,e_{2} \rangle| \,\mathrm{d}\tau\mathrm{d}t \\
		&\hphantom{\qquad\qquad}\leq C\sum_{\eta \in J} \int_0^T |\tilde{n}| |N(v,\tht)(t)|  \, \mathrm{d}t \\
		&\hphantom{\qquad\qquad}\leq C\sum_{\eta \in J} \int_0^T (|\mathscr{F}[\nabla_h(v \cdot \nabla)v](t)| + |\mathscr{F}_b[\nabla_h(v \cdot \nabla)\tht](t)|)  \, \mathrm{d}t \\
		&\hphantom{\qquad\qquad}\leq C \int_0^T \| v \|_{H^m}^2 \,\ud t + C \int_0^T \| \nabla_h \tht \|_{H^{m-1}}^2 \,\ud t + C \sup_{t \in [0,T]} \| \tht(t) \|_{H^m} \sum_{\eta \in J} \int_0^T |\eta| |\mathscr{F}_bv_d | \,\mathrm{d}t .
	\end{align*}
	On the other hand,
	\begin{align*}
		\sum_{\eta \in D_3} \int_0^T \frac {|\tilde{n}|^2}{|\eta|} (e^{-\lambda_- t} - e^{-\lambda_+ t}) |\langle \mathscr{F}_b\bold{u}_0,\bold{a}_- \rangle| |\langle \bold{b}_-,e_{2} \rangle| \,\mathrm{d}t &\leq C\sum_{\eta \in J} \int_0^T \frac {|\tilde{n}|^2}{|\eta|} e^{-\frac {|\tilde{n}|^2}{|\eta|^2} t} |\mathscr{F}_b \bold{u}_0|  \, \mathrm{d}t \\
		&\leq C\sum_{\eta \in J} |\eta| |\mathscr{F}_b \bold{u}_0| \\
		&\leq C \| (v_0,\tht_0) \|_{H^m}.
	\end{align*}
	We can see
	\begin{align*}
		&\sum_{\eta \in D_3}\int_0^T \int_0^t \frac {|\tilde{n}|^2}{|\eta|} (e^{-\lambda_- (t - \tau)} - e^{-\lambda_+ (t - \tau)}) |\langle N(v,\tht)(\tau),\bold{a}_- \rangle| |\langle \bold{b}_-,e_{2} \rangle| \,\mathrm{d}\tau\mathrm{d}t \\
		&\hphantom{\qquad\qquad}\leq \sum_{\eta \in J} \int_0^T |\eta| |N(v,\tht)(t)|  \, \mathrm{d}t \\
		&\hphantom{\qquad\qquad}\leq \sum_{\eta \in J} \int_0^T |\eta| |\mathscr{F}(v \cdot \nabla)v(t)| \, \mathrm{d}t + \sum_{\eta \in J} \int_0^T |\eta| |\mathscr{F}_b(v \cdot \nabla)\tht(t)|  \, \mathrm{d}t.
	\end{align*}
	As estimating $I_7$ on the set $D_3$, we can deduce
	\begin{gather*}
		\sum_{\eta \in J} \int_0^T |\eta| (|\mathscr{F}(v \cdot \nabla)v(t)| \, \mathrm{d}t \\
		\leq C\int_0^T \left\{ \left(\sum_{\eta \in I} |\eta| |\mathscr{F}v(t)| \right)^2 + \left( \sum_{\eta \in I} |\mathscr{F}v(t)| \right) \left( \sum_{\eta \in I} |\eta|^2 |\mathscr{F}v(t)| \right) \right\}  \, \mathrm{d}t \\
		\leq C \int_0^T \| v(t) \|_{H^m}^2 \,\ud t
	\end{gather*}	
	and
	\begin{gather*}
		\sum_{\eta \in J} \int_0^T |\eta| |\mathscr{F}_b(v \cdot \nabla)\tht(t)|  \, \mathrm{d}t \\
		\leq \sum_{\eta \in J} \int_0^T |\eta| |\mathscr{F}_b(v_h \cdot \nabla_h)\tht(t)|  \, \mathrm{d}t + \sum_{\eta \in J} \int_0^T |\eta| |\mathscr{F}_b[v_d \partial_d \tht](t)|  \, \mathrm{d}t \\
		\leq C\int_0^T \left\{ \left( \sum_{\eta \in I} |\eta| |\mathscr{F}_cv_h(t)| \right) \left( \sum_{\eta \in J} |\tilde{n}| |\mathscr{F}_b\tht(t)| \right) + \left( \sum_{\eta \in I} |\mathscr{F}_cv_h(t)| \right) \left( \sum_{\eta \in J} |\eta| |\tilde{n}| |\mathscr{F}_b \tht(t)|\right) \right\}  \, \mathrm{d}t \\
		+ C\int_0^T \left\{ \left(\sum_{\eta \in J} |\eta| |\mathscr{F}_bv_d(t)| \right) \left( \sum_{\eta \in J} |\eta| |\mathscr{F}_b\tht (t)| \right) + \left( \sum_{\eta \in J} |\mathscr{F}_bv_d(t)| \right) \left( \sum _{\eta \in J} |\eta|^2 |\mathscr{F}_b\tht(t)| \right) \right\}  \, \mathrm{d}t \\
		\leq C \int_0^T \| v \|_{H^m}^2 \,\ud t + C \int_0^T \| \nabla_h \tht \|_{H^{m-1}}^2 \,\ud t + C \sup_{t \in [0,T]} \| \tht(t) \|_{H^m} \sum_{\eta \in J} \int_0^T |\eta| |\mathscr{F}_bv_d | \,\mathrm{d}t
	\end{gather*}
	for $m > 2+d/2$. Collecting the estimates for $J_5$, $J_6$, $J_7$, and $J_8$, we obtain \eqref{tht_L1_est2}. This completes the proof.
\end{proof}

\section{Proof of temporal decay estimates}\label{sec6}

In this section, let  $(v,\tht)$ be a smooth global-in-time solution to \eqref{EQ}. In addition, we assume that \eqref{sol_bdd_0} or \eqref{sol_bdd_1} holds in each case with
\begin{equation}\label{sm_ass}
	\|  (v_0,\tht_0) \|_{H^m} \leq \delta
\end{equation}
for sufficiently small $\delta>0$. The next three propositions are for the temporal decay estimates of $\| \bar{\tht}(t) \|_{L^2}$, $\| v(t) \|_{L^2}$, and $\| v_d(t) \|_{L^2}$ in both cases $\alpha=0$ and $\alpha = 1$. After that, we prove \eqref{tem_0} and \eqref{tem_1} combining with the temporal decay estimates for $\| v(t) \|_{\dot{H}^m}$ and $\| v_d(t) \|_{\dot{H}^m}$.

\begin{proposition}
	Let $d \in \mathbb{N}$ with $d \geq 2$ and $\alpha \in \{0,1\}$. Let $m \in \mathbb{N}$ with $m > 1+\frac d2 +\alpha$ and $(v, \theta)$ be a smooth global solution to \eqref{EQ} with \eqref{sol_bdd_1} or \eqref{sol_bdd_0}. Suppose that \eqref{sm_ass} be satisfied. Then, there exists a constant $C > 0$ such that
	\begin{equation}\label{tht_L2}
		\| \bar{\tht}(t) \|_{L^2}^2 \leq C (1+t)^{-\frac {m}{1+\alpha}}.
	\end{equation}
\end{proposition}
\begin{proof}
From the $v$ equations in \eqref{EQ}, we have
\begin{align*}
\frac 12 \frac {\ud}{\ud t} \int_{\Omega} |v|^2 \,\ud x &= -\|\Lambda^{\alpha} v\|_{L^2}^2 + \int_{\Omega} v_d \tht \,\ud x.
\end{align*}
On the other hand, we have from \eqref{avg_vd} and the $\tht$ equation in \eqref{EQ} that
\begin{align*}
\frac 12 \frac {\ud}{\ud t} \int_{\Omega} |\bar{\tht}|^2 \,\ud x &= - \int_{\Omega} (v \cdot \nabla) (\widetilde{\tht} +\bar{\tht}) \cdot \bar{\tht} \,\ud x -\int_{\Omega} v_d \bar{\tht} \,\ud x \\
&= - \int_{\Omega} v_d \partial_d \widetilde{\tht} \cdot \bar{\tht} \,\ud x -\int_{\Omega} v_d \tht \,\ud x,
\end{align*}
where $$\widetilde{\tht} := \int_{\bbT^{d-1}} \tht(x) \,\ud x_h.$$ Hence,
\begin{gather*}
\frac 12 \frac {\ud}{\ud t} (\| v \|_{L^2}^2 + \| \bar{\tht} \|_{L^2}^2) = -\| \Lambda^{\alpha} v \|_{L^2}^2 - \int_{\Omega} v_d \partial_d \widetilde{\tht} \cdot \bar{\tht} \,\ud x.
\end{gather*}
We can deduce from \eqref{df_vd}, \eqref{df_tht}, and \eqref{avg_vd} that
\begin{gather*}
-\frac {\ud}{\ud t} \int_{\Omega} v_d \Lambda^{-2\alpha}\tht \,\ud x = -\int_{\Omega} \partial_tv_d \Lambda^{-2\alpha} \bar{\tht} \,\ud x - \int_{\Omega} \partial_t\tht \Lambda^{-2\alpha}v_d \,\ud x \\
\leq \|(v \cdot \nabla)v\|_{L^2} \| \Lambda^{-2\alpha} \bar{\tht} \|_{L^2} + \| (v \cdot \nabla)\tht \|_{L^2} \|\Lambda^{-2\alpha} v_d \|_{L^2} - \frac 12 \| R_h \Lambda^{-\alpha} \tht \|_{L^2}^2 + \frac 32 \| \Lambda^{\alpha} v_d \|_{L^2}^2 .
\end{gather*}
Combining the above, we have
\begin{gather*}
\frac 12 \frac {\ud}{\ud t} \left(\| v \|_{L^2}^2 + \| \bar{\tht} \|_{L^2}^2 - \int_{\Omega} v_d \Lambda^{-2\alpha}\tht \,\ud x \right) \leq -\frac 14 (\| \Lambda^{\alpha} v \|_{L^2}^2 + \| R_h \Lambda^{-\alpha} \tht \|_{L^2}^2) \\
+ C\|v\|_{L^2} \| \nabla \tht \|_{L^{\infty}} \| \Lambda^{-2\alpha} v_d \|_{L^2} + C \| v \|_{L^2} \| v \|_{\dot{H}^m} \| \bar{\tht} \|_{L^2} - \int_{\Omega} v_d \partial_d \widetilde{\tht} \cdot \bar{\tht} \,\ud x.
\end{gather*}
To estimate the integral on the right-hand side, we note
\begin{gather*}
\left| - \int_{\Omega} v_d \partial_d \widetilde{\tht} \cdot \bar{\tht} \,\ud x \right| \leq \left| - \int_{\Omega} v_d \partial_d \widetilde{\tht} \cdot -\Delta_h (-\Delta)^{-1} \bar{\tht} \,\ud x \right| + \left| - \int_{\Omega} v_d \partial_d \widetilde{\tht} \cdot -\partial_d^2 (-\Delta)^{-1} \bar{\tht} \,\ud x \right|.
\end{gather*}
We consider $\alpha = 0$ case first. The right-hand side is bounded by
\begin{gather*}
\| v_d \|_{L^2} \| \partial_d \tht \|_{L^{\infty}} \| R_h^2 \tht \|_{L^2} + \left| -\int_{\Omega} (\nabla_h \cdot v_h) \partial_d \widetilde{\tht} \cdot \partial_d (-\Delta)^{-1} \bar{\tht} \,\ud x - \int_{\Omega} v_d \partial_d^2 \widetilde{\tht} \cdot \partial_d (-\Delta)^{-1} \bar{\tht} \,\ud x \right| \\
\leq \| v_d \|_{L^2} \| \partial_d \tht \|_{L^{\infty}} \| R_h^2 \tht \|_{L^2}  + \| v_h \|_{L^2} \| \partial_d \widetilde{\tht} \|_{L^{\infty}} \| R_h \tht \|_{L^2} + \left| - \int_{\Omega} v_d \partial_d^2 \widetilde{\tht} \cdot \partial_d (-\Delta)^{-1} \bar{\tht} \,\ud x \right|.
\end{gather*}
Using $v_d = -(\Delta_h)(-\Delta)^{-1} v_d + \partial_d \nabla_h (-\Delta)^{-1} v_h$, we have
\begin{gather*}
\left| - \int_{\Omega} v_d \partial_d^2 \widetilde{\tht} \cdot \partial_d (-\Delta)^{-1} \bar{\tht} \,\ud x \right| \\
\leq \left| - \int_{\Omega} \nabla_h (-\Delta)^{-1} v_d \partial_d^2 \widetilde{\tht} \cdot \nabla_h \partial_d (-\Delta)^{-1} \bar{\tht} \,\ud x \right| + \left| \int_{\Omega} \partial_d (-\Delta)^{-1}v_h \partial_d^2 \widetilde{\tht} \cdot \nabla_h \partial_d (-\Delta)^{-1} \bar{\tht} \,\ud x \right| \\
\leq \| \nabla (-\Delta)^{-1} v \|_{L^2_{x_h}L^{\infty}_{x_d}} \| \partial_d^2 \widetilde{\tht} \|_{L^{2}} \| R_h \tht \|_{L^2} \\
\leq \| v \|_{L^2} \| \tht \|_{H^m} \| R_h \tht \|_{L^2}.
\end{gather*}
For $\alpha = 1$, we can see
\begin{gather*}
\left| - \int_{\Omega} v_d \partial_d \widetilde{\tht} \cdot -\Delta_h (-\Delta)^{-1} \bar{\tht} \,\ud x \right| = \left| - \int_{\Omega} \nabla_h v_d \partial_d \widetilde{\tht} \cdot \nabla_h (-\Delta)^{-1} \bar{\tht} \,\ud x \right| \\
\leq \| \nabla_h v_d \|_{L^2} \| \partial_d \tht \|_{L^{\infty}} \| \Lambda^{-1} R_h \tht \|_{L^2}
\end{gather*}
and
\begin{gather*}
\left| - \int_{\Omega} v_d \partial_d \widetilde{\tht} \cdot -\partial_d^2 (-\Delta)^{-1} \bar{\tht} \,\ud x \right| \\
\leq \left| \int_{\Omega} (\nabla_h \cdot \partial_d v_h) \partial_d \widetilde{\tht} \cdot (-\Delta)^{-1} \bar{\tht} \,\ud x \right| + 2\left| \int_{\Omega} (\nabla_h \cdot v_h) \partial_d^2 \widetilde{\tht} \cdot (-\Delta)^{-1} \bar{\tht} \,\ud x \right| + \left| \int_{\Omega} v_d \partial_d^3 \widetilde{\tht} \cdot (-\Delta)^{-1} \bar{\tht} \,\ud x \right| \\
\leq \| \partial_d v_h \|_{L^2} \| \partial_d \tht \|_{L^{\infty}} \| \Lambda^{-1} R_h \tht \|_{L^2} + 2\| v_h \|_{L^2} \| \partial_d^2 \tht \|_{L^{\infty}} \| \Lambda^{-1} R_h \tht \|_{L^2} \\
+ \| \nabla (-\Delta)^{-1} v \|_{L^2_{x_h}L^{\infty}_{x_d}} \| \partial_d^3 \widetilde{\tht} \|_{L^{2}} \| (-\Delta)^{-1} \tht \|_{L^2},
\end{gather*}
where $v_d = -(\Delta_h)(-\Delta)^{-1} v_d + \partial_d \nabla_h (-\Delta)^{-1} v_h$ also used here. Hence, we can deduce
\begin{gather*}
\left| - \int_{\Omega} v_d \partial_d \widetilde{\tht} \cdot \bar{\tht} \,\ud x \right| \leq C \| \Lambda^{\alpha} v \|_{L^2} \| \tht \|_{H^m} \| \Lambda^{-\alpha} R_h \tht \|_{L^2}
\end{gather*}
in both cases. Combining the above and using \eqref{sm_ass}, we can have
\begin{gather*}
\frac 12 \frac {\ud}{\ud t} \left(\| v \|_{L^2}^2 + \| \bar{\tht} \|_{L^2}^2 - \int_{\Omega} v_d \Lambda^{-2\alpha}\tht \,\ud x \right) \\
\leq -(\frac 14 - C(\| v \|_{H^m}^2 + \| \tht \|_{H^m}^2)) (\| \Lambda^{\alpha} v \|_{L^2}^2 + \| R_h \Lambda^{-\alpha} \tht \|_{L^2}^2) + C \| v \|_{L^2} \| v \|_{\dot{H}^m} \| \bar{\tht} \|_{L^2} \\
\leq -\frac 18 (\| v \|_{L^2}^2 + \| R_h \Lambda^{-\alpha} \tht \|_{L^2}^2) + C \| v \|_{\dot{H}^m}^2 \| \bar{\tht} \|_{L^2}^2.
\end{gather*}
Let $M\geq 1$ which will be specified later. Since
\begin{equation*}
\begin{aligned}
\frac 1M \| \bar{\tht} \|_{L^2}^2 - \| R_h \Lambda^{-\alpha} \tht \|_{L^2}^2 &= \sum_{|\tilde{n}| \neq 0} \left( \frac 1M - \frac {|\tilde{n}|^2}{|\eta|^{2(1+\alpha)}} \right) |\mathscr{F} \tht(\eta)|^2 \\
&\leq \frac 1M \sum_{\frac {|\tilde{n}|^2}{|\eta|^{2(1+\alpha)}} \leq \frac 1M, |\tilde{n}| \neq 0} |\mathscr{F} \tht(\eta)|^2 \\
&\leq \frac 1{M^{1+\frac {m-1-\alpha}{1+\alpha}}} \| \bar{\tht} \|_{\dot{H}^{m-1-\alpha}}^2 \\
&\leq \frac 1{M^{\frac {m}{1+\alpha}}} \| R_h \tht \|_{\dot{H}^{m-\alpha}}^2
\end{aligned}
\end{equation*}
and
\begin{equation}\label{smp_ineq_3}
\left| \int_{\Omega} v_d \Lambda^{-2\alpha} \tht \,\ud x \right| \leq \| v_d \|_{L^2} \| \Lambda^{-2\alpha} \bar{\tht} \|_{L^2} \leq \frac 12 \| v \|_{L^2}^2 + \frac 12 \| \bar{\tht} \|_{L^2}^2,
\end{equation}
it holds
\begin{equation}\label{M_est}
\begin{gathered}
-\frac 18 (\| v \|_{L^2}^2 + \| R_h \Lambda^{-\alpha} \tht \|_{L^2}^2) \leq -\frac 1{8M} \left(\| v \|_{L^2}^2 + \| \bar{\tht} \|_{L^2}^2 -\frac 12 \int_{\Omega} v_d \Lambda^{-2\alpha} \tht \,\ud x\right) \\
+ \frac 1{16M} \int_{\Omega} v_d \Lambda^{-2\alpha} \tht \,\ud x + \frac 18 \left( \frac 1{M} \| \bar{\tht} \|_{L^2}^2 - \| R_h \Lambda^{-\alpha} \tht \|_{L^2}^2 \right) \\
\leq -\frac 1{16M} \left( \| v \|_{L^2}^2 + \| \bar{\tht} \|_{L^2}^2 -\int_{\Omega} v_d \Lambda^{-2\alpha} \tht \,\ud x \right) + \frac 1{8M^{\frac {m}{1+\alpha}}} \| R_h \tht \|_{\dot{H}^{m-\alpha}}^2.
\end{gathered}
\end{equation}
Thus,
\begin{gather*}
\frac 12 \frac {\ud}{\ud t} \left(\| v \|_{L^2}^2 + \| \bar{\tht} \|_{L^2}^2 - \int_{\Omega} v_d \Lambda^{-2\alpha}\tht \,\ud x \right) \\
\leq -\frac 1{16M} \left( \| v \|_{L^2}^2 + \| \bar{\tht} \|_{L^2}^2 -\int_{\Omega} v_d \Lambda^{-2\alpha} \tht \,\ud x \right) + \frac 1{8M^{\frac {m}{1+\alpha}}} \| R_h \tht \|_{\dot{H}^{m-\alpha}}^2 + C \| v \|_{\dot{H}^m}^2 \| \bar{\tht} \|_{L^2}^2.
\end{gather*}
Taking $M = 1+\frac t{8\frac {m}{1+\alpha}}$ and multiplying both terms by $2M^{\frac {m}{1+\alpha}}$, we obtain by \eqref{smp_ineq_3} that
\begin{gather*}
\frac {\ud}{\ud t} \left( (1+\frac t{8\frac {m}{1+\alpha}})^{\frac {m}{1+\alpha}} \left(\| v \|_{L^2}^2 + \| \bar{\tht} \|_{L^2}^2 - \int_{\Omega} v_d \Lambda^{-2\alpha}\tht \,\ud x \right) \right) \\
\leq  C \| R_h \tht \|_{\dot{H}^{m-\alpha}}^2 + C \| v \|_{\dot{H}^m}^2 (1+\frac t{8\frac {m}{1+\alpha}})^{\frac {m}{1+\alpha}} \left(\| v \|_{L^2}^2 + \| \bar{\tht} \|_{L^2}^2 - \int_{\Omega} v_d \Lambda^{-2\alpha}\tht \,\ud x \right).
\end{gather*}
Using Gr\"onwall's inequality, we obtain \eqref{tht_L2}. This completes the proof.
\end{proof}

\begin{proposition}
	Let $d \in \mathbb{N}$ with $d \geq 2$ and $\alpha \in \{0,1\}$. Let $m \in \mathbb{N}$ with $m > 1+\frac d2 +\alpha$ and $(v, \theta)$ be a smooth global solution to \eqref{EQ} with \eqref{sol_bdd_1} or \eqref{sol_bdd_0}. Suppose that \eqref{sm_ass} be satisfied. Then, there exists a constant $C > 0$ such that
	\begin{equation}\label{v_L2}
		\| v(t) \|_{L^2}^2 + \| R_h \Lambda^{-\alpha} \tht(t) \|_{L^2}^2 \leq C (1+t)^{-(1+\frac m{1+\alpha})}.
	\end{equation}
\end{proposition}
\begin{proof}
From the $v$ equations in \eqref{EQ}, we have
\begin{align*}
\frac 12 \frac {\ud}{\ud t} \int_{\Omega} |v|^2 \,\ud x &\leq -\|\Lambda^{\alpha} v\|_{L^2}^2 + C\left( \sum_{|\tilde{n}| \neq 0} \frac{|\tilde{n}|^2}{|\eta|^{2(1+\alpha)}} |\mathscr{F}_b \tht(\eta)|^{2} \right)^{\frac 12}  \left( \sum_{|\tilde{n}| \neq 0} |\eta|^{2\alpha} |\mathscr{F} v(\eta)|^2 \right)^{\frac 12} \\
&\leq -\frac 12 \| v \|_{L^2}^2 + C \sum_{|\tilde{n}| \neq 0} \frac{|\tilde{n}|^{2}}{|\eta|^{2(1+\alpha)}} |\mathscr{F}_b \tht(\eta)|^2.
\end{align*}
Since Duhamel's principle implies
\begin{align*}
\| v(t) \|_{L^2}^2 &\leq e^{-t} \| v_0 \|_{L^2}^2 + C\int_0^t e^{-(t-\tau)} \sum_{|\tilde{n}| \neq 0} \frac{|\tilde{n}|^{2}}{|\eta|^{2(1+\alpha)}} |\mathscr{F}_b \tht(\eta)|^2 \,\ud \tau \\
&\leq e^{-t} \| v_0 \|_{L^2}^2 + C\sup_{\tau \in [0,t]} (1+\tau)^{1+\frac m{1+\alpha}} \| R_h\Lambda^{-\alpha} \tht(\tau) \|_{L^2}^2 \int_0^t e^{-(t-\tau)} (1+\tau)^{-(1+\frac m{1+\alpha})} \,\ud \tau,
\end{align*}
we have
\begin{equation}\label{v_L2_est}
\sup_{\tau \in [0,t]} (1+\tau)^{1+\frac m{1+\alpha}} \| v(\tau) \|_{L^2}^2 \leq C\left( \| v_0 \|_{L^2}^2 + \sup_{\tau \in [0,t]} (1+\tau)^{1+\frac m{1+\alpha}} \| R_h\Lambda^{-\alpha} \tht(\tau) \|_{L^2}^2 \right).
\end{equation}
On the other hand, from \eqref{df_vd} and \eqref{df_tht}, we have
\begin{gather*}
\frac 12 \frac {\ud}{\ud t} \sum_{\eta \in J} |\mathscr{F}_b \Lambda^{-\alpha} v_d(\eta)|^2 \leq -\| v_d\|_{L^2}^2 + \| (v \cdot \nabla)v\|_{L^2} \| \Lambda^{-2\alpha} v_d \|_{L^2} + \sum_{\eta \in J} \frac{|\tilde{n}|^{2}}{|\eta|^{2(1+\alpha)}} \mathscr{F}_b \tht(\eta)\mathscr{F}_b v_d(\eta)
\end{gather*}
and
\begin{gather*}
\frac 12 \frac {\ud}{\ud t} \sum_{\eta \in J} |\mathscr{F}_b R_h \Lambda^{-\alpha} \tht(\eta)|^2 \leq -\int_{\Omega} (v \cdot \nabla)\tht R_h^{2} \Lambda^{-2\alpha}\tht \,\ud x - \sum_{\eta \in J} \frac{|\tilde{n}|^{2}}{|\eta|^{2(1+\alpha)}} \mathscr{F}_b \tht(\eta)\mathscr{F}_b v_d(\eta)
\end{gather*}
respectively. Thus,
\begin{gather*}
\frac 12 \frac {\ud}{\ud t} (\|\Lambda^{-\alpha} v_d\|_{L^2}^2 + \| R_h\Lambda^{-\alpha} \tht \|_{L^2}^2) \leq -\| v_d\|_{L^2}^2 + \|(v \cdot \nabla)v\|_{L^2} \| v_d \|_{L^2} -\int_{\Omega} (v \cdot \nabla)\tht R_h^{2} \Lambda^{-2\alpha} \tht \,\ud x.
\end{gather*}
Moreover, we can deduce from \eqref{df_vd} and \eqref{df_tht} that
\begin{gather*}
-\frac {\ud}{\ud t} \int_{\Omega} v_d R_h^{2} \Lambda^{-4\alpha}\tht \,\ud x = -\int_{\Omega} \partial_tv_d R_h^{2} \Lambda^{-4\alpha}\tht \,\ud x - \int_{\Omega} \partial_t\tht R_h^{2} \Lambda^{-4\alpha}v_d \,\ud x \\
\leq \|(v \cdot \nabla)v\|_{L^2} \| R_h^{2} \Lambda^{-4\alpha} \tht \|_{L^2} + \int_{\Omega} (v \cdot \nabla)\tht R_h^{2} \Lambda^{-4\alpha} v_d \,\ud x - \frac 12 \| R_h^{2} \Lambda^{-2\alpha} \tht \|_{L^2}^2 + \frac 32 \|  v_d \|_{L^2}^2 .
\end{gather*}
Combining the above, we have
\begin{gather*}
\frac 12 \frac {\ud}{\ud t} \left(\|\Lambda^{-\alpha}v_d\|_{L^2}^2 + \| R_h \Lambda^{-\alpha} \tht \|_{L^2}^2 - \int_{\Omega} v_d R_h^{2} \Lambda^{-4\alpha}\tht \,\ud x \right) \leq -\frac 14 (\| v_d\|_{L^2}^2 + \| R_h^{2}\Lambda^{-2\alpha} \tht \|_{L^2}^2) \\
+ C\|v\|_{L^2} \| \nabla v \|_{L^\infty} (\| v_d \|_{L^2} + \| R_h^{2} \Lambda^{-2\alpha}\tht \|_{L^2}) -\int_{\Omega} (v \cdot \nabla)\tht (R_h^{2} \Lambda^{-2\alpha}\tht - \frac 12 R_h^{2} \Lambda^{-2\alpha}v_d) \,\ud x.
\end{gather*}
By $(v \cdot \nabla)\tht = (v_h \cdot \nabla_h) \tht + v_d \partial_d \tht$, we deduce
\begin{gather*}
\left| -\int_{\Omega} (v \cdot \nabla)\tht (R_h^{2} \Lambda^{-2\alpha}\tht - \frac 12 R_h^{2} \Lambda^{-2\alpha}v_d) \,\ud x \right| \\
\leq \| v_h \|_{L^2} \| \nabla_h \tht \|_{L^{\infty}} (\| v_d \|_{L^2} + \| R_h^{2} \Lambda^{-2\alpha} \tht \|_{L^2}) + C(\| v_d \|_{L^2}^2 + \| R_h^2 \Lambda^{-2\alpha} \tht \|_{L^2}^2) \| \partial_d \tht \|_{L^{\infty}}.
\end{gather*}
Thus, by $W^{1,\infty}(\Omega) \hookrightarrow H^{m-\alpha}(\Omega)$, \eqref{sm_ass}, and Young's inequality, we have
\begin{gather*}
\frac 12 \frac {\ud}{\ud t} \left(\|\Lambda^{-\alpha}v_d\|_{L^2}^2 + \| R_h \Lambda^{-\alpha} \tht \|_{L^2}^2 - \int_{\Omega} v_d R_h^{2} \Lambda^{-4\alpha}\tht \,\ud x \right) \\
\leq -(\frac 14 - C \| \tht \|_{H^m})  ( \| R_h^{2}\Lambda^{-2\alpha} \tht \|_{L^2}^2 + \| v_d \|_{L^2}^2) + C\|v\|_{L^2} (\| \nabla v \|_{L^\infty} + \| \nabla_h \tht \|_{L^{\infty}}) (\| v_d \|_{L^2} + \| R_h^{2} \Lambda^{-2\alpha}\tht \|_{L^2})  \\
\leq -\frac 18 ( \| R_h^{2} \Lambda^{-2\alpha} \tht \|_{L^2}^2 + \| v_d \|_{L^2}^2)+ C \| v \|_{L^2}^2 (\| v \|_{\dot{H}^m}^2 + \| R_h \tht \|_{\dot{H}^{m-\alpha}}^2).
\end{gather*}
Let $M\geq 1$ which will be specified later. Since
\begin{equation*}
\begin{aligned}
\frac 1M \| R_h \Lambda^{-\alpha} \tht \|_{L^2}^2 - \| R_h^{2} \Lambda^{-2\alpha} \tht \|_{L^2}^2 &= \sum_{\eta \in J} \left( \frac 1M - \frac {|\tilde{n}|^2}{|\eta|^{2(1+\alpha)}} \right) |\mathscr{F} R_h \Lambda^{-\alpha}  \tht(\eta)|^2 \\
&\leq \frac 1M \sum_{\frac {|\tilde{n}|^2}{|\eta|^{2(1+\alpha)}} \leq \frac 1M, |\tilde{n}| \neq 0} |\mathscr{F} R_h \Lambda^{-\alpha} \tht(\eta)|^2 \\
&\leq \frac 1{M^{1+\frac {m}{1+\alpha}}} \| R_h \tht \|_{\dot{H}^{m-\alpha}}^2,
\end{aligned}
\end{equation*}
together with
\begin{equation}\label{smp_ineq_2}
\left| -\int_{\Omega} v_d R_h^2 \Lambda^{-4\alpha} \tht \,\ud x \right| \leq \| \Lambda^{-2\alpha} v_d \|_{L^2} \| R_h^2 \Lambda^{-2\alpha} \tht \|_{L^2} \leq \frac 12 \| v_d \|_{L^2}^2 + \frac 12 \| R_h \Lambda^{-2\alpha} \tht \|_{L^2}^2,
\end{equation}
we can have as estimating \eqref{M_est} that
\begin{gather*}
-\frac 18 \left( \| R_h^{2} \Lambda^{-2\alpha} \tht \|_{L^2}^2 + \| v_d \|_{L^2}^2 \right) \\
\leq -\frac 1{16M} \left(\| R_h \Lambda^{-\alpha} \tht \|_{L^2}^2 + \| v_d \|_{L^2}^2 -\int_{\Omega} v_d R_h^2 \Lambda^{-4\alpha} \tht \,\ud x \right) + \frac 1{8M^{1+\frac {m}{1+\alpha}}} \| R_h \tht \|_{\dot{H}^{m-\alpha}}^2.
\end{gather*}
Thus,
\begin{gather*}
\frac 12 \frac {\ud}{\ud t} \left(\|\Lambda^{-\alpha}v_d\|_{L^2}^2 + \| R_h \Lambda^{-\alpha} \tht \|_{L^2}^2 - \int_{\Omega} v_d R_h^{2} \Lambda^{-4\alpha}\tht \,\ud x \right) \\
\leq -\frac 1{16M} \left(\| R_h \Lambda^{-\alpha} \tht \|_{L^2}^2 + \| \Lambda^{-\alpha} v_d \|_{L^2}^2 -\int_{\Omega} v_d R_h^2 \Lambda^{-4\alpha} \tht \,\ud x \right) \\
+ \frac 1{8M^{1+\frac {m}{1+\alpha}}} \| R_h \tht \|_{\dot{H}^{m-\alpha}}^2 + C \| v \|_{L^2}^2 (\| v \|_{\dot{H}^m}^2 + \| R_h \tht \|_{\dot{H}^{m-\alpha}}^2).
\end{gather*}
We take $M = 1+\frac t{8(1+\frac m{1+\alpha})}$. Then, we can have with \eqref{v_L2_est} that
\begin{gather*}
\frac {\ud}{\ud t} \left( (1+\frac t{8(1+\frac m{1+\alpha})})^{1+\frac m{1+\alpha}} \left( \| \Lambda^{-\alpha} v_d\|_{L^2}^2 + \| R_h \Lambda^{-\alpha}  \tht \|_{L^2}^2 - \int_{\Omega} v_d R_h^2 \Lambda^{-4\alpha} \tht \,\ud x \right) \right) \\
\leq  C \sup_{\tau \in [0,t]} (1+\frac \tau{8(1+\frac m{1+\alpha})})^{1+\frac m{1+\alpha}} \left( \| \Lambda^{-\alpha} v_d(\tau) \|_{L^2}^2 + \| R_h \Lambda^{-\alpha} \tht(\tau) \|_{L^2}^2\right) (\| v\|_{\dot{H}^m}^2 + \| R_h \tht \|_{\dot{H}^{m-\alpha}}^2) \\
+ C (\| v \|_{\dot{H}^m}^2 + \| R_h \tht \|_{\dot{H}^{m-\alpha}}^2).
\end{gather*}
We integrate it over time and use \eqref{smp_ineq_2} with $$\int_0^{\infty} (\| v \|_{\dot{H}^{m}}^2 + \| R_h \tht \|_{\dot{H}^{m-\alpha}}^2) \,\ud t \leq C.$$ Then, for $$f(t) := \sup_{\tau \in [0,t]} (1+\frac \tau{8(1+\frac m{1+\alpha})})^{1+\frac m{1+\alpha}} \left( \| \Lambda^{-\alpha} v_d(\tau) \|_{L^2}^2 + \| R_h \Lambda^{-\alpha} \tht(\tau) \|_{L^2}^2\right),$$ it holds
\begin{gather*}
f(t) \leq C + \int_0^t f(\tau) (\| v \|_{\dot{H}^m}^2 +  \| R_h \tht \|_{\dot{H}^{m-\alpha}}^2 + \| \nabla v_d \|_{L^{\infty}}) \,\ud \tau.
\end{gather*}
Applying Gr\"onwall's ineqaulity, we obtain
\begin{equation*}
\sup_{\tau \in [0,t]} (1+\frac \tau{8(1+\frac m{1+\alpha})})^{1+\frac m{1+\alpha}} \left( \| \Lambda^{-\alpha} v_d(\tau) \|_{L^2}^2 + \| R_h \Lambda^{-\alpha} \tht(\tau) \|_{L^2}^2\right)  \leq C.
\end{equation*}
With \eqref{v_L2_est}, we deduce \eqref{v_L2}. This completes the proof.
\end{proof}

\begin{proposition}
	Let $d \in \mathbb{N}$ with $d \geq 2$ and $\alpha \in \{0,1\}$. Let $m \in \mathbb{N}$ with $m > 2+\frac d2 +\alpha$ and $(v, \theta)$ be a smooth global solution to \eqref{EQ} with \eqref{sol_bdd_1} or \eqref{sol_bdd_0}. Suppose that \eqref{sm_ass} be satisfied. Then, there exists a constant $C > 0$ such that
	\begin{equation}\label{vd_L2}
		\| v_d(t) \|_{L^2}^2 + \| R_h^{2} \Lambda^{-2\alpha} \tht(t) \|_{L^2}^2 \leq C (1+t)^{-(2 + \frac {m}{1+\alpha})}.
	\end{equation}
\end{proposition}
\begin{proof}
Recalling the definition of $\bfb_{\pm}$, we can verify that $$\mathscr{F}_b v_d = \frac 1{\lambda_{+}} \frac {|\tilde{n}|^2}{|\eta|^2} \mathscr{F}_b \tht + \frac {|\tilde{n}|^2}{|\eta|^2} \left( \frac 1{\lambda_{-}} -\frac 1{\lambda_{+}} \right) \langle \mathscr{F}_b \mathbf{u}, \mathbf{a}_{+} \rangle \langle \mathbf{b}_{+}, e_2 \rangle, \qquad \eta \in J.$$ We note that $$\frac 1{|\lambda_{+}|} \leq \frac {|\eta|}{|\tilde{n}|} \leq \frac 2{|\eta|^{2\alpha}}, \qquad \eta \in D_1 \qquad \mbox{and} \qquad \frac 1{|\lambda_{+}|} \leq \frac 2{|\eta|^{2\alpha}}, \qquad \eta \not\in D_1.$$ Together with $$\frac {|\tilde{n}|^2}{|\eta|^2} \left( \frac 1{\lambda_{-}} -\frac 1{\lambda_{+}} \right) \langle \mathscr{F}_b \mathbf{u}, \mathbf{a}_{+} \rangle \langle \mathbf{b}_{+}, e_2 \rangle = \frac 1{\lambda_+} \langle \mathscr{F}_b \mathbf{u}, \mathbf{a}_{+} \rangle,$$ we have
\begin{equation}\label{vd_est_u}
\| v_d (t) \|_{L^2} \leq C \| R_h^2 \Lambda^{-2\alpha} \tht(t) \|_{L^2} + C \left( \sum_{\eta \in J} \frac {1}{|\eta|^{4\alpha}} | \langle \mathscr{F}_b \mathbf{u},\mathbf{a}_{+} \rangle |^2 \right)^{\frac 12}.
\end{equation}
We show
\begin{equation}\label{u+_est}
\left( \sum_{\eta \in J} \frac {1}{|\eta|^{4\alpha}} | \langle \mathscr{F}_b \mathbf{u},\mathbf{a}_{+} \rangle |^2 \right)^{\frac 12} \leq (1+t)^{-(1+\frac m{2(1+\alpha)})} (C + C \delta \sup_{\tau \in [0,t]} (1+\tau)^{1+\frac m{2(1+\alpha)}} \| v_d(\tau) \|_{L^2}),
\end{equation}
where $\delta$ is a small constant in \eqref{sm_ass}. Then by taking $\delta$ small enough, we obtain
\begin{equation}\label{vd_Rh_est}
	\sup_{\tau \in [0,t]} (1+\tau)^{1+\frac m{2(1+\alpha)}} \| v_d(\tau) \|_{L^2} \leq C + C \sup_{\tau \in [0,t]} (1+\tau)^{1+\frac m{2(1+\alpha)}} \| R_h^2 \Lambda^{-2\alpha} \tht(\tau) \|_{L^2}.
\end{equation}
We recall \eqref{df_u+} and have
\begin{equation*}
	\langle \mathscr{F}_b\bold{u}(t),\bold{a}_+ \rangle = e^{-\lambda_+ t} \langle \mathscr{F}_b \bold{u}_0,\bold{a}_+ \rangle - \int_0^t e^{-\lambda_+ (t - \tau)} \langle N(v,\tht)(\tau),\bold{a}_+ \rangle \,\mathrm{d}\tau .
\end{equation*}
Since $|e^{-\lambda_+ t}| \leq e^{-|\eta|^{2\alpha} \frac t2}$ for $\eta \in J$, it follows by the Minkowski inequality
\begin{gather*}
	\left( \sum_{\eta \in J} \frac {1}{|\eta|^{4\alpha}} | \langle \mathscr{F}_b \mathbf{u},\mathbf{a}_{+} \rangle |^2 \right)^{\frac 12} \\
	\leq \left( \sum_{\eta \in J} \frac {1}{|\eta|^{4\alpha}} e^{-|\eta|^{2\alpha}t} | \langle \mathscr{F}_b \mathbf{u}_0,\mathbf{a}_{+} \rangle |^2 \right)^{\frac 12} + \int_0^t \left( \sum_{\eta \in J} \frac {1}{|\eta|^{4\alpha}} e^{-|\eta|^{2\alpha}(t - \tau)} |\langle N(v,\tht)(\tau),\bold{a}_+ \rangle|^2 \right)^{\frac 12} \,\mathrm{d}\tau.
\end{gather*}
From the simple fact $|\mathbf{a}_{+}|^2 = |\lambda_{+}|^2 + \frac {|\tilde{n}|^4}{|\eta|^4} \leq C|\eta|^{4\alpha}$ with \eqref{sm_ass}, we have $$\left( \sum_{\eta \in J} \frac {1}{|\eta|^{4\alpha}} e^{-|\eta|^{2\alpha}t} | \langle \mathscr{F}_b \mathbf{u}_0,\mathbf{a}_{+} \rangle |^2 \right)^{\frac 12} \leq Ce^{-t} \| \mathbf{u}_0 \|_{L^2} \leq C (1+t)^{-(1+\frac m{2(1+\alpha)})}.$$ We note that
\begin{equation}\label{N_est}
|\langle N(v,\tht)(\tau),\bold{a}_+ \rangle| \leq |\mathscr{F} (v\cdot \nabla)v||\lambda_{+}| + |\mathscr{F}_b (v \cdot \nabla)\tht| \frac {|\tilde{n}|^2}{|\eta|^2}.
\end{equation}
Thus, it holds 
\begin{gather*}
	\int_0^t \left( \sum_{\eta \in J} \frac {1}{|\eta|^{4\alpha}}  e^{-|\eta|^{2}(t - \tau)} |\langle N(v,\tht)(\tau),\bold{a}_+ \rangle|^2 \right)^{\frac 12} \,\mathrm{d}\tau \\
	\leq \int_0^t \left( \sum_{\eta \in J} e^{-|\eta|^{2}(t - \tau)} | \mathscr{F}(v \cdot \nabla)v |^2 \right)^{\frac 12} \,\mathrm{d}\tau + \int_0^t \left( \sum_{\eta \in J} e^{-|\eta|^{2}(t - \tau)} \frac {1}{|\eta|^{4\alpha}}  | \mathscr{F}_b(v \cdot \nabla)\tht |^2 \right)^{\frac 12} \,\mathrm{d}\tau\\
	\leq \int_0^t e^{-(t - \tau)} (\| (v \cdot \nabla)v (\tau) \|_{L^2}  + \| (v_h \cdot \nabla_h)\bar{\tht} (\tau) \|_{L^2} + \| v_d \partial_d \tht (\tau) \|_{L^2}) \,\mathrm{d}\tau.
\end{gather*}
We have used $$(v \cdot \nabla)\tht = (v_h \cdot \nabla_h)\bar{\tht} + v_d \partial_d \tht$$ in the last inequality. We note by $H^{m-2-\alpha} \hookrightarrow L^{\infty}$ $$\| (v \cdot \nabla)v \|_{L^2} \leq \| v \|_{L^2} \| \nabla v \|_{L^{\infty}} \leq C \| v \|_{L^2}^{1+\frac {1+\alpha}m} \| v \|_{\dot{H}^m}^{1-\frac {1+\alpha}m},$$ $$\| (v_h \cdot \nabla_h) \bar{\tht} \|_{L^2} \leq \| v \|_{L^2} \| \nabla_h \bar{\tht} \|_{L^{\infty}} \leq C \| v \|_{L^2} \| \bar{\tht} \|_{L^2}^{\frac {1+\alpha}m} \| \bar{\tht} \|_{\dot{H}^{m}}^{1-\frac {1+\alpha}m},$$ and $$\| v_d \partial_d \tht \|_{L^2} \leq \| v_d \|_{L^2} \| \partial_d \tht \|_{L^{\infty}} \leq C \| v_d \|_{L^2} \| \tht \|_{H^m}.$$
Combining \eqref{v_L2}, \eqref{tht_L2} and our assumptions, we can see 
\begin{gather*}
	(1+\tau)^{1+\frac m{2(1+\alpha)}} (\| (v \cdot \nabla)v (\tau) \|_{L^2}  + \| (v_h \cdot \nabla_h)\bar{\tht}  (\tau) \|_{L^2} + \| v_d \partial_d \tht \|_{L^2}) \\
	\leq C + C\delta \sup_{\tau \in [0,t]} (1+\tau)^{1+\frac m{2(1+\alpha)}} \| v_d(\tau) \|_{L^2},
\end{gather*}
where $\delta$ is a small constant in \eqref{sm_ass}. Hence,
\begin{gather*}
\int_0^t e^{-(t - \tau)} (\| (v \cdot \nabla)v (\tau) \|_{L^2}  + \| (v_h \cdot \nabla_h)\bar{\tht}  (\tau) \|_{L^2} + \| v_d \partial_d \tht \|_{L^2}) \,\mathrm{d}\tau \\
 \leq C (1+t)^{-(1+\frac m{2(1+\alpha)})} (C + C\delta \sup_{\tau \in [0,t]} (1+\tau)^{1+\frac m{2(1+\alpha)}} \| v_d(\tau) \|_{L^2}).
 \end{gather*}
Collecting the above estimates, we obtain \eqref{u+_est} and \eqref{vd_Rh_est}.

Now, we show 
\begin{equation}\label{Rh_tht_est}
\| R_h^2 \Lambda^{-2\alpha} \tht(t) \|_{L^2} \leq C(1+t)^{-(1+\frac m{2(1+\alpha)})}.
\end{equation}
Since we have from \eqref{df_vd} and \eqref{df_tht},
\begin{gather*}
\frac 12 \frac {\ud}{\ud t} \sum_{\eta \in J} |\mathscr{F} R_h \Lambda^{-2\alpha} v_d(\eta)|^2 \\
\leq -\| R_h \Lambda^{-\alpha}v_d\|_{L^2}^2 + \| (v \cdot \nabla)v\|_{L^2} \| R_h^2\Lambda^{-4\alpha} v_d \|_{L^2} + \sum_{\eta \in J} \frac{|\tilde{n}|^{4}}{|\eta|^{4(1+\alpha)}} \mathscr{F}_b \tht(\eta)\mathscr{F}_b v_d(\eta)
\end{gather*}
and
\begin{gather*}
\frac 12 \frac {\ud}{\ud t} \sum_{\eta \in J} |\mathscr{F} R_h^2 \Lambda^{-2\alpha} \tht(\eta)|^2 \leq -\int_{\Omega} (v \cdot \nabla)\tht R_h^{4} \Lambda^{-4\alpha}\tht \,\ud x - \sum_{\eta \in J} \frac{|\tilde{n}|^{4}}{|\eta|^{4(1+\alpha)}} \mathscr{F}_b \tht(\eta)\mathscr{F}_b v_d(\eta)
\end{gather*}
respectively, it holds
\begin{gather*}
\frac 12 \frac {\ud}{\ud t} (\|R_h \Lambda^{-2\alpha} v_d\|_{L^2}^2 + \| R_h^2 \Lambda^{-2\alpha} \tht \|_{L^2}^2) \\
\leq -\| R_h \Lambda^{-\alpha} v_d\|_{L^2}^2 + \|(v \cdot \nabla)v\|_{L^2} \| R_h \Lambda^{-\alpha} v_d \|_{L^2} -\int_{\Omega} (v \cdot \nabla)\tht R_h^{4} \Lambda^{-4\alpha} \tht \,\ud x.
\end{gather*}
We can infer from \eqref{df_vd} and \eqref{df_tht} that
\begin{gather*}
-\frac {\ud}{\ud t} \int_{\Omega} v_d R_h^{4} \Lambda^{-6\alpha}\tht \,\ud x = -\int_{\Omega} \partial_tv_d R_h^{4} \Lambda^{-6\alpha}\tht \,\ud x - \int_{\Omega} \partial_t\tht R_h^{4} \Lambda^{-6\alpha}v_d \,\ud x \\
\leq \|(v \cdot \nabla)v\|_{L^2} \| R_h^{4} \Lambda^{-6\alpha} \tht \|_{L^2} + \int_{\Omega} (v \cdot \nabla)\tht R_h^{4} \Lambda^{-6\alpha} v_d \,\ud x - \frac 12 \| R_h^{3} \Lambda^{-3\alpha} \tht \|_{L^2}^2 + \frac 32 \| R_h \Lambda^{-\alpha} v_d \|_{L^2}^2.
\end{gather*}
Combining the above, we have
\begin{gather*}
\frac 12 \frac {\ud}{\ud t} \left(\|R_h \Lambda^{-2\alpha}v_d\|_{L^2}^2 + \| R_h^2 \Lambda^{-2\alpha} \tht \|_{L^2}^2 - \int_{\Omega} v_d R_h^{4} \Lambda^{-6\alpha}\tht \,\ud x \right) \leq -\frac 14 (\| R_h \Lambda^{-\alpha} v_d\|_{L^2}^2 + \| R_h^{3}\Lambda^{-3\alpha} \tht \|_{L^2}^2) \\
+ \| v \|_{L^2} \| \nabla v \|_{L^{\infty}} (\| R_h \Lambda^{-\alpha} v_d \|_{L^2} + \| R_h^{3} \Lambda^{-3\alpha}\tht \|_{L^2}) -\int_{\Omega} (v \cdot \nabla)\tht (R_h^{4} \Lambda^{-4\alpha}\tht - \frac 12 R_h^{4} \Lambda^{-6\alpha}v_d) \,\ud x.
\end{gather*}
We estimate the last integral with $(v \cdot \nabla)\tht = (v_h \cdot \nabla_h) \tht + v_d \partial_d \tht$. H\"older's inequality implies
\begin{gather*}
\left| -\int_{\Omega} (v_h \cdot \nabla_h)\tht (R_h^{4} \Lambda^{-4\alpha}\tht - \frac 12 R_h^{4} \Lambda^{-6\alpha}v_d) \,\ud x\right| \\
\leq C\| v \|_{L^2} \| \nabla_h \tht \|_{L^p}(\| R_h^4 \Lambda^{-6\alpha} v_d \|_{L^q} + \| R_h^{3} \Lambda^{-4\alpha}\tht \|_{L^q}),
\end{gather*}
where $\frac 1p + \frac 1q = \frac 12$. We take $\frac 1q = \frac 12 - \frac \alpha d$. Then for $\epsilon \in (0,1)$ with $m>2+ \frac d2 + \alpha+2\epsilon$, we can see
$$\| \nabla_h \tht \|_{L^p} \leq C\| R_h \tht \|_{\dot{H}^{1 + \frac d2 + \epsilon}} \leq C\| R_h \tht \|_{\dot{H}^{m-1-\epsilon}}, \qquad \alpha = 0,$$ and $$\| \nabla_h \tht \|_{L^p} \leq C \| R_h \tht \|_{\dot{H}^{\frac d2 + \epsilon}} \leq C\| R_h \tht \|_{\dot{H}^{m-3-\epsilon}}, \qquad \alpha=1,$$ together with $\| R_h^4 \Lambda^{-6\alpha} v_d \|_{L^q} + \| R_h^{3} \Lambda^{-4\alpha}\tht \|_{L^q} \leq C\| R_h \Lambda^{-\alpha} v_d \|_{L^2} + C\| R_h^{3} \Lambda^{-3\alpha}\tht \|_{L^2}$, we have
\begin{gather*}
\left| -\int_{\Omega} (v_h \cdot \nabla_h)\tht (R_h^{4} \Lambda^{-4\alpha}\tht - \frac 12 R_h^{4} \Lambda^{-6\alpha}v_d) \,\ud x\right| \\
\leq C\| v \|_{L^2} \| R_h \Lambda^{-\alpha} \tht \|_{\dot{H}^{m-1-\alpha - \epsilon}}(\| R_h \Lambda^{-\alpha} v_d \|_{L^2} + \| R_h^{3} \Lambda^{-3\alpha}\tht \|_{L^2}).
\end{gather*}
On the other hand, it holds
\begin{gather*}
-\int_{\Omega} v_d \partial_d\tht (R_h^{4} \Lambda^{-4\alpha}\tht - \frac 12 R_h^{4} \Lambda^{-6\alpha}v_d) \,\ud x = -\int_{\Omega} \nabla_h (v_d \partial_d \tht) \cdot \nabla_h(-\Delta)^{-1} (R_h^2 \Lambda^{-4\alpha} \tht - \frac 12 R_h^2 \Lambda^{-6\alpha} v_d) \,\ud x \\
= -\int_{\Omega} (\nabla_hv_d \partial_d \tht) \cdot \nabla_h(-\Delta)^{-1} (R_h^2 \Lambda^{-4\alpha} \tht - \frac 12 R_h^2 \Lambda^{-6\alpha} v_d) \,\ud x \\
- \int_{\Omega} (v_d \partial_d \nabla_h\tht) \cdot \nabla_h (-\Delta)^{-1} (R_h^2 \Lambda^{-4\alpha} \tht - \frac 12 R_h^2 \Lambda^{-6\alpha} v_d) \,\ud x.
\end{gather*}
The second integral on the right-hand side is bounded by 
\begin{gather*}
\left| - \int_{\Omega} (v_d \partial_d \nabla_h\tht) \cdot \nabla_h (-\Delta)^{-1} (R_h^2 \Lambda^{-4\alpha} \tht - \frac 12 R_h^2 \Lambda^{-6\alpha} v_d) \,\ud x \right| \\
\leq C \| v_d \|_{L^2} \| \partial_d \nabla_h \tht \|_{L^{\infty}} ( \| R_h^3\Lambda^{-3\alpha} \tht \|_{L^2} + \| R_h \Lambda^{-\alpha} v_d \|_{L^2}).
\end{gather*}
We note
\begin{gather*}
\left| -\int_{\Omega} (\nabla_hv_d \partial_d \tht) \cdot \nabla_h(-\Delta)^{-1} (R_h^2 \Lambda^{-4\alpha} \tht - \frac 12 R_h^2 \Lambda^{-6\alpha} v_d) \,\ud x \right| \\
= \left| \int_{\Omega} (\nabla_h \Delta (-\Delta)^{-1} v_d \partial_d \tht) \cdot \nabla_h(-\Delta)^{-1} (R_h^2 \Lambda^{-4\alpha} \tht - \frac 12 R_h^2 \Lambda^{-6\alpha} v_d) \,\ud x \right|.
\end{gather*}
When $\alpha = 0$, with the integration by parts, it holds
\begin{gather*}
\left| \int_{\Omega} (\nabla_h \Delta (-\Delta)^{-1} v_d \partial_d \tht) \cdot \nabla_h(-\Delta)^{-1} (R_h^2 \tht - \frac 12 R_h^2 v_d) \,\ud x \right| \\
\leq \| R_h v_d \|_{L^2} (\| \partial_d \nabla \tht \|_{L^{\infty}} + \| \partial_d \tht \|_{L^{\infty}}) ( \| R_h^3  \tht \|_{L^2} + \| R_h v_d \|_{L^2}) \\
\leq C \| \tht \|_{H^m} ( \| R_h^3 \Lambda^{-3\alpha} \tht \|_{L^2}^2 + \| R_h \Lambda^{-\alpha} v_d \|_{L^2}^2).
\end{gather*}
For $\alpha = 1$, we have
\begin{gather*}
\left| \int_{\Omega} (\nabla_h \Delta (-\Delta)^{-1} v_d \partial_d \tht) \cdot \nabla_h(-\Delta)^{-1} (R_h^2 \Lambda^{-4} \tht - \frac 12 R_h^2 \Lambda^{-6} v_d) \,\ud x \right| \\
\leq C \| R_h \Lambda^{-1} v_d \|_{L^2} (\| \Delta \partial_d \tht \|_{L^{\infty}} + \| \nabla \partial_d \tht \|_{L^{\infty}}+ \| \partial_d \tht \|_{L^{\infty}}) ( \| R_h^3 \Lambda^{-3} \tht \|_{L^2} + \| R_h \Lambda^{-1} v_d \|_{L^2}) \\
\leq C \| \tht \|_{H^m} ( \| R_h^3 \Lambda^{-3\alpha} \tht \|_{L^2}^2 + \| R_h \Lambda^{-\alpha} v_d \|_{L^2}^2).
\end{gather*}
Therefore,
\begin{gather*}
\left| -\int_{\Omega} (v \cdot \nabla)\tht (R_h^{4} \Lambda^{-4\alpha}\tht - \frac 12 R_h^{4} \Lambda^{-6\alpha}v_d) \,\ud x \right| \\
\leq C\| v \|_{L^2} \| R_h \Lambda^{-\alpha} \tht \|_{\dot{H}^{m-1-\alpha - \epsilon}}(\| R_h \Lambda^{-\alpha} v_d \|_{L^2} + \| R_h^{3} \Lambda^{-3\alpha}\tht \|_{L^2}) \\
+ C \| v_d \|_{L^2} \| \partial_d \nabla_h \tht \|_{L^{\infty}} ( \| R_h^3\Lambda^{-3\alpha} \tht \|_{L^2} + \| R_h \Lambda^{-\alpha} v_d \|_{L^2}) + C \| \tht \|_{H^m} ( \| R_h^3 \Lambda^{-3\alpha} \tht \|_{L^2}^2 + \| R_h \Lambda^{-\alpha} v_d \|_{L^2}^2).
\end{gather*}
With Young's inequality and \eqref{sm_ass} we can have
\begin{gather*}
\frac 12 \frac {\ud}{\ud t} \left(\|R_h \Lambda^{-2\alpha}v_d\|_{L^2}^2 + \| R_h^2 \Lambda^{-2\alpha} \tht \|_{L^2}^2 - \int_{\Omega} v_d R_h^{4} \Lambda^{-6\alpha}\tht \,\ud x \right) \\
\leq -(\frac 14 - C\| \tht \|_{H^m}) (\| R_h \Lambda^{-\alpha} v_d\|_{L^2}^2 + \| R_h^{3}\Lambda^{-3\alpha} \tht \|_{L^2}^2) \\
+ C\| v \|_{L^2} (\| v \|_{\dot{H}^{m-1-\alpha-\epsilon}} + \| R_h \Lambda^{-\alpha} \tht \|_{\dot{H}^{m-1-\alpha - \epsilon}}) (\| R_h \Lambda^{-\alpha} v_d \|_{L^2} + \| R_h^{3} \Lambda^{-3\alpha}\tht \|_{L^2}) \\
+ C \| v_d \|_{L^2} \| \partial_d \nabla_h \tht \|_{L^{\infty}} ( \| R_h^3\Lambda^{-3\alpha} \tht \|_{L^2} + \| R_h \Lambda^{-\alpha} v_d \|_{L^2}) \\
\leq -\frac 18 (\| R_h \Lambda^{-\alpha} v_d\|_{L^2}^2 + \| R_h^{3}\Lambda^{-3\alpha} \tht \|_{L^2}^2) + C\| v \|_{L^2}^2 (\| v \|_{\dot{H}^{m-1-\alpha-\epsilon}}^2 + \| R_h \Lambda^{-\alpha} \tht \|_{\dot{H}^{m-1-\alpha- \epsilon}}^2) \\
+ C \| v_d \|_{L^2}^2 \| R_h \tht \|_{\dot{H}^{m-\alpha}}^2.
\end{gather*}
Let $M\geq 1$ which will be specified later. Since
\begin{equation*}
\begin{aligned}
\frac 1M \| R_h^2 \Lambda^{-2\alpha} \tht \|_{L^2}^2 - \| R_h^{3} \Lambda^{-3\alpha} \tht \|_{L^2}^2 &= \sum_{|\tilde{n}| \neq 0} \left( \frac 1M - \frac {|\tilde{n}|^2}{|\eta|^{2(1+\alpha)}} \right) |\mathscr{F} R_h^2 \Lambda^{-2\alpha}  \tht(\eta)|^2 \\
&\leq \frac 1M \sum_{\frac {|\tilde{n}|^2}{|\eta|^{2(1+\alpha)}} \leq \frac 1M, |\tilde{n}| \neq 0} |\mathscr{F} R_h^2 \Lambda^{-2\alpha} \tht(\eta)|^2 \\
&\leq \frac 1{M^{2+\frac {m}{1+\alpha}}} \| R_h \tht \|_{\dot{H}^{m-\alpha}}^2
\end{aligned}
\end{equation*}
and
\begin{equation}\label{smp_ineq_3}
\left| -\int_{\Omega} v_d R_h^4 \Lambda^{-6\alpha} \tht \,\ud x \right| \leq \| R_h\Lambda^{-3\alpha} v_d \|_{L^2} \| R_h^3 \Lambda^{-3\alpha} \tht \|_{L^2} \leq \frac 12 \| R_h \Lambda^{-\alpha} v_d \|_{L^2}^2 + \frac 12 \| R_h^3 \Lambda^{-3\alpha} \tht \|_{L^2}^2,
\end{equation}
we have
\begin{gather*}
-\frac 18 \left( \| R_h^{3} \Lambda^{-3\alpha} \tht \|_{L^2}^2 + \| R_h \Lambda^{-\alpha} v_d \|_{L^2}^2 \right) \\
\leq -\frac 1{16M} \left(\| R_h^2 \Lambda^{-2\alpha} \tht \|_{L^2}^2 + \| R_h \Lambda^{-\alpha} v_d \|_{L^2}^2 -\int_{\Omega} v_d R_h^4 \Lambda^{-6\alpha} \tht \,\ud x \right) + \frac 1{8M^{2+\frac {m}{1+\alpha}}} \| R_h \tht \|_{\dot{H}^{m-\alpha}}^2.
\end{gather*}
Thus,
\begin{gather*}
\frac 12 \frac {\ud}{\ud t} \left(\|R_h \Lambda^{-2\alpha}v_d\|_{L^2}^2 + \| R_h^2 \Lambda^{-2\alpha} \tht \|_{L^2}^2 - \int_{\Omega} v_d R_h^{4} \Lambda^{-6\alpha}\tht \,\ud x \right) \\
\leq -\frac 1{16M} \left(\| R_h^2 \Lambda^{-2\alpha} \tht \|_{L^2}^2 + \| R_h \Lambda^{-2\alpha} v_d \|_{L^2}^2 -\int_{\Omega} v_d R_h^4 \Lambda^{-6\alpha} \tht \,\ud x \right) + \frac 1{8M^{2+\frac {m}{1+\alpha}}} \| R_h \tht \|_{\dot{H}^{m-\alpha}}^2.
 \\
 + C\| v \|_{L^2}^2 (\| v \|_{\dot{H}^{m-1-\alpha- \epsilon}}^2 + \| R_h \Lambda^{-\alpha} \tht \|_{\dot{H}^{m-1-\alpha- \epsilon}}^2) + C \| v_d \|_{L^2}^2 \| R_h \tht \|_{\dot{H}^{m-\alpha}}^2.
\end{gather*}
We take $M = 1+\frac t{8(2+\frac{m}{1+\alpha})}$ and multiply $2M^{2+\frac {m}{1+\alpha}}$ both sides. Then,
\begin{gather*}
\frac {\ud}{\ud t} \left( (1+\frac t{8(2+\frac{m}{1+\alpha})})^{2+\frac {m}{1+\alpha}} (\|R_h \Lambda^{-2\alpha}v_d\|_{L^2}^2 + \| R_h^2 \Lambda^{-2\alpha} \tht \|_{L^2}^2 - \int_{\Omega} v_d R_h^{4} \Lambda^{-6\alpha}\tht \,\ud x) \right) \leq C \| R_h \tht \|_{\dot{H}^{m-\alpha}}^2 \\
+ C(1+\frac t{8(2+\frac{m}{1+\alpha})})^{1+\frac {m}{1+\alpha}} \| v \|_{L^2}^2 (1+\frac t{8(2+\frac{m}{1+\alpha})}) (\| v \|_{\dot{H}^{m-1-\alpha- \epsilon}}^2 + \| R_h \Lambda^{-\alpha} \tht \|_{\dot{H}^{m-1-\alpha- \epsilon}}^2) \\
+ C (1+\frac t{8(2+\frac{m}{1+\alpha})})^{2+\frac {m}{1+\alpha}} \| v_d \|_{L^2}^2 \| R_h \tht \|_{\dot{H}^{m-\alpha}}^2.
\end{gather*}
Since the interpolation inequality implies $$\| v \|_{\dot{H}^{m-1-\alpha- \epsilon}} + \| R_h \Lambda^{-\alpha} \tht \|_{\dot{H}^{m-1-\alpha- \epsilon}} \leq \| v \|_{L^2}^{\frac {1+\alpha+\epsilon}m} \| v \|_{\dot{H}^m}^{1-\frac {1+\alpha+\epsilon}m} + \| R_h \Lambda^{-\alpha} \tht \|_{L^2}^{\frac{1+\alpha+\epsilon}m} \| R_h \Lambda^{-\alpha} \tht \|_{\dot{H}^m}^{1-\frac {1+\alpha+\epsilon}m},$$ we have from \eqref{v_L2}
\begin{gather*}
C(1+\frac t{8(2+\frac{m}{1+\alpha})})^{1+\frac {m}{1+\alpha}} \| v \|_{L^2}^2 (1+\frac t{8(2+\frac{m}{1+\alpha})}) (\| v \|_{\dot{H}^{m-1-\alpha- \epsilon}}^2 + \| R_h \Lambda^{-\alpha} \tht \|_{\dot{H}^{m-1-\alpha- \epsilon}}^2) \\
\leq C (1+t)^{1-\frac{1+\alpha+\epsilon}m (1+\frac m{1+\alpha})} (\| v \|_{\dot{H}^{m}}^2 + \| R_h \Lambda^{-\alpha} \tht \|_{\dot{H}^{m-\alpha}}^2)^{1-\frac {1+\alpha+\epsilon}m} \\
\leq C(1+t)^{-\frac{1+\alpha+\epsilon}m - \frac \epsilon{1+\alpha}} (\| v \|_{\dot{H}^{m}}^2 + \| R_h \Lambda^{-\alpha} \tht \|_{\dot{H}^{m-\alpha}}^2)^{1-\frac {1+\alpha+\epsilon}m} .
\end{gather*}
By \eqref{vd_Rh_est}, it holds $$(1+\frac t{8(2+\frac{m}{1+\alpha})})^{2+\frac {m}{1+\alpha}} \| v_d \|_{L^2}^2 \leq C + C \sup_{\tau \in [0,t]} (1+\frac \tau{8(2+\frac{m}{1+\alpha})})^{2+\frac {m}{1+\alpha}} \| R_h^2 \Lambda^{-2\alpha} \tht(\tau) \|_{L^2}^2.$$
Then, we deduce that
\begin{gather*}
\frac {\ud}{\ud t} \left( (1+\frac t{8(2+\frac{m}{1+\alpha})})^{2+\frac {m}{1+\alpha}} (\|R_h \Lambda^{-2\alpha}v_d\|_{L^2}^2 + \| R_h^2 \Lambda^{-2\alpha} \tht \|_{L^2}^2 - \int_{\Omega} v_d R_h^{4} \Lambda^{-6\alpha}\tht \,\ud x) \right) \\
\leq C \| R_h \tht \|_{\dot{H}^{m-\alpha}}^2 + C(1+t)^{-\frac{1+\alpha+\epsilon}m - \frac \epsilon{1+\alpha}} (\| v \|_{\dot{H}^{m}}^2 + \| R_h \Lambda^{-\alpha} \tht \|_{\dot{H}^{m-\alpha}}^2)^{1-\frac {1+\alpha+\epsilon}m} \\
+C\left( \sup_{\tau \in [0,t]} (1+\frac \tau{8(2+\frac{m}{1+\alpha})})^{2+\frac {m}{1+\alpha}} (\|R_h \Lambda^{-2\alpha}v_d\|_{L^2}^2 + \| R_h^2 \Lambda^{-2\alpha} \tht \|_{L^2}^2) \right) \| R_h \tht \|_{\dot{H}^{m-\alpha}}^2.
\end{gather*}
Since we can verify $$\int_0^\infty \left( C \| R_h \tht \|_{\dot{H}^{m-\alpha}}^2 + C(1+t)^{-\frac{1+\alpha+\epsilon}m - \frac \epsilon{1+\alpha}} (\| v \|_{\dot{H}^{m}}^2 + \| R_h \Lambda^{-\alpha} \tht \|_{\dot{H}^{m-\alpha}}^2)^{1-\frac {1+\alpha+\epsilon}m} \right) \,\ud t \leq C,$$ using Gr\"onwall's inequality and \eqref{smp_ineq_3}, \eqref{Rh_tht_est} is obtained. Then, \eqref{vd_L2} follows from \eqref{vd_Rh_est}. This completes the proof.
\end{proof}

\subsection{Proof of Theorem~\ref{thm2}:Temporal decay part}
In this section, we completes the proof of Theorem~\ref{thm2}. For this purpose, we suppose \eqref{v_Hs} and \eqref{vd_Hs} hold true, which will be proved in the following proposition. From \eqref{v_L2} and \eqref{v_Hs}, we obtain $$(1+t)^{1+m-s} \| v(t) \|_{H^s}^2 \leq C.$$ On the other hand, \eqref{vd_L2} and \eqref{vd_Hs} imply $$(1+t)^{2+m-s} \| v_d(t) \|_{H^s}^2 \leq C.$$ Hence, it suffices to prove $$(1+t)^{m} \| \tht(t) - \sigma \|_{L^2}^2 \leq C$$ because of \eqref{sol_bdd_0}.  We recall \eqref{df_sgm} and use \eqref{tht_L2} to have
\begin{align*}
	\| \tht - \sigma \|_{L^2} &\leq \| \bar{\tht} \|_{L^2} + \left\| \int_{\bbT^{d-1}} \int_t^\infty \left( (v \cdot \nabla)\tht + v_d \right) \,\ud \tau \ud x_h
 \right\|_{L^2} \\
	&\leq C(1+t)^{-\frac m2} +\int_t^\infty \| (v \cdot \nabla)\tht + v_d \|_{L^2} \,\ud \tau.
\end{align*}
Since 
\begin{align*}
\| (v \cdot \nabla)\tht + v_d \|_{L^2} &\leq \| (v_h \cdot \nabla_h)\tht \|_{L^2} + \| v_d \partial_d \tht \|_{L^2} + \| v_d \|_{L^2} \\
&\leq C \| v \|_{L^2} \| R_h \tht \|_{H^{m}} + C\| v_d \|_{L^2} \| \tht \|_{H^m} + \| v_d \|_{L^2} \\
&\leq C (1+\tau)^{-(\frac 12 +\frac {m}{2})} \| R_h \tht \|_{H^{m}} + C (1+\tau)^{-(1+\frac m{2})}
\end{align*}
by \eqref{v_L2} and \eqref{vd_L2}, we can deduce $$\int_t^\infty \| (v \cdot \nabla)\tht + v_d \|_{L^2} \,\ud \tau \leq C \left\| (1+\tau)^{-(\frac 12 +\frac {m}{2})} \right\|_{L^2(t,\infty)} + C (1+t)^{-\frac m{2}} \leq C (1+t)^{-\frac m{2}}.$$ This completes the proof.

\begin{proposition}
	Let $d \in \mathbb{N}$ with $d \geq 2$ and $\alpha = 0$. Let $m \in \mathbb{N}$ with $m > 2+\frac d2$ and $(v, \theta)$ be a smooth global solution to \eqref{EQ} with \eqref{sol_bdd_0}. Suppose that \eqref{sm_ass} be satisfied. Then, there exists a constant $C > 0$ such that
	\begin{equation}\label{v_Hs}
		\| v(t) \|_{\dot{H}^{m}}^2 + \| R_h \tht(t) \|_{\dot{H}^{m}}^2 \leq C (1+t)^{-1}
	\end{equation}
	and
	\begin{equation}\label{vd_Hs}
		\| v_d(t) \|_{\dot{H}^{m}}^2 + \| R_h v(t) \|_{\dot{H}^{m}}^2 + \| R_h^2 \tht(t) \|_{\dot{H}^{m}}^2 \leq C (1+t)^{-2}.
	\end{equation}
\end{proposition}
\begin{proof}
From the $v$ equations in \eqref{EQ}, it follows $$\frac 12 \frac {\ud}{\ud t} \| v \|_{\dot{H}^m}^2 + \| v \|_{\dot{H}^m}^2 \leq C\| \nabla v \|_{L^\infty} \| v \|_{\dot{H}^m}^2 - \sum_{\eta \in J} |\eta|^{2m} \mathscr{F}_b \tht(\eta)\mathscr{F}_b v_d(\eta).$$ Using \eqref{vd_v} gives
\begin{align*}
	\left| - \sum_{\eta \in J} |\eta|^{2m} \mathscr{F} \tht(\eta)\mathscr{F} v_d(\eta) \right| &\leq C\left( \sum_{\eta \in J} |\eta|^{2m} |\mathscr{F}_b R_h \tht(\eta)|^2 \right)^{\frac 12} \left( \sum_{\eta \in J} |\eta|^{2m} |\mathscr{F} v(\eta)|^2 \right)^{\frac 12} \\
	&\leq \frac 14 \| v \|_{\dot{H}^m}^2 + C \| R_h \tht \|_{\dot{H}^{m}}^2.
\end{align*}
By \eqref{sm_ass} with \eqref{sol_bdd_0}, we have 
\begin{align*}
\frac 12 \frac {\ud}{\ud t} \| v \|_{\dot{H}^m}^2 &\leq - (\frac 34 -C \| (v_0,\tht_0) \|_{H^m}^2) \| v \|_{\dot{H}^{m}}^2  + C\| R_h \tht \|_{\dot{H}^{m}}^2 \leq -\frac 12 \| v \|_{\dot{H}^m}^2  + C\| R_h \tht \|_{\dot{H}^{m}}^2.
\end{align*}
Then, applying Duhamel's principle shows
$$\| v(t) \|_{\dot{H}^m}^2 \leq e^{-t} \| v_0 \|_{\dot{H}^m}^2 + C \int_0^t e^{-(t-\tau)} \| R_h \tht \|_{\dot{H}^{m}}^2 \,\ud \tau,$$ thus,
\begin{equation}\label{v_Hs_est}
\| v(t) \|_{\dot{H}^m}^2 \leq C(1+t)^{-1} \left( \| v_0 \|_{\dot{H}^m}^2 + \sup_{\tau \in [0,t]} (1+\tau) \| R_h \tht(\tau) \|_{\dot{H}^{m}}^2 \right).
\end{equation}
Since \eqref{vd_v} implies $\| v_d (t) \|_{\dot{H}^m}^2 \leq C\| R_h v(t) \|_{\dot{H}^m}^2$, we can similarly obtain 
\begin{equation}\label{vd_Hs_est}
\| v_d (t) \|_{\dot{H}^m}^2 + \| R_h v(t) \|_{\dot{H}^m}^2 \leq C(1+t)^{-2} \left( \| v_0 \|_{\dot{H}^m}^2 + \sup_{\tau \in [0,t]} (1+\tau)^2 \| R_h^2 \tht(\tau) \|_{\dot{H}^{m}}^2 \right).
\end{equation}
We omit the details.

Now, we show that
\begin{equation}\label{claim}
\| R_h^2 \tht(t) \|_{\dot{H}^{m}}^2 \leq C(1+t)^{-2}.
\end{equation}
Since this implies  $$\| R_h \tht(t) \|_{\dot{H}^{m}}^2 \leq \| R_h^2 \tht(t) \|_{\dot{H}^{m}} \| \tht(t) \|_{\dot{H}^{m}} \leq C(1+t)^{-1},$$ \eqref{v_Hs} and \eqref{vd_Hs} follow by \eqref{v_Hs_est} and \eqref{vd_Hs_est} respectively. Since $H^m(\Omega)$ is a Banach algebra, we can show from \eqref{df_vd} that
\begin{gather*}
\frac 12 \frac {\ud}{\ud t} \| R_h v_d \|_{\dot{H}^m}^2 + \| R_h v_d \|_{\dot{H}^m}^2 \leq C\| R_h v \|_{\dot{H}^m} \| v \|_{\dot{H}^m} \| R_h v_d \|_{\dot{H}^m} + \sum_{|\tilde{n}| \neq 0} |\tilde{n}|^4|\eta|^{2m-4} \mathscr{F}_b \tht(\eta)\mathscr{F}_b v_d(\eta).
\end{gather*}
From \eqref{df_tht},
\begin{gather*}
\frac 12 \frac {\ud}{\ud t} \| R_h^2 \tht \|_{\dot{H}^m}^2 \leq -\sum_{|\gamma| = m-2} \int_{\Omega} \partial^\gamma \partial_h^2 (v \cdot \nabla) \tht \cdot \partial^\gamma\partial_h^2 \tht \,\ud x - \sum_{|\tilde{n}| \neq 0} |\tilde{n}|^4|\eta|^{2m-4} \mathscr{F}_b \tht(\eta)\mathscr{F}_b v_d(\eta).
\end{gather*}
Thus,
\begin{gather*}
\frac 12 \frac {\ud}{\ud t} \left( \| R_h v_d \|_{\dot{H}^m}^2 + \| R_h^2 \tht \|_{\dot{H}^m}^2 \right) + \| R_h v_d \|_{\dot{H}^{m}}^2 \\
\leq C\| R_h v \|_{\dot{H}^m} \| v \|_{\dot{H}^m} \| R_h v_d \|_{\dot{H}^m} -\sum_{|\gamma| = m-2} \int_{\Omega} \partial^\gamma \partial_h^2 (v \cdot \nabla) \tht \cdot \partial^\gamma\partial_h^2 \tht \,\ud x .
\end{gather*}

We note that
\begin{gather*}
\sum_{|\gamma| = m-2} \int_{\Omega} \partial^\gamma \partial_h^2 (v \cdot \nabla) \tht \cdot \partial^\gamma\partial_h^2 \tht \,\ud x = \sum_{|\gamma| = m-2} \left( K_1 + K_2 + K_3 \right),
\end{gather*}
where
\begin{align*}
K_1 &= \int_{\Omega} \partial^\gamma (\partial_h^2 v \cdot \nabla) \tht \cdot \partial^\gamma\partial_h^2 \tht \,\ud x, \\
K_2 &= \int_{\Omega} \partial^\gamma (\partial_h v \cdot \nabla) \partial_h \tht \cdot \partial^\gamma\partial_h^2 \tht \,\ud x, \\
K_3 &:= \int_{\Omega} \partial^\gamma (v \cdot \nabla) \partial_h^2 \tht \cdot \partial^\gamma\partial_h^2 \tht \,\ud x .
\end{align*}
The integration by parts and $(v \cdot \nabla) \tht = (v_h \cdot \nabla_h) \tht + v_d \partial_d \tht$ show $$K_1 + K_2 = - \int_{\Omega} \partial^\gamma (\partial_h v_h \cdot \nabla_h) \tht \cdot \partial^\gamma\partial_h^3 \tht \,\ud x - \int_{\Omega} \partial^\gamma (\partial_h v_d \partial_d \tht) \cdot \partial^\gamma\partial_h^3 \tht \,\ud x. $$ Again using the integration by parts with the continuous embedding $L^{\infty}(\Omega) \hookrightarrow H^{m-1}(\Omega)$, we obtain $$|K_1 + K_2| \leq C\| R_h v_h \|_{\dot{H}^m} \| R_h \tht \|_{\dot{H}^m} \| R_h^3 \tht \|_{\dot{H}^m} + C\| R_h v_d \|_{\dot{H}^m} \| \tht \|_{H^m} \| R_h^3 \tht \|_{\dot{H}^m}. $$  On the other hand, due to the cancellation property, we can have $$|K_3| \leq C\| v_h \|_{\dot{H}^m} \| R_h^3 \tht \|_{\dot{H}^m} \| R_h^2 \tht \|_{\dot{H}^m} + C \| \nabla v_d \|_{L^{\infty}} \| R_h^2 \tht \|_{\dot{H}^m}^2 + \int_{\Omega} \partial_d^{m-4} (\partial_d^2 v_d \partial_d \partial_h^2 \tht) \cdot \partial_d^{m-2} \partial_h^2 \tht \,\ud x.$$ The divergence-free condition and the integration by parts imply
\begin{align*}
\left| \int_{\Omega} \partial_d^{m-4} (\partial_d^2 v_d \partial_d \partial_h^2 \tht) \cdot \partial_d^{m-2} \partial_h^2 \tht \,\ud x \right| &= \left| \int_{\Omega} \partial_d^{m-4} (\partial_d \nabla_h \cdot v_h \partial_d \partial_h^2 \tht) \cdot \partial_d^{m-2} \partial_h^2 \tht \,\ud x \right| \\
&\leq C\| v \|_{\dot{H}^m} \| R_h^3 \tht \|_{\dot{H}^m} \| R_h^2 \tht \|_{\dot{H}^m}.
\end{align*}
Thus, $$|K_3| \leq C\| v \|_{\dot{H}^m} \| R_h^3 \tht \|_{\dot{H}^m} \| R_h^2 \tht \|_{\dot{H}^m} + C \| \nabla v_d \|_{L^{\infty}} \| R_h^2 \tht \|_{\dot{H}^m}^2.$$ From the above estimates, we deduce
\begin{gather*}
\frac 12 \frac {\ud}{\ud t} \left( \| R_h v_d \|_{\dot{H}^m}^2 + \| R_h^2 \tht \|_{\dot{H}^m}^2 \right) + \| R_h v_d \|_{\dot{H}^{m}}^2 \\
\leq C(\| R_h v \|_{\dot{H}^m} + \| R_h^2 \tht \|_{\dot{H}^m})(\| v \|_{\dot{H}^m} + \| R_h \tht \|_{\dot{H}^m}) (\| R_h v_d \|_{\dot{H}^m} + \| R_h^3 \tht \|_{\dot{H}^m}) \\
+ C\| R_h v_d \|_{\dot{H}^m} \| \tht \|_{H^m} \| R_h^3 \tht \|_{\dot{H}^m} + C \| \nabla v_d \|_{L^{\infty}} \| R_h^2 \tht \|_{\dot{H}^m}^2.
\end{gather*} 
On the other hand, using \eqref{df_vd} and 
\begin{equation}\label{Leray_est}
	-\Delta \langle \mathbb{P} (v \cdot \nabla)v , e_d \rangle = -\Delta (v \cdot \nabla)v_d + \partial_d \nabla \cdot (v \cdot \nabla)v,
\end{equation} we have
\begin{align*}
-\int_{\Omega} \partial_tv_d (-\Delta)^{m} R_h^4 \tht \,\ud x &= \int_{\Omega} (v \cdot \nabla)v_d (-\Delta)^{m} R_h^4 \tht \,\ud x - \int_{\Omega} \partial_d \nabla \cdot ((v \cdot \nabla)v) (-\Delta)^{m-1} R_h^4 \tht \,\ud x \\
&\hphantom{\qquad\qquad}- \int_{\Omega} v_d (-\Delta)^{m} R_h^4 \tht \,\ud x - \| R_h^3 \tht \|_{\dot{H}^{m}}^2 \\
&\leq \int_{\Omega} (v \cdot \nabla)v_d (-\Delta)^m R_h^4 \tht \,\ud x + C\| R_h v \|_{\dot{H}^m} \| v \|_{\dot{H}^m} \| R_h^3 \tht \|_{\dot{H}^m} \\
&\hphantom{\qquad\qquad}+ \frac 12 \| R_hv_d \|_{\dot{H}^{m}}^2 - \frac 12\| R_h^3 \tht \|_{\dot{H}^{m}}^2.
\end{align*}
Since \eqref{df_tht} yields
\begin{equation*}
-\int_{\Omega} \partial_t\tht (-\Delta)^{m} R_h^4 v_d \,\ud x \leq \int_{\Omega} (v \cdot \nabla)\tht (-\Delta)^{m} R_h^4 v_d \,\ud x + \| R_h^2 v_d \|_{\dot{H}^{m}}^2,
\end{equation*}
we have
\begin{gather*}
-\frac {\ud}{\ud t} \int_{\Omega} v_d (-\Delta)^{m} R_h^4 \tht \,\ud x \leq \int_{\Omega} (v \cdot \nabla)v_d (-\Delta)^{m} R_h^4 \tht \,\ud x + \int_{\Omega} (v \cdot \nabla)\tht (-\Delta)^{m} R_h^4 v_d \,\ud x \\
+ C\| R_h v \|_{\dot{H}^m} \| v \|_{\dot{H}^s} \| R_h^3 \tht \|_{\dot{H}^m} - \frac 12 \| R_h^3 \tht \|_{\dot{H}^{m}}^2+ \frac 32 \| R_h v_d \|_{\dot{H}^{m}}^2.
\end{gather*}
We note that $$\int_{\Omega} (v \cdot \nabla)v_d (-\Delta)^{m} R_h^4 \tht \,\ud x + \int_{\Omega} (v \cdot \nabla)\tht (-\Delta)^{m} R_h^4 v_d \,\ud x = \sum_{|\gamma| = m-2} \left( K_4 + K_5 + K_6 \right),$$ where
\begin{align*}
K_4 &:= \int_{\Omega} \partial^{\gamma}(\partial_h^2 v \cdot \nabla)v_d \partial^{\gamma} \partial_h^2 \tht \,\ud x + \int_{\Omega} \partial^{\gamma}( \partial_h^2 v \cdot \nabla)\tht \partial^{\gamma} \partial_h^2 v_d \,\ud x, \\
K_5 &:= \int_{\Omega} \partial^{\gamma}(\partial_h v \cdot \nabla) \partial_h v_d \partial^{\gamma} \partial_h^2 \tht \,\ud x + \int_{\Omega} \partial^{\gamma}( \partial_h v \cdot \nabla)\partial_h \tht \partial^{\gamma} \partial_h^2 v_d \,\ud x, \\
K_6 &:= \int_{\Omega} \partial^{\gamma}(v \cdot \nabla) \partial_h^2 v_d \partial^{\gamma} \partial_h^2 \tht \,\ud x + \int_{\Omega} \partial^{\gamma}(v \cdot \nabla)\partial_h^2 \tht \partial^{\gamma} \partial_h^2 v_d \,\ud x.
\end{align*}
We can see 
\begin{gather*}
|K_4| \leq C\| R_h^2 v_h \|_{\dot{H}^m} \| R_h v_d \|_{\dot{H}^m} \| R_h^2 \tht \|_{\dot{H}^m} + C\| R_h^2 v_d \|_{\dot{H}^m} \| v_d \|_{\dot{H}^m} \| R_h^2 \tht \|_{\dot{H}^m} \\
+ C \| R_h^2 v_h \|_{\dot{H}^m} \| R_h \tht \|_{\dot{H}^m} \| R_h^2 v_d \|_{\dot{H}^m} + C \| R_h^2 v_d \|_{\dot{H}^m}^2 \| \tht \|_{H^m}
\end{gather*}
and $$|K_5| \leq C\| R_h v \|_{\dot{H}^m} \| R_hv_d \|_{\dot{H}^m} \| R_h^2 \tht \|_{\dot{H}^m} + C \| R_h v \|_{\dot{H}^m} \| R_h \tht \|_{\dot{H}^m} \| R_h^2 v_d \|_{\dot{H}^m}.$$
Due to the cancellation property, we have
\begin{gather*}
|K_6| \leq C\| v \|_{\dot{H}^m} \| R_h^2 v_d \|_{\dot{H}^m} \| R_h^2 \tht \|_{\dot{H}^m}.
\end{gather*}
Collecting the above estimates gives
\begin{gather*}
-\frac {\ud}{\ud t} \int_{\Omega} v_d (-\Delta)^m R_h^4 \tht \,\ud x \leq - \frac 12 \| R_h^3 \tht \|_{\dot{H}^m}^2 + \frac 32 \| R_h v_d \|_{\dot{H}^{m}}^2 \\
+ C(\| R_h v \|_{\dot{H}^m} + \| R_h^2 \tht \|_{\dot{H}^m})(\| v \|_{\dot{H}^m} + \| R_h \tht \|_{\dot{H}^m})(\| R_h v_d \|_{\dot{H}^m} + \| R_h^3 \tht \|_{\dot{H}^m}) + C \| R_h v_d \|_{\dot{H}^m}^2 \| \tht \|_{H^m}.
\end{gather*}
Thus, we arrived at
\begin{gather*}
\frac 12 \frac {\ud}{\ud t} \left( \| R_h v_d \|_{\dot{H}^m}^2 + \| R_h^2 \tht \|_{\dot{H}^m}^2 - \int_{\Omega} v_d (-\Delta)^m R_h^4 \tht \,\ud x \right) \leq -(\frac 14 - C\| \tht \|_{H^m}) \left( \| R_h v_d \|_{\dot{H}^{m}}^2 + \| R_h^3 \tht \|_{\dot{H}^m}^2 \right) \\
+ C(\| R_h v \|_{\dot{H}^m} + \| R_h^2 \tht \|_{\dot{H}^m})(\| v \|_{\dot{H}^m} + \| R_h \tht \|_{\dot{H}^m}) (\| R_h v_d \|_{\dot{H}^m} + \| R_h^3 \tht \|_{\dot{H}^m}) + C \| \nabla v_d \|_{L^{\infty}} \| R_h^2 \tht \|_{\dot{H}^m}^2.
\end{gather*} 
By Young's inequality and \eqref{sm_ass}, it follows
\begin{gather*}
\frac 12 \frac {\ud}{\ud t} \left( \| R_h v_d \|_{\dot{H}^m}^2 + \| R_h^2 \tht \|_{\dot{H}^m}^2 - \int_{\Omega} v_d (-\Delta)^m R_h^4 \tht \,\ud x \right) \leq -\frac 18 \left( \|R_h v_d\|_{\dot{H}^m}^2 + \| R_h^3 \tht \|_{\dot{H}^m}^2 \right) \\
+ C(\| R_h v \|_{\dot{H}^m}^2 + \| R_h^2 \tht \|_{\dot{H}^m}^2)(\| v \|_{\dot{H}^m}^2 + \| R_h \tht \|_{\dot{H}^m}^2 + \| \nabla v_d \|_{L^{\infty}}).
\end{gather*}
We consider $M \geq 1$ which will be specified later. Since
\begin{equation*}
\begin{aligned}
\frac 1M \| R_h^2 \tht \|_{\dot{H}^m}^2 - \| R_h^3 \tht \|_{\dot{H}^m}^2 &= \sum_{|\tilde{n}| \neq 0} \left( \frac 1M - \frac {|\tilde{n}|^2}{|\eta|^2} \right) |\eta|^{2m} |\mathscr{F} R_h^2 \tht(\eta)|^2 \\
&\leq \frac 1M \sum_{\frac {|\tilde{n}|^2}{|\eta|^2} \leq \frac 1M, |\tilde{n}| \neq 0} |\eta|^{2m} |\mathscr{F} R_h^2 \tht(\eta)|^2 \\
&\leq \frac 1{M^2} \| R_h \tht \|_{\dot{H}^m}^2
\end{aligned}
\end{equation*}
and
\begin{equation}\label{smp_ineq'}
\left| \int_{\Omega} v_d (-\Delta)^m R_h^4 \tht \,\ud x \right| \leq \| R_h v_d \|_{\dot{H}^m} \| R_h^3 \tht \|_{\dot{H}^m} \leq \frac 12 \| R_h v_d \|_{\dot{H}^m}^2 + \frac 12 \| R_h^2 \tht \|_{\dot{H}^m}^2,
\end{equation}
it holds
\begin{gather*}
-\frac 18 \left( \|R_h v_d\|_{\dot{H}^m}^2 + \| R_h^3 \tht \|_{\dot{H}^m}^2 \right) \leq -\frac 1{8M} \left(\|R_h v_d\|_{\dot{H}^m}^2 + \| R_h^2 \tht \|_{\dot{H}^m}^2 \right) + \frac 1{16M} \int_{\Omega} v_d (-\Delta)^m R_h^4 \tht \,\ud x \\
+ \frac 1{8M^2} \| R_h \tht \|_{\dot{H}^m}^2 - \frac 1{16M} \int_{\Omega} v_d (-\Delta)^m R_h^4 \tht \,\ud x \\
\leq -\frac 1{16M} \left(\|R_h v_d\|_{\dot{H}^m}^2 + \| R_h^2 \tht \|_{\dot{H}^m}^2 - \int_{\Omega} v_d (-\Delta)^m R_h^4 \tht \,\ud x \right) + \frac 1{8M^2} \| R_h \tht \|_{\dot{H}^m}^2.
\end{gather*}
Hence,
\begin{gather*}
\frac 12 \frac {\ud}{\ud t} \left( \| R_h v_d \|_{\dot{H}^m}^2 + \| R_h^2 \tht \|_{\dot{H}^m}^2 - \int_{\Omega} v_d (-\Delta)^m R_h^4 \tht \,\ud x \right) \\
\leq -\frac 1{16M} \left(\|R_h v_d\|_{\dot{H}^m}^2 + \| R_h^2 \tht \|_{\dot{H}^m}^2 - \int_{\Omega} v_d (-\Delta)^m R_h^4 \tht \,\ud x \right) + \frac 1{8M^2} \| R_h \tht \|_{\dot{H}^m}^2 \\
 + C (\| R_h v \|_{\dot{H}^m}^2  + \| R_h^2 \tht \|_{\dot{H}^m}^2) (\| v \|_{\dot{H}^m}^2 +  \| R_h \tht \|_{\dot{H}^m}^2 + \| \nabla v_d \|_{L^{\infty}}).
\end{gather*}
Here, we take $M = 1+\frac t{16}$. Then multiplying the both sides by $2M^2$ and using \eqref{vd_Hs_est}, we have
\begin{gather*}
\frac {\ud}{\ud t} \left( (1+\frac t{16})^2 (\|R_h v_d\|_{\dot{H}^m}^2 + \| R_h^2 \tht \|_{\dot{H}^m}^2 - \int_{\Omega} v_d (-\Delta)^m R_h^4 \tht \,\ud x) \right) \\
\leq C\| R_h \tht \|_{\dot{H}^m}^2 + C \| v_0 \|_{\dot{H}^m}^2 (\| v \|_{\dot{H}^m}^2 +  \| R_h \tht \|_{\dot{H}^m}^2 + \| \nabla v_d \|_{L^{\infty}}) \\
+ C \sup_{\tau \in [0,t]} (1+\frac \tau{16})^2 \left( \|R_h v_d(\tau) \|_{\dot{H}^m}^2 + \| R_h^2 \tht(\tau) \|_{\dot{H}^m}^2 \right) (\| v \|_{\dot{H}^m}^2 +  \| R_h \tht \|_{\dot{H}^m}^2 + \| \nabla v_d \|_{L^{\infty}}).
\end{gather*}
We integrate it over time and use \eqref{smp_ineq'} and $$\int_0^{\infty} (\| v \|_{\dot{H}^m}^2 + \| R_h \tht \|_{\dot{H}^m}^2 + \| \nabla v_d \|_{L^{\infty}}) \,\ud t \leq C.$$ Then, for $$f(t) := \sup_{\tau \in [0,t]} (1+\frac \tau{16})^2 \left( \|R_h v_d(\tau) \|_{\dot{H}^m}^2 + \| R_h^2 \tht(\tau) \|_{\dot{H}^m}^2 \right),$$ it holds
\begin{gather*}
f(t) \leq C + \int_0^t f(\tau) (\| v \|_{\dot{H}^m}^2 +  \| R_h \tht \|_{\dot{H}^m}^2 + \| \nabla v_d \|_{L^{\infty}}) \,\ud \tau.
\end{gather*}
By Gr\"onwall's ineqaulity, we obtain \eqref{claim}. This completes the proof.
\end{proof}

\subsection{Proof of Theorem~\ref{thm1}:Temporal decay part}
Now, we finish the proof of Theorem~\ref{thm1} assuming \eqref{v_Hs_1} and \eqref{vd_Hs_1}, which is given in Proposition~\ref{prop_Hs_1}. We also provide Proposition~\ref{prop_rmk} for improved temporal estimates. From \eqref{v_L2} and \eqref{v_Hs_1}, we obtain $$(1+t)^{1+\frac {m-s}2} \| v(t) \|_{H^s}^2 \leq C.$$ On the other hand, \eqref{vd_L2} and \eqref{vd_Hs_1} imply $$(1+t)^{2+\frac{m-s}2} \| v_d(t) \|_{H^s}^2 \leq C.$$ It suffices to prove $$(1+t)^{\frac m2} \| \tht(t) - \sigma \|_{L^2}^2 \leq C$$ due to \eqref{sol_bdd_1}. Recalling \eqref{df_sgm}, we can estimate from \eqref{tht_L2} that 
\begin{align*}
	\| \tht - \sigma \|_{L^2} &\leq \| \bar{\tht} \|_{L^2} + \left\| \int_{\bbT^{d-1}} \int_t^\infty \left( (v \cdot \nabla)\tht + v_d \right) \,\ud \tau \ud x_h
 \right\|_{L^2} \\
	&\leq C(1+t)^{-\frac m2} +\int_t^\infty \| (v \cdot \nabla)\tht + v_d \|_{L^2} \,\ud \tau.
\end{align*}
Since
\begin{align*}
\| (v \cdot \nabla)\tht + v_d \|_{L^2} &\leq \| (v_h \cdot \nabla_h)\tht \|_{L^2} + \| v_d \partial_d \tht \|_{L^2} + \| v_d \|_{L^2} \\
&\leq C \| v \|_{L^2} \| R_h \tht \|_{H^{m-1}} + C\| v_d \|_{L^2} \| \tht \|_{H^m} + \| v_d \|_{L^2} \\
&\leq C (1+\tau)^{-(\frac 12 +\frac {m}{4})} \| R_h \tht \|_{H^{m-1}} + C (1+\tau)^{-(1+\frac m{4})}
\end{align*}
by \eqref{v_L2} and \eqref{vd_L2}, it holds $$\int_t^\infty \| (v \cdot \nabla)\tht + v_d \|_{L^2} \,\ud \tau \leq C \left\| (1+\tau)^{-(\frac 12 +\frac {m}{4})} \right\|_{L^2(t,\infty)} + C (1+t)^{-\frac m{4}} \leq C (1+t)^{-\frac m{4}}.$$ This completes the proof.
\begin{proposition}\label{prop_Hs_1}
	Let $d \in \mathbb{N}$ with $d \geq 2$ and $\alpha =1$. Let $m \in \mathbb{N}$ with $m > 3+\frac d2$ and $(v, \theta)$ be a smooth global solution to \eqref{EQ} with \eqref{sol_bdd_1}. Suppose that \eqref{sm_ass} be satisfied. Then, there exists a constant $C > 0$ such that
	\begin{equation}\label{v_Hs_1}
		\| v(t) \|_{\dot{H}^{m}}^2 + \| R_h \Lambda^{-1} \tht(t) \|_{\dot{H}^{m}}^2 \leq C (1+t)^{-1},
	\end{equation}
	and
	\begin{equation}\label{vd_Hs_1}
		\| v_d(t) \|_{\dot{H}^{m}}^2 + \| R_h \Lambda^{-1} v(t) \|_{\dot{H}^{m}}^2 + \| R_h^2 \Lambda^{-2}\tht(t) \|_{\dot{H}^{m}}^2 \leq C (1+t)^{-2}.
	\end{equation}
\end{proposition}
\begin{proof}
From the $v$ equations in \eqref{EQ}, it follows $$\frac 12 \frac {\ud}{\ud t} \| v \|_{\dot{H}^m}^2 + \| v \|_{\dot{H}^{m+1}}^2 \leq C\| \nabla v \|_{L^\infty} \| v \|_{\dot{H}^m}^2 - \sum_{|\tilde{n}| \neq 0} |\eta|^{2m} \mathscr{F}_b \tht(\eta)\mathscr{F}_b v_d(\eta).$$ Using \eqref{vd_v} gives
\begin{align*}
	\left| - \sum_{|\tilde{n}| \neq 0} |\eta|^{2m} \mathscr{F}_b \tht(\eta)\mathscr{F}_b v_d(\eta) \right| &\leq \left( \sum_{|\tilde{n}| \neq 0} |\eta|^{2(m-1)} |\mathscr{F}_b R_h \tht(\eta)|^2 \right)^{\frac 12} \left( \sum_{|\tilde{n}| \neq 0} |\eta|^{2(m+1)} |\mathscr{F} v(\eta)|^2 \right)^{\frac 12} \\
	&\leq \frac 14 \| v \|_{\dot{H}^{m+1}}^2 + \| R_h \Lambda^{-1} \tht \|_{\dot{H}^{m}}^2.
\end{align*}
By \eqref{sm_ass} and \eqref{sol_bdd_1}, we have 
\begin{align*}
\frac 12 \frac {\ud}{\ud t} \| v \|_{\dot{H}^m}^2 &\leq - (\frac 34 -C \| (v_0,\tht_0) \|_{H^m}^2) \| v \|_{\dot{H}^{m+1}}^2  + C\| R_h \tht \|_{\dot{H}^{m-1}}^2 \leq -\frac 12 \| v \|_{\dot{H}^m}^2  + C\| R_h \Lambda^{-1} \tht \|_{\dot{H}^{m}}^2.
\end{align*}
Then, applying Duhamel's principle shows
$$\| v(t) \|_{\dot{H}^m}^2 \leq e^{-t} \| v_0 \|_{\dot{H}^m}^2 + C \int_0^t e^{-(t-\tau)} \| R_h \tht \|_{\dot{H}^{m-1}}^2 \,\ud \tau.$$ Thus, from $$(1+\tau) \| R_h \Lambda^{-1} \tht(\tau) \|_{\dot{H}^{m}}^2 \leq (1+\tau) \| R_h^2 \Lambda^{-2} \tht(\tau) \|_{\dot{H}^{m}} \sup_{\tau \in [0,\infty)} \| \tht(\tau) \|_{\dot{H}^{m}},$$ it follows
\begin{equation}\label{v_Hs_est_1}
\| v(t) \|_{\dot{H}^m}^2 \leq C(1+t)^{-1} \left( C + \sup_{\tau \in [0,t]} (1+\tau)^2 \| R_h^2 \Lambda^{-2} \tht(\tau) \|_{\dot{H}^{m}}^2 \right).
\end{equation}

Recalling from the $v$ equations in \eqref{EQ} that
\begin{equation*}
\partial_t R_h v + (-\Delta) R_h v + R_h(v \cdot \nabla)v = R_h \mathbb{P} \tht e_d. 
\end{equation*}
and using \eqref{sm_ass}, we can see
\begin{align*}
\frac 12 \frac {\ud}{\ud t} \| R_h v \|_{\dot{H}^{m-1}}^2 &\leq - \| R_h v \|_{\dot{H}^m}^2 + C \| v \|_{\dot{H}^{m}} \| R_h v \|_{\dot{H}^{m}}^2 + \| R_h^2 \Lambda^{-2} \tht \|_{\dot{H}^{m}} \| R_h v \|_{\dot{H}^{m}} \\
&\leq -\frac 12 \| R_h v \|_{\dot{H}^m}^2 + \| R_h^2 \Lambda^{-2} \tht \|_{\dot{H}^{m}}^2.
\end{align*}
By Duhamel's principle,
\begin{equation}\label{Rhv_est_1}
\| R_hv(t) \|_{\dot{H}^{m-1}}^2 \leq \left( C + \sup_{\tau \in [0,t]} (1+\tau)^{2} \| R_h^2 \Lambda^{-2} \tht(\tau) \|_{\dot{H}^{m}}^2 \right).
\end{equation}

Now, we show that
\begin{equation}\label{claim_1}
\| R_h^2 \Lambda^{-2} \tht(t) \|_{\dot{H}^{m}}^2 \leq C(1+t)^{-2}.
\end{equation}
We can show from \eqref{df_vd} that
\begin{gather*}
\frac 12 \frac {\ud}{\ud t} \| R_h \Lambda^{-2} v_d \|_{\dot{H}^{m}}^2 + \| R_h \Lambda^{-1} v_d \|_{\dot{H}^{m}}^2 \\
= - \sum_{|\gamma| = m-3} \int_{\Omega} \nabla_h \partial^{\gamma} \langle \mathbb{P} (v \cdot \nabla)v,e_d \rangle \cdot \nabla_h \partial^{\gamma} v_d \,\ud x + \sum_{|\tilde{n}| \neq 0} |\tilde{n}|^4|\eta|^{2(m-4)} \mathscr{F}_b \tht(\eta)\mathscr{F}_b v_d(\eta).
\end{gather*}
From \eqref{df_tht},
\begin{gather*}
\frac 12 \frac {\ud}{\ud t} \| R_h^2 \Lambda^{-2} \tht \|_{\dot{H}^{m}}^2 = -\sum_{|\gamma| = m-4} \int_{\Omega} \partial^\gamma \partial_h^2 (v \cdot \nabla) \tht \cdot \partial^\gamma\partial_h^2 \tht \,\ud x - \sum_{|\tilde{n}| \neq 0} |\tilde{n}|^4|\eta|^{2(m-4)} \mathscr{F}_b \tht(\eta)\mathscr{F}_b v_d(\eta).
\end{gather*}
Thus,
\begin{gather*}
\frac 12 \frac {\ud}{\ud t} \left( \| R_h \Lambda^{-2} v_d \|_{\dot{H}^{m}}^2 + \| R_h^2 \Lambda^{-2} \tht \|_{\dot{H}^{m}}^2 \right) + \| R_h \Lambda^{-1} v_d \|_{\dot{H}^{m}}^2 \\
\leq - \sum_{|\gamma| = m-4} \int_{\Omega} \nabla_h \partial^{\gamma} \langle \mathbb{P} (v \cdot \nabla)v,e_d \rangle \cdot \nabla_h \partial^{\gamma} (-\Delta) v_d \,\ud x -\sum_{|\gamma| = m-4} \int_{\Omega} \partial^\gamma \partial_h^2 (v \cdot \nabla) \tht \cdot \partial^\gamma\partial_h^2 \tht \,\ud x.
\end{gather*}
Using integration by parts and H\"older's inequality, we have
\begin{gather*}
\left| - \sum_{|\gamma| = m-4} \int_{\Omega} \nabla_h \partial^{\gamma} \langle \mathbb{P} (v \cdot \nabla)v,e_d \rangle \cdot \nabla_h \partial^{\gamma} (-\Delta) v_d \,\ud x \right| \\
\leq C(\| (v_h \cdot \nabla_h) v \|_{\dot{H}^{m-3}} + \| \nabla_h (v_d \partial_d v) \|_{\dot{H}^{m-4}}) \| R_h v_d \|_{\dot{H}^{m-1}} \\
\leq C\| v \|_{\dot{H}^{m}} \| R_h v \|_{\dot{H}^{m-1}} \| R_h v_d \|_{\dot{H}^{m-1}} + C\| v \|_{H^m} \| R_h v_d \|_{\dot{H}^{m-1}}^2.
\end{gather*}
We note that
\begin{gather*}
\sum_{|\gamma| = m-4} \int_{\Omega} \partial^\gamma \partial_h^2 (v \cdot \nabla) \tht \cdot \partial^\gamma\partial_h^2 \tht \,\ud x = \sum_{|\gamma| = m-4} \left( K_7 + K_8 + K_9 \right),
\end{gather*}
where
\begin{align*}
K_7 &= \int_{\Omega} \partial^\gamma (\partial_h^2 v \cdot \nabla) \tht \cdot \partial^\gamma\partial_h^2 \tht \,\ud x, \\
K_8 &= \int_{\Omega} \partial^\gamma (\partial_h v \cdot \nabla) \partial_h \tht \cdot \partial^\gamma\partial_h^2 \tht \,\ud x, \\
K_9 &:= \int_{\Omega} \partial^\gamma (v \cdot \nabla) \partial_h^2 \tht \cdot \partial^\gamma\partial_h^2 \tht \,\ud x .
\end{align*}
The integration by parts and $(v \cdot \nabla) \tht = (v_h \cdot \nabla_h) \tht + v_d \partial_d \tht$ show $$K_7 + K_8 = - \int_{\Omega} \partial^\gamma (\partial_h v_h \cdot \nabla_h) \tht \cdot \partial^\gamma\partial_h R_h^2 (-\Delta) \tht \,\ud x - \int_{\Omega} \partial^\gamma (\partial_h v_d \partial_d \tht) \cdot \partial^\gamma\partial_h R_h^2 (-\Delta)  \tht \,\ud x. $$ By the use of the integration by parts, we can estimate the second integral on the right-hand side as 
\begin{align*}
\left| - \int_{\Omega} \partial^\gamma (\partial_h v_d \partial_d \tht) \cdot \partial^\gamma\partial_h R_h^2 (-\Delta)  \tht \,\ud x \right| &= \left|- \int_{\Omega} (-\Delta) \partial^{\gamma} (\partial_h v_d \partial_d \tht) \cdot \partial^{\gamma} \partial_h R_h^2 \tht \,\ud x \right| \\
&\leq C\| R_h v_d \|_{\dot{H}^{m-1}} \| \tht \|_{H^m} \| R_h^3 \tht \|_{\dot{H}^{m-3}}.
\end{align*}
Similarly, the first one is bounded by 
\begin{gather*}
\left| \int_{\Omega} (-\Delta) \partial^{\gamma} (\partial_h v_h \cdot \nabla_h) \tht \cdot \partial^{\gamma} \partial_h R_h^2 \tht \,\ud x \right| 
\leq C \| R_h v \|_{\dot{H}^{m-1}} \| R_h \tht \|_{\dot{H}^{m-1}} \| R_h^3 \tht \|_{\dot{H}^{m-3}}.
\end{gather*}
Thus,
\begin{gather*}
|K_7 + K_8| \leq C\| R_h v_d \|_{\dot{H}^{m-1}} \| \tht \|_{H^m} \| R_h^3 \tht \|_{\dot{H}^{m-3}} + C \| R_h v \|_{\dot{H}^{m-1}} \| R_h \tht \|_{\dot{H}^{m-1}} \| R_h^3 \tht \|_{\dot{H}^{m-3}}.
\end{gather*}
Due to the cancellation property, we have for $|\gamma'|=1$ that
\begin{align*}
K_9 &= \int_{\Omega} \partial^{\gamma-\gamma'} (\partial^{\gamma'} v_h \cdot \nabla_h) \partial_h^2 \tht \cdot \partial^\gamma\partial_h^2 \tht \,\ud x + \int_{\Omega} \partial^{\gamma-\gamma'} (\partial^{\gamma'} v_d \partial_d \partial_h^2 \tht) \cdot \partial^\gamma\partial_h^2 \tht \,\ud x.
\end{align*}
By integration by parts and the calculus inequality, it can be shown that 
\begin{gather*}
\left| \int_{\Omega} \partial^{\gamma-\gamma'} (\partial^{\gamma'} v_h \cdot \nabla_h) \partial_h^2 \tht \cdot \partial^\gamma\partial_h^2 \tht \,\ud x \right| \\
\leq \left| \int_{\Omega} \partial^{\gamma-\gamma'} (\partial^{\gamma'} \nabla_h \cdot v_h \partial_h^2 \tht) \cdot \partial^\gamma\partial_h^2 \tht \,\ud x \right| + \left| \int_{\Omega} \partial^{\gamma-\gamma'} (\partial^{\gamma'} v_h  \partial_h^2 \tht) \cdot \nabla_h \partial^\gamma\partial_h^2 \tht \,\ud x \right|.
\end{gather*}
The first term on the right-hand side is bounded by
\begin{align*}
\left| \int_{\Omega} \partial^{\gamma-\gamma'} (\partial^{\gamma'} \partial_d v_d \partial_h^2 \tht) \cdot \partial^\gamma\partial_h^2 \tht \,\ud x \right| &\leq C(\| v_d \|_{\dot{H}^{m-3}} \| \partial_h^2 \tht \|_{L^{\infty}} + \| \nabla \partial_d v_d \|_{L^p} \| \partial^{\gamma-\gamma'} \partial_h^2 \tht \|_{L^q}) \| R_h^2 \tht \|_{\dot{H}^{m-2}} \\
&\leq C \| R_h v_d \|_{\dot{H}^{m-1}} \| R_h \tht \|_{\dot{H}^{m-1}} \| R_h^2 \tht \|_{\dot{H}^{m-2}},
\end{align*}
where $\frac 1p + \frac 1q = \frac 12$ and $\frac 1p = \frac 2d + (\frac 12 - \frac {m-2}d)\frac 2{m-2}$.On the other hand, the integration by parts yields
\begin{gather*}
\left| \int_{\Omega} \partial^{\gamma-\gamma'} (\partial^{\gamma'} v_h  \partial_h^2 \tht) \cdot \nabla_h \partial^\gamma\partial_h^2 \tht \,\ud x \right| \\
= \left| \int_{\Omega} (-\Delta) \partial^{\gamma-\gamma'} (\partial^{\gamma'} v_h  \partial_h^2 \tht) \cdot \partial^\gamma\partial_h R_h^2 \tht \,\ud x \right| \\
\leq C ( \| \nabla v_h \|_{\dot{H}^{m-3}} \| \partial_h^2 \tht \|_{L^{\infty}} + \| \nabla v_h \|_{L^{\infty}} \| \partial_h^2 \tht \|_{\dot{H}^{m-3}}) \| R_h^3 \tht \|_{\dot{H}^{m-3}} \\
\leq C \| v \|_{\dot{H}^{m-2}} \| R_h \tht \|_{\dot{H}^{m-1}} \| R_h^3 \tht \|_{\dot{H}^{m-3}}.
\end{gather*}
Hence,
$$|K_9| \leq C \| R_h v_d \|_{\dot{H}^{m-1}} \| R_h \tht \|_{\dot{H}^{m-1}} \| R_h^2 \tht \|_{\dot{H}^{m-2}} + C \| v \|_{\dot{H}^{m-2}} \| R_h \tht \|_{\dot{H}^{m-1}} \| R_h^3 \tht \|_{\dot{H}^{m-3}}.$$
By the above estimates, we deduce
\begin{gather*}
\frac 12 \frac {\ud}{\ud t} \left( \| R_h v_d \|_{\dot{H}^{m-2}}^2 + \| R_h^2 \tht \|_{\dot{H}^{m-2}}^2 \right) + \| R_h v_d \|_{\dot{H}^{m-1}}^2 \\
\leq C(\| v \|_{H^m} + \| \tht \|_{H^m}) (\| R_h v_d \|_{\dot{H}^{m-1}}^2 + \| R_h^2 \tht \|_{\dot{H}^{m-2}}) \\
 + C (\| R_h v \|_{\dot{H}^{m-1}} + \| R_h^2 \tht \|_{\dot{H}^{m-2}} + \| v \|_{\dot{H}^{m-2}}) (\| v \|_{\dot{H}^{m}} + \| R_h \tht \|_{\dot{H}^{m-1}}) ( \| R_h v_d \|_{\dot{H}^{m-1}} + \| R_h^3 \tht \|_{\dot{H}^{m-3}}).
\end{gather*}
On the other hand, using \eqref{df_vd} and \eqref{Leray_est}, we have
\begin{align*}
-\int_{\Omega} \partial_tv_d (-\Delta)^{m-3} R_h^4 \tht \,\ud x &= \int_{\Omega} (v \cdot \nabla)v_d (-\Delta)^{m-3} R_h^4 \tht \,\ud x - \int_{\Omega} \partial_d \nabla \cdot ((v \cdot \nabla)v) (-\Delta)^{m-3} R_h^4 \tht \,\ud x \\
&\hphantom{\qquad\qquad}- \int_{\Omega} (-\Delta) v_d (-\Delta)^{m-3} R_h^4 \tht \,\ud x - \| R_h^3 \tht \|_{\dot{H}^{m-3}}^2 \\
&\leq \int_{\Omega} (v \cdot \nabla)v_d (-\Delta)^{m-3} R_h^4 \tht \,\ud x + C \| v \|_{\dot{H}^{m-2}}\| v \|_{\dot{H}^{m-1}} \| R_h^3 \tht \|_{\dot{H}^{m-3}} \\
&\hphantom{\qquad\qquad}+ \frac 12 \| R_hv_d \|_{\dot{H}^{m-1}}^2 - \frac 12\| R_h^3 \tht \|_{\dot{H}^{m-3}}^2.
\end{align*}
Since \eqref{df_tht} yields
\begin{equation*}
-\int_{\Omega} \partial_t\tht (-\Delta)^{m-3} R_h^4 v_d \,\ud x \leq \int_{\Omega} (v \cdot \nabla)\tht (-\Delta)^{m-3} R_h^4 v_d \,\ud x + \| R_h^2 v_d \|_{\dot{H}^{m-3}}^2,
\end{equation*}
we have
\begin{gather*}
-\frac {\ud}{\ud t} \int_{\Omega} v_d (-\Delta)^{m-3} R_h^4 \tht \,\ud x \leq \int_{\Omega} (v \cdot \nabla)v_d (-\Delta)^{m-3} R_h^4 \tht \,\ud x + \int_{\Omega} (v \cdot \nabla)\tht (-\Delta)^{m-3} R_h^4 v_d \,\ud x \\
+ C\| v \|_{\dot{H}^{m-2}} \| v \|_{\dot{H}^{m}} \| R_h^3 \tht \|_{\dot{H}^{m-3}} - \frac 12 \| R_h^3 \tht \|_{\dot{H}^{m-3}}^2+ \frac 32 \| R_h v_d \|_{\dot{H}^{m-1}}^2.
\end{gather*}
It is clear $$\left| \int_{\Omega} (v \cdot \nabla)v_d (-\Delta)^{m-3} R_h^4 \tht \,\ud x \right| \leq C \| v \|_{\dot{H}^{m-3}} \| v_d \|_{\dot{H}^{m-2}} \| R_h^3 \tht \|_{\dot{H}^{m-3}} \leq C \| v \|_{\dot{H}^{m}} \| R_h v_d \|_{\dot{H}^{m-1}} \| R_h^3 \tht \|_{\dot{H}^{m-3}}.$$ Since $H^{m-3}(\Omega)$ is banach algebra, we can see
\begin{gather*}
	\left| \int_{\Omega} (v \cdot \nabla)\tht (-\Delta)^{m-3} R_h^4 v_d \,\ud x \right| \\
	\leq \| (v_h \cdot \nabla_h)\tht \|_{\dot{H}^{m-3}}  \| R_h^4 v_d \|_{\dot{H}^{m-3}} + C\| v_d \partial_d \tht \|_{\dot{H}^{m-3}} \| R_h^4 v_d \|_{\dot{H}^{m-3}} \\
	\leq C\| v \|_{\dot{H}^{m-3}} \| \nabla_h \tht \|_{\dot{H}^{m-3}} \| R_h^4 v_d \|_{\dot{H}^{m-3}} + C\| v_d \|_{\dot{H}^{m-3}} \| \partial_d \tht \|_{\dot{H}^{m-3}} \| R_h^4 v_d \|_{\dot{H}^{m-3}} \\
	\leq C\| v \|_{\dot{H}^{m-2}} \| R_h \tht \|_{\dot{H}^{m-1}} \|R_h v_d \|_{\dot{H}^{m-1}} + C\| \tht \|_{\dot{H}^{m}} \| R_h v_d \|_{\dot{H}^{m-1}}^2.
\end{gather*}
Collecting the above estimates gives
\begin{gather*}
-\frac {\ud}{\ud t} \int_{\Omega} v_d (-\Delta)^{m-3} R_h^4 \tht \,\ud x \leq - \frac 12 \| R_h^3 \tht \|_{\dot{H}^{m-3}}^2 + \frac 32 \| R_h v_d \|_{\dot{H}^{m-1}}^2 + C\| v \|_{\dot{H}^{m-2}} \| v \|_{\dot{H}^{m}} \| R_h^3 \tht \|_{\dot{H}^{m-3}} \\
+ C \| v \|_{\dot{H}^{m}} \| R_h v_d \|_{\dot{H}^{m-1}} \| R_h^3 \tht \|_{\dot{H}^{m-3}} +C\| v \|_{\dot{H}^{m-2}} \| R_h \tht \|_{\dot{H}^{m-1}} \|R_h v_d \|_{\dot{H}^{m-1}} + C\| \tht \|_{\dot{H}^{m}} \| R_h v_d \|_{\dot{H}^{m-1}}^2.
\end{gather*}
Now, we arrived at
\begin{gather*}
\frac 12 \frac {\ud}{\ud t} \left( \| R_h v_d \|_{\dot{H}^{m-2}}^2 + \| R_h^2 \tht \|_{\dot{H}^{m-2}}^2 - \int_{\Omega} v_d (-\Delta)^{m-3} R_h^4 \tht \,\ud x \right) \\
\leq -(\frac 14 - C(\| v \|_{H^m} + \| \tht \|_{H^m})) \left( \| R_h v_d \|_{\dot{H}^{m-1}}^2 + \| R_h^3 \tht \|_{\dot{H}^{m-3}}^2 \right) \\
+ C (\| R_h v \|_{\dot{H}^{m-1}} + \| R_h^2 \tht \|_{\dot{H}^{m-2}} + \| v \|_{\dot{H}^{m-2}}) (\| v \|_{\dot{H}^{m}} + \| R_h \tht \|_{\dot{H}^{m-1}}) ( \| R_h v_d \|_{\dot{H}^{m-1}} + \| R_h^3 \tht \|_{\dot{H}^{m-3}}) \\
\leq -\frac 18 \left( \| R_h v_d \|_{\dot{H}^{m-1}}^2 + \| R_h^3 \tht \|_{\dot{H}^{m-3}}^2 \right) \\
+ C (\| R_h v \|_{\dot{H}^{m-1}}^2 + \| R_h^2 \tht \|_{\dot{H}^{m-2}}^2 + \| v \|_{\dot{H}^{m-2}}^2) (\| v \|_{\dot{H}^{m}}^2 + \| R_h \tht \|_{\dot{H}^{m-1}}^2).
\end{gather*}
We have used Young's inequality and \eqref{sm_ass} in the last inequality. We consider $M \geq 1$ which will be specified later. Since
\begin{equation*}
\begin{aligned}
\frac 1M \| R_h^2 \tht \|_{\dot{H}^{m-2}}^2 - \| R_h^3 \tht \|_{\dot{H}^{m-3}}^2 &= \sum_{|\tilde{n}| \neq 0} \left( \frac 1M - \frac {|\tilde{n}|^2}{|\eta|^4} \right) |\eta|^{2(m-2)} |\mathscr{F} R_h^2 \tht(\eta)|^2 \\
&\leq \frac 1M \sum_{\frac {|\tilde{n}|^2}{|\eta|^4} \leq \frac 1M, |\tilde{n}| \neq 0} |\eta|^{2(m-2)} |\mathscr{F} R_h^2 \tht(\eta)|^2 \\
&\leq \frac 1{M^2} \| R_h \tht \|_{\dot{H}^{m-1}}^2
\end{aligned}
\end{equation*}
and
\begin{equation}\label{smp_ineq}
\left| \int_{\Omega} v_d (-\Delta)^{m-3} R_h^4 \tht \,\ud x \right| \leq \| R_h v_d \|_{\dot{H}^{m-3}} \| R_h^3 \tht \|_{\dot{H}^{m-3}} \leq \frac 12 \| R_h v_d \|_{\dot{H}^{m-1}}^2 + \frac 12 \| R_h^3 \tht \|_{\dot{H}^{m-3}}^2,
\end{equation}
it holds
\begin{gather*}
-\frac 18 \left( \| R_h v_d \|_{\dot{H}^{m-1}}^2 + \| R_h^3 \tht \|_{\dot{H}^{m-3}}^2 \right) \leq -\frac 1{8M} \left( \| R_h v_d \|_{\dot{H}^{m-1}}^2 + \| R_h^2 \tht \|_{\dot{H}^{m-2}}^2 \right) + \frac 1{16M} \int_{\Omega} v_d (-\Delta)^{m-3} R_h^4 \tht \,\ud x \\
+ \frac 1{8M^2} \| R_h \tht \|_{\dot{H}^{m-1}}^2 - \frac 1{16M} \int_{\Omega} v_d (-\Delta)^{m-3} R_h^4 \tht \,\ud x \\
\leq -\frac 1{16M} \left(\|R_h v_d\|_{\dot{H}^{m-2}}^2 + \| R_h^2 \tht \|_{\dot{H}^{m-2}}^2 - \int_{\Omega} v_d (-\Delta)^{m-3} R_h^4 \tht \,\ud x \right) + \frac 1{8M^2} \| R_h \tht \|_{\dot{H}^{m-1}}^2.
\end{gather*}
Hence,
\begin{gather*}
\frac 12 \frac {\ud}{\ud t} \left( \| R_h v_d \|_{\dot{H}^{m-2}}^2 + \| R_h^2 \tht \|_{\dot{H}^{m-2}}^2 - \int_{\Omega} v_d (-\Delta)^{m-3} R_h^4 \tht \,\ud x \right) \\
\leq -\frac 1{16M} \left(\|R_h v_d\|_{\dot{H}^{m-2}}^2 + \| R_h^2 \tht \|_{\dot{H}^{m-2}}^2 - \int_{\Omega} v_d (-\Delta)^{m-3} R_h^4 \tht \,\ud x \right) + \frac 1{8M^2} \| R_h \tht \|_{\dot{H}^{m-1}}^2 \\
+ C (\| R_h v \|_{\dot{H}^{m-1}}^2 + \| R_h^2 \tht \|_{\dot{H}^{m-2}}^2 + \| v \|_{\dot{H}^{m-2}}^2) (\| v \|_{\dot{H}^{m}}^2 + \| R_h \tht \|_{\dot{H}^{m-1}}^2).
\end{gather*}
Here, we take $M = 1+\frac t{16}$. Then multiplying the both sides by $2M^2$ and using \eqref{vd_Hs_est}, we have
\begin{gather*}
\frac {\ud}{\ud t} \left( (1+\frac t{16})^2 (\| R_h v_d \|_{\dot{H}^{m-2}}^2 + \| R_h^2 \tht \|_{\dot{H}^{m-2}}^2 - \int_{\Omega} v_d (-\Delta)^{m-3} R_h^4 \tht \,\ud x) \right) \leq C \| R_h \tht \|_{\dot{H}^{m-1}}^2 \\
+ C (1+\frac t{16})^2 (\| R_h v \|_{\dot{H}^{m-1}}^2 + \| R_h^2 \tht \|_{\dot{H}^{m-2}}^2 + \| v \|_{\dot{H}^{m-2}}^2) (\| v \|_{\dot{H}^{m}}^2 + \| R_h \tht \|_{\dot{H}^{m-1}}^2).
\end{gather*}
From \eqref{Rhv_est_1}, it is clear $$C (1+\frac t{16})^2 (\| R_h v \|_{\dot{H}^{m-1}}^2 + \| R_h^2 \tht \|_{\dot{H}^{m-2}}^2) \leq C + C\sup_{\tau \in [0,t]} (1+\tau)^{2} \| R_h^2 \tht(\tau) \|_{\dot{H}^{m-2}}^2. $$ Using \eqref{v_L2}, \eqref{v_Hs_est_1} and the interpolation inequality gives 
\begin{align*}
(1+\frac t{16})^2 \| v \|_{\dot{H}^{m-2}}^2 &\leq (1+\frac t{16})^2 \| v \|_{L^2}^{\frac 4m} \| v \|_{\dot{H}^{m}}^{2-\frac 4m} \\
&\leq C(1+t)^{1-\frac 2m} \| v \|_{\dot{H}^{m}}^{2-\frac 4m} \\
&\leq C + C\sup_{\tau \in [0,t]} (1+\tau)^2 \| R_h^2 \tht(\tau) \|_{\dot{H}^{m-2}}^2.
\end{align*}
Therefore, 
\begin{gather*}
\frac {\ud}{\ud t} \left( (1+\frac t{16})^2 (\| R_h v_d \|_{\dot{H}^{m-2}}^2 + \| R_h^2 \tht \|_{\dot{H}^{m-2}}^2 - \int_{\Omega} v_d (-\Delta)^{m-3} R_h^4 \tht \,\ud x) \right) \\
\leq \left( C + C\sup_{\tau \in [0,t]} (1+\frac \tau{16})^2 (\| R_h v_d(\tau) \|_{\dot{H}^{m-2}}^2 + \| R_h^2 \tht(\tau) \|_{\dot{H}^{m-2}}^2) \right) (\| v \|_{\dot{H}^{m}}^2 + \| R_h \tht \|_{\dot{H}^{m-1}}^2).
\end{gather*}
We integrate it over time and use \eqref{smp_ineq} with $$\int_0^{\infty} (\| v \|_{\dot{H}^{m}}^2 + \| R_h \tht \|_{\dot{H}^{m-1}}^2) \,\ud t \leq C.$$ Then, for $$f(t) := \sup_{\tau \in [0,t]} (1+\frac \tau{16})^2 \left( \|R_h v_d(\tau) \|_{\dot{H}^{m-2}}^2 + \| R_h^2 \tht(\tau) \|_{\dot{H}^{m-2}}^2 \right),$$ it holds
\begin{gather*}
f(t) \leq C + \int_0^t f(\tau) (\| v \|_{\dot{H}^s}^2 +  \| R_h \tht \|_{\dot{H}^s}^2 + \| \nabla v_d \|_{L^{\infty}}) \,\ud \tau.
\end{gather*}
By Gr\"onwall's ineqaulity, we obtain \eqref{claim_1}. 

Now, we prove that 
\begin {equation}\label{vd_est_1}
\| v_d \|_{\dot{H}^m} \leq C(1+t)^{-1}.
\end{equation}
As showing \eqref{vd_est_u}, we can have
\begin{equation*}
\| v_d (t) \|_{\dot{H}^m} \leq C \| R_h^2 \Lambda^{-2} \tht(t) \|_{\dot{H}^{m}} + C \left( \sum_{\eta \in J} |\eta|^{2(m-2)} | \langle \mathscr{F} \mathbf{u},\mathbf{a}_{+} \rangle |^2 \right)^{\frac 12}.
\end{equation*}
Thus, it suffices to show
\begin{equation}\label{u+_est_1}
\left( \sum_{\eta \in J} |\eta|^{2(m-2)} | \langle \mathscr{F} \mathbf{u},\mathbf{a}_{+} \rangle |^2 \right)^{\frac 12} \leq C (1+t)^{-1}.
\end{equation}
We can see from \eqref{df_u+} that
\begin{equation*}
	\langle \mathscr{F}_b\bold{u}(t),\bold{a}_+ \rangle = e^{-\lambda_+ t} \langle \mathscr{F}\bold{u}_0,\bold{a}_+ \rangle - \int_0^t e^{-\lambda_+ (t - \tau)} \langle N(v,\tht)(\tau),\bold{a}_+ \rangle \,\mathrm{d}\tau .
\end{equation*}
Due to $|e^{-\lambda_+ t}| \leq e^{-\frac {|\eta|^{2}}2 t}$ for $\eta \in J$, it follows by the Minkowski inequality
\begin{gather*}
	\left( \sum_{\eta \in J} |\eta|^{2(m-2)} | \langle \mathscr{F} \mathbf{u},\mathbf{a}_{+} \rangle |^2 \right)^{\frac 12} \\
	\leq \left( \sum_{\eta \in J} |\eta|^{2(m-2)} e^{-|\eta|^{2}t} | \langle \mathscr{F} \mathbf{u}_0,\mathbf{a}_{+} \rangle |^2 \right)^{\frac 12} + \int_0^t \left( \sum_{\eta \in J} |\eta|^{2(m-2)}  e^{-|\eta|^{2}(t - \tau)} |\langle N(v,\tht)(\tau),\bold{a}_+ \rangle|^2 \right)^{\frac 12} \,\mathrm{d}\tau.
\end{gather*}
From the simple fact $|\mathbf{a}_{+}|^2 = |\lambda_{+}|^2 + \frac {|\tilde{n}|^4}{|\eta|^4} \leq C|\eta|^{4}$ with \eqref{sm_ass}, we have $$\left( \sum_{\eta \in J} |\eta|^{2(m-2)}  e^{-|\eta|^{2\alpha}t} | \langle \mathscr{F} \mathbf{u}_0,\mathbf{a}_{+} \rangle |^2 \right)^{\frac 12} \leq Ce^{-t} \| \mathbf{u}_0 \|_{H^m}.$$ By \eqref{N_est} we have
\begin{gather*}
	\int_0^t \left( \sum_{\eta \in J} |\eta|^{2(m-2)}  e^{-|\eta|^{2}(t - \tau)} |\langle N(v,\tht)(\tau),\bold{a}_+ \rangle|^2 \right)^{\frac 12} \,\mathrm{d}\tau \\
	\leq \int_0^t \left( \sum_{\eta \in J} e^{-|\eta|^{2}(t - \tau)} |\eta|^{2m} | \mathscr{F}(v \cdot \nabla)v |^2 \right)^{\frac 12} \,\mathrm{d}\tau + \int_0^t \left( \sum_{\eta \in J} e^{-|\eta|^{2}(t - \tau)} |\eta|^{2(m-2)} | \mathscr{F}_b(v \cdot \nabla)\tht |^2 \right)^{\frac 12} \,\mathrm{d}\tau\\
	\leq \int_0^t e^{-(t - \tau)} (\| (v \cdot \nabla)v (\tau) \|_{\dot{H}^{m}}  + \| (v \cdot \nabla)\tht (\tau) \|_{\dot{H}^{m-2}}) \,\mathrm{d}\tau.
\end{gather*}
We note $$\| (v \cdot \nabla)v \|_{\dot{H}^m} \leq C \| v \|_{\dot{H}^{m+1}} \| v \|_{\dot{H}^{m-2}} \leq C \| v \|_{\dot{H}^{m+1}} \| v \|_{\dot{H}^m}^{1-\frac {2}m} \| v \|_{L^2}^{\frac 2m}$$ and 
\begin{align*}
\| (v \cdot \nabla)\tht \|_{\dot{H}^{m-2}} &\leq \| v \|_{\dot{H}^{m-2}} \| \tht \|_{\dot{H}^{m-1}} \leq \| v \|_{\dot{H}^m}^{1-\frac {2}m} \| v \|_{L^2}^{\frac 2m} \| \tht \|_{H^m}.
\end{align*}
Combining \eqref{v_L2} and \eqref{v_Hs_1}, we can see 
\begin{gather*}
	(1+\tau) (\| (v \cdot \nabla)v(\tau) \|_{\dot{H}^m} + \| (v \cdot \nabla)\tht(\tau) \|_{\dot{H}^{m-2}}) \leq C (\| v \|_{\dot{H}^{m+1}} + \| \tht \|_{H^m}).
\end{gather*}
Thus,
\begin{gather*}
\int_0^t e^{-(t - \tau)} (\| (v \cdot \nabla)v (\tau) \|_{\dot{H}^{m}}  + \| (v \cdot \nabla)\tht (\tau) \|_{\dot{H}^{m-2}}) \,\mathrm{d}\tau \leq C (1+t)^{-1}.
 \end{gather*}
Collecting the above estimates, we obtain \eqref{u+_est_1}, which implies \eqref{vd_est_1}. This completes the proof.
\end{proof}

\begin{proposition}\label{prop_rmk}
	Let $d \in \mathbb{N}$ with $d \geq 2$ and $\alpha =1$. Let $m \in \mathbb{N}$ with $m > 3+\frac d2$ and $(v, \theta)$ be a smooth global solution to \eqref{EQ} with \eqref{sol_bdd_1}. Suppose that \eqref{sm_ass} be satisfied. Then, for any $\epsilon \in (0,1)$, there exists a constant $C > 0$ such that
	\begin{equation}\label{v_L2_1}
		\| \Lambda^{-\epsilon} v(t) \|_{L^2} \leq C (1+t)^{-(\frac 34 + \frac{m}4)}
	\end{equation}
	and
	\begin{equation}\label{v_Hs_1_1}
		\| \Lambda^{-\epsilon} v(t) \|_{\dot{H}^{m+1}} \leq C t^{-\frac 12}.
	\end{equation}
\end{proposition}
\begin{proof}
Since \eqref{vd_L2} implies $\| \Lambda^{-\epsilon} v_d(t) \|_{L^2} \leq C t^{-(\frac 34 + \frac{m}4)}$, it suffices to show that $$\| \Lambda^{-\epsilon} v_h(t) \|_{L^2} \leq C t^{-(\frac 34 + \frac{m}4)}.$$ Applying Duhamel's principle to \eqref{df_vh} with \eqref{avg_vh}, we obtain
\begin{equation*}
	\begin{gathered}
		\left( \sum_{\eta \in I} |\eta|^{-2\epsilon} |\mathscr{F}_c v_h|^2 \right)^{\frac 12} \\
		\leq e^{- t} \| v_0 \|_{L^2} + \int_0^t e^{-(t-\tau)} \| (v \cdot \nabla)v \|_{L^2} \,\ud \tau + \left( \sum_{\eta \in I} \left| \int_0^t e^{-|\eta|^2 (t-\tau)} \frac {|\tilde{n}|}{|\eta|^{1+\epsilon}} |\mathscr{F}_b \tht (\eta)| \,\ud \tau \right|^2 \right)^{\frac 12}.
	\end{gathered}
\end{equation*}
for any $\epsilon \in (0,1)$. We clearly have by \eqref{v_L2} and \eqref{v_Hs_1} that $$\int_0^t e^{-(t-\tau)} \| (v \cdot \nabla)v \|_{L^2}^2 \,\ud \tau \leq \int_0^t e^{-(t-\tau)} \| v \|_{L^2} \| v \|_{H^m}^2 \,\ud \tau \leq C(1+t)^{-(2+\frac m2)}.$$ On the other hand, we can see 
\begin{gather*}
	\left| \int_0^t e^{-|\eta|^2 (t-\tau)} \frac {|\tilde{n}|}{|\eta|^{1+\epsilon}} |\mathscr{F}_b \tht (\eta)| \,\ud \tau \right|^2 \\
	\leq C \left| e^{-\frac t2} \int_0^{\frac t2} |\mathscr{F}_b \tht (\eta)| \,\ud \tau \right|^2 + C \left| \int_{\frac t2} ^t (t-\tau)^{-(1-\frac \epsilon4)} \frac {|\tilde{n}|}{|\eta|^{3+\frac \epsilon2}} |\mathscr{F}_b \tht (\eta)| \,\ud \tau \right|^2.
\end{gather*} Thus, we have
\begin{gather*}
	\left( \sum_{\eta \in I} \left| \int_0^t e^{-|\eta|^2 (t-\tau)} \frac {|\tilde{n}|}{|\eta|^{1+\epsilon}} |\mathscr{F}_b \tht (\eta)| \,\ud \tau \right|^2 \right)^{\frac 12} \\
	\leq C \left( \sum_{\eta \in I} \left| e^{-\frac t2} \int_0^{\frac t2} |\mathscr{F}_b \tht (\eta)| \,\ud \tau \right|^2 \right)^{\frac 12} + C \left( \sum_{\eta \in I} \left| \int_{\frac t2} ^t (t-\tau)^{-(1-\frac \epsilon4)} \frac {|\tilde{n}|}{|\eta|^{3+\frac \epsilon2}} |\mathscr{F}_b \tht (\eta)| \,\ud \tau \right|^2 \right)^{\frac 12} \\
	\leq C t e^{-\frac t2} + C \int_{\frac t2}^t (t-\tau)^{-(1-\frac \epsilon4)} \| R_h \Lambda^{-(2+\frac \epsilon2)} \tht \|_{L^2} \,\ud \tau.
\end{gather*}
We can infer from \eqref{v_L2} and \eqref{vd_L2} that $$\| R_h \Lambda^{-(2+\frac \epsilon2)} \tht \|_{L^2} \leq \| R_h^{\frac 32 + \frac \epsilon 4} \Lambda^{-(\frac 32 + \frac \epsilon 4)} \tht \|_{L^2} \leq C(1+t)^{-(\frac 32 + \frac \epsilon 4 +\frac m2)}.$$ Since this implies $$\int_{\frac t2}^t (t-\tau)^{-(1-\frac \epsilon4)} \| R_h \Lambda^{-(2+\frac \epsilon2)} \tht \|_{L^2} \,\ud \tau \leq C(1+t)^{-(\frac 32 + \frac m2)},$$ combining the above estimates gives \eqref{v_L2_1}.

Using \eqref{avg_vh} and \eqref{avg_vd}, we can infer from \eqref{df_vh} and \eqref{df_vd} that
\begin{equation*}
	\begin{gathered}
		\left( \sum_{\eta \in I} |\eta|^{2(m+1-\epsilon)} |\mathscr{F} v|^2 \right)^{\frac 12} \leq \left( \sum_{\eta \in I} e^{-|\eta|^2t} |\eta|^{2(m+1-\epsilon)} |\mathscr{F} v_0|^2 \right)^{\frac 12} \\
		+ \left( \sum_{\eta \in I} \left| \int_0^t e^{-|\eta|^2 (t-\tau)} |\eta|^{m+1-\epsilon} | \mathscr{F} (v \cdot \nabla)v | \,\ud \tau \right|^2 \right)^{\frac 12} + \left( \sum_{\eta \in I} \left| \int_0^t e^{-|\eta|^2 (t-\tau)} |\tilde{n}| |\eta|^{m-\epsilon} |\mathscr{F}_b \tht (\eta)| \,\ud \tau \right|^2 \right)^{\frac 12}
	\end{gathered}
\end{equation*}
for any $\epsilon \in (0,1)$. We can see $$\left( \sum_{\eta \in I} e^{-|\eta|^2t} |\eta|^{2(m+1-\epsilon)} |\mathscr{F} v_0|^2 \right)^{\frac 12} \leq Ct^{-\frac{1-\epsilon}2} e^{-\frac t2} \| v_0 \|_{H^m}.$$ Since
\begin{align*}
	\left( \sum_{\eta \in I} \left| \int_0^t e^{-|\eta|^2 (t-\tau)} |\eta|^{m+1-\epsilon} | \mathscr{F} (v \cdot \nabla)v | \,\ud \tau \right|^2 \right)^{\frac 12} &\leq C\int_0^t (t-\tau)^{-\frac {2-\epsilon}2} e^{- \frac {t-\tau}2} \| (v \cdot \nabla)v \|_{\dot{H}^{m-1}} \,\ud \tau \\
	&\leq C \int_0^t (t-\tau)^{-\frac {2-\epsilon}2} e^{- \frac {t-\tau}2} \| v \|_{\dot{H}^m}^2 \,\ud \tau
\end{align*}
and
\begin{align*}
	\left( \sum_{\eta \in I} \left| \int_0^t e^{-|\eta|^2 (t-\tau)} |\tilde{n}| |\eta|^{m-\epsilon} |\mathscr{F}_b \tht (\eta)| \,\ud \tau \right|^2 \right)^{\frac 12} &\leq C\int_0^t (t-\tau)^{-\frac {2-\epsilon}2} e^{- \frac {t-\tau}2} \| R_h \tht \|_{\dot{H}^{m-1}} \,\ud \tau,
\end{align*}
we have by \eqref{v_Hs_1} that
\begin{gather*}
	\left( \sum_{\eta \in I} \left| \int_0^t e^{-|\eta|^2 (t-\tau)} |\eta|^{m+1-\epsilon} | \mathscr{F} (v \cdot \nabla)v | \,\ud \tau \right|^2 \right)^{\frac 12} + \left( \sum_{\eta \in I} \left| \int_0^t e^{-|\eta|^2 (t-\tau)} |\tilde{n}| |\eta|^{m-\epsilon} |\mathscr{F}_b \tht (\eta)| \,\ud \tau \right|^2 \right)^{\frac 12} \\
	\leq Ct^{-\frac 12}.
\end{gather*}
Collecting the above estimates gives \eqref{v_Hs_1_1}. This completes the proof.
\end{proof}

\section{Sharpness of decay rates}\label{sec:sharp}
In this section, we prove that the decay rates in Theorem~\ref{thm1} and \ref{thm2} are sharp in the following sense. We recall the linearized system of \eqref{UEQ}:
\begin{equation}\label{LUEQ}
	\partial_t \mathscr{F}_b\bold{u} + M \mathscr{F}_b \bold{u} = 0, \qquad M := \begin{pmatrix}
		|\eta|^{2\alpha} & -\frac {|\tilde{n}|^2}{|\eta|^2} \\
		1 & 0
	\end{pmatrix},
\end{equation}
where $\mathbf{u} = (v_d,\tht)^{T}$. The eigenvalues and eigenvectors of the linear operator is previously given by
\begin{equation*}
	\lambda_{\pm}(\eta) = \frac {|\eta|^{2\alpha} \pm \sqrt{|\eta|^{4\alpha} - {4 |\tilde{n}|^2}/{|\eta|^2}}}{2}, \qquad \overline{\bold{a}_{\pm} (\eta)} = 
	\begin{pmatrix}
		\lambda_{\pm} \\
		-\frac {|\tilde{n}|^2}{|\eta|^2}
	\end{pmatrix} ,
\end{equation*}
and it holds
\begin{equation}\label{df_lu}
	\mathscr{F}_b \mathbf{u} = \sum_{j = \pm} \langle \mathscr{F}_b\bold{u}(t),\bold{a}_j \rangle \bold{b}_j = \sum_{j=\pm} e^{-\lambda_j t} \langle \mathscr{F}\bold{u}_0,\bold{a}_j \rangle \bold{b}_j,
\end{equation}
where
\begin{equation*}
	\begin{pmatrix}
		\bold{b}_+ \\
		\bold{b}_-
	\end{pmatrix} =
	\frac {1}{\lambda_{+}-\lambda_{-}} 
	\begin{pmatrix}
		1 & \frac {| \eta|^2}{|\tilde{n}|^2} \lambda_{-} \\
		-1 & -\frac {| \eta|^2}{|\tilde{n}|^2} \lambda_{+}
	\end{pmatrix}.
\end{equation*}
Note that If we consider $\mathbf{u}_0$ such that $\mathscr{F}_b \mathbf{u}_0 = 0$ for $\eta \not\in D_3$, then $|\mathbf{a}_{\pm}||\mathbf{b}_{\pm}| \leq C$ for some $C>0$ not depending on $\eta$. Now, we are ready to provide the sharpness of the decay rates.
\begin{proposition}
	Let $m \in \bbN$. Then for any $\epsilon > 0$, there exists an initial data $\mathbf{u}_0 \in X^m(\Omega)$ such that the solution $\mathbf{u}=(v_d,\tht)$ to \eqref{LUEQ} satisfies
	\begin{equation}\label{tht_lin_decay}
		\| \bar{\tht}(t) \|_{H^s} \geq Ct^{-\frac {m-s}{2(1+\alpha)} - \epsilon}
	\end{equation}
	and
	\begin{equation}\label{vd_lin_decay}
		\| v_d(t) \|_{H^s} \geq Ct^{-1-\frac {m-s}{2(1+\alpha)} - \epsilon}
	\end{equation}
	for any $s \in [0,m]$ and $t \geq C$.
\end{proposition}
\begin{proof}
	Let $\epsilon > 0$ and $$\mathbf{u}_0 := (0, \sum_{\eta \in J} \mathscr{F}_b \tht_0(\eta) \mathscr{B}_{\eta}(x))^{T}, $$ where $ \mathscr{F}_b \tht_0(\eta) := |q|^{-(m+\frac 12 + \epsilon)}$ for $\eta \in D_3 \cap \{|n| = 1\}$, $\mathscr{F}_b \tht_0(\eta) = 0$ for $\eta \not\in D_3 \cap \{|n| = 1\}$. For simplicity, we use the notation $A:= D_3 \cap  \{|n| = 1\}$. We show \eqref{tht_lin_decay} first. From \eqref{df_lu}, we can see $$\| \bar{\tht} \|_{H^s} \geq \| e^{-\lambda_{-}t} \langle \mathscr{F}_b\bold{u}_0,\bold{a}_- \rangle \langle \bold{b}_-,e_2 \rangle \|_{H^s} - \| e^{-\lambda_{+}t} \langle \mathscr{F}_b\bold{u}_0,\bold{a}_+ \rangle \langle \bold{b}_+, e_2 \rangle \|_{H^s}.$$ Due to $|e^{-\lambda_{+}t}| \leq e^{-|\eta|^{2\alpha} \frac t2} \leq e^{-\frac t2}$ it is clear that 
	\begin{equation}\label{+_est}
		\| e^{-\lambda_{+}t} \langle \mathscr{F}_b\bold{u}_0,\bold{a}_+ \rangle \langle \bold{b}_+, e_2 \rangle \|_{H^s} \leq Ce^{-\frac t2}.
	\end{equation}
	On the other hand, we have $$|e^{-\lambda_{-}t} \langle \mathscr{F}_b\bold{u}_0,\bold{a}_- \rangle \langle \bold{b}_-,e_2 \rangle| = C|e^{-\lambda_{-}t} \mathscr{F}_b \tht| \geq Ce^{-\frac {2|\tilde{n}|^2}{|\eta|^{2(1+\alpha)}} t} |q|^{-(m+\frac 12 + \epsilon)}.$$ Thus, 
	\begin{align*}
		\|e^{-\lambda_{-}t} \langle \mathscr{F}_b\bold{u}_0,\bold{a}_- \rangle \langle \bold{b}_-,e_2 \rangle\|_{H^s} &\geq C\left( \sum_{J \in A} e^{-\frac {4|\tilde{n}|^2}{|\eta|^{2(1+\alpha)}}t} |q|^{-(2(m-s)+ 1 + 2\epsilon)} \right)^{\frac 12} \\
		&\geq C\left( \sum_{|q| \geq C_1} e^{-\frac {Ct}{q^{2(1+\alpha)}}} |q|^{-(2(m-s)+ 1 + 2\epsilon)} \right)^{\frac 12} \\
		&\geq Ct^{-\frac {m-s}{2(1+\alpha)}-\frac {1+2\epsilon}{4(1+\alpha)}} \left( \sum_{|q| \geq C_1} e^{-\frac {Ct}{q^{2(1+\alpha)}}} (\frac {Ct}{|q|^{2(1+\alpha)}})^{\frac {m-s}{1+\alpha}+\frac {1+2\epsilon}{2(1+\alpha)}} \right)^{\frac 12} ,
	\end{align*}
	for some $C_1>0$. Let $$f(\tau) := e^{-\frac {Ct}{|\tau|^{2(1+\alpha)}}} (\frac {Ct}{|\tau|^{2(1+\alpha)}})^{\frac {m-s}{1+\alpha}+\frac {1+2\epsilon}{2(1+\alpha)}}.$$ Then, we can verify that there exists $C_2 \geq C_1$ not depending on $t$ such that $f(\tau)$ is decreasing on the interval $(C_2 t^{\frac 1{2(1+\alpha)}},\infty)$. Thus, it holds for $t \geq 1$ that
	\begin{align*}
		\sum_{|q| \geq C_1} e^{-\frac {Ct}{|q|^{2(1+\alpha)}}} (\frac {Ct}{|q|^{2(1+\alpha)}})^{\frac {m-s}{1+\alpha}+\frac {1+2\epsilon}{2(1+\alpha)}} \geq \int_{|\tau| \geq C_2 t^{\frac 1{2(1+\alpha)}}} f(\tau) \,\ud \tau.
	\end{align*}
	By the change of variable $\tilde{\tau} = \tau t^{-\frac 1{2(1+\alpha)}}$, we can see
	\begin{align*}
		\int_{|\tau| \geq C_2 t^{\frac 1{2(1+\alpha)}}} f(\tau) \,\ud \tau = t^{\frac 1{2(1+\alpha)}} \int_{|\tilde{\tau}| \geq C_2} e^{-\frac 1{|\tilde{\tau}|^{2(1+\alpha)}}} |\tilde{\tau}|^{-(2(m-s)+1+2\epsilon)} \,\ud \tilde{\tau} \geq C
	\end{align*} for some $C>0$. 
	Combining the above yields $$\|e^{-\lambda_{-}t} \langle \mathscr{F}_b\bold{u}_0,\bold{a}_- \rangle \langle \bold{b}_-,e_2 \rangle\|_{H^s} \geq Ct^{-\frac {m-s+\epsilon}{2(1+\alpha)}}.$$ Therefore, \eqref{tht_lin_decay} is obtained.
	
	The proof of \eqref{vd_lin_decay} is similar with the previous one. By \eqref{df_lu} and \eqref{+_est}, it holds $$\| v_d \|_{H^s} \geq \| e^{-\lambda_{-}t} \langle \mathscr{F}_b\bold{u}_0,\bold{a}_- \rangle \langle \bold{b}_-,e_1 \rangle \|_{H^s} - Ce^{-\frac t2}.$$ Note that $$|e^{-\lambda_{-}t} \langle \mathscr{F}_b\bold{u}_0,\bold{a}_- \rangle \langle \bold{b}_-,e_1 \rangle| = \frac {|\tilde{n}|^2}{|\eta|^{2(1+\alpha)}} |e^{-\lambda_{-}t} \mathscr{F}_b \tht| \geq  \frac {|\tilde{n}|^2}{|\eta|^{2(1+\alpha)}}e^{-\frac {|\tilde{n}|^2}{|\eta|^{2(1+\alpha)}} \frac t2} |q|^{-(m+\frac 12 + \epsilon)}.$$ Using that $\frac {|\tilde{n}|^2}{|\eta|^{2(1+\alpha)}} \geq \frac {C}{|q|^{2(1+\alpha)}}$ for all $\eta \in A$, repeating the above procedures, we obtain \eqref{vd_lin_decay}. This completes the proof.
\end{proof}

\subsection*{Competing Interests}
 The authors have no competing interests to declare that are relevant to the content of this article.
\subsection*{Data Availability}
Data sharing not applicable to this article as no datasets were generated or analysed during the current study.

\section*{Acknowledgement}
 J. Jang's 
research is supported in part by the NSF DMS-grant 2009458. J. Kim is supported by a KIAS Individual Grant (MG086501) at Korea Institute for Advanced Study.  
\noindent 
\bibliographystyle{amsplain}


\end{document}